\documentclass[11pt]{article}
\usepackage{}
\usepackage{amsmath}
\usepackage{amsfonts}
\usepackage{graphicx}
\usepackage{amssymb}
 \usepackage{leftidx}

\newtheorem{condition**}{A*}
\newtheorem{condition***}{C*}
\newtheorem{condition*}{C}

\newtheorem{definition}{Definition}[section]
\newtheorem{theorem}{Theorem}[section]
\newtheorem{lemma}{Lemma}[section]
\newtheorem{remark}{Remark}[section]

\newenvironment{keywords}{{\bf Key words: }}{}

\renewcommand{\thefootnote}{\fnsymbol{footnote}}

\begin{document}

\title{Linear-Quadratic-Gaussian Mixed Mean-field Games with Heterogeneous Input Constraints}

\author{Ying Hu$\;^{a}$, Jianhui Huang$\;^{b}$, Tianyang Nie$\;^{c,\ast}$\bigskip\\{\small ~$^{a}$IRMAR, Universit\'e Rennes 1, Campus de Beaulieu, 35042 Rennes Cedex, France}
\\{\small and School of
Mathematical Sciences, Fudan University, Shanghai 200433, China}\\{\small $^{b}$Department of Applied Mathematics, The Hong Kong Polytechnic University, Hong Kong, China}\\{\small ~$^{c}$School of Mathematics, Shandong University, Jinan, Shandong
250100, China}}
\maketitle

\begin{abstract}
We consider a class of linear-quadratic-Gaussian mean-field games with a major agent and considerable heterogeneous minor agents in the presence of mean-field interactions. The individual admissible controls are constrained in closed convex subsets $\Gamma_{k}$ of $\mathbb{R}^{m}.$ The decentralized strategies for individual agents and consistency condition system are represented in an unified manner through a class of mean-field forward-backward stochastic differential equations involving projection operators on $\Gamma_{k}$. The well-posedness of consistency system is established in both the local and global cases by the contraction mapping and discounting method respectively. Related $\varepsilon-$Nash equilibrium property is also verified.\end{abstract}

\\
\footnotetext[1]{{\scriptsize Corresponding author: Tianyang NIE (nietianyang@sdu.edu.cn).}}
\renewcommand{\thefootnote}{\arabic{footnote}}
\footnotetext[1]{{\scriptsize
The work of Ying Hu is partially supported  by Lebesgue center of mathematics ``Investissements d'avenir"
program - ANR-11-LABX-0020-01, by ANR-15-CE05-0024 and by ANR-16-CE40-0015-01; The work of James Jianhui Huang is supported by RGC Grant 500613P,
15300514P, 15327516P; The work of Tianyang Nie is supported by the National Natural Sciences Foundations of China (No. 11601285, 11571205, 61573217), the Natural
Science Foundation of Shandong Province (No. ZR2016AQ13) and the Fundamental Research Funds of Shandong University (No.2015HW023).}}
\renewcommand{\thefootnote}{\fnsymbol{footnote}}
\footnotetext{\textit{{\scriptsize E-mail addresses:}}
{\scriptsize Ying.Hu@univ-rennes1.fr (Ying\ Hu); james.huang@polyu.edu.hk (James Jianhui\ Huang); nietianyang@sdu.edu.cn (Tianyang\ NIE).}}

\medskip

\section{Introduction}

Mean-field games (MFGs) for stochastic large-population systems have been well-studied because of their wide applications in various fields, such as economics, engineering, social sciences, and operational research. Large-population systems are distinguished by a large number of agents (or players), such that the individual influence of any single agent on the overall population is negligible, but the effects of their statistical behaviors cannot be ignored at the population scale. Mathematically, all agents are weakly-coupled in their dynamics or cost functionals through the state-average (in a linear state case) or the general empirical measure (in a nonlinear state case), which characterizes the statistical effect generated by the population from a macroscopic perspective. Because of these features, when the number of agents is sufficiently high, complicated coupling features arise and it is unrealistic for a given agent to obtain all other agents' information. Consequently, for an agent to design centralized strategies based on the information of all peers in a large-population system is intractable. Alternatively, one reasonable and practical direction is to transform a high-dimensional and weakly-coupled problem to a low-dimensional and decoupled one, thus the complexity in both analysis and computation can be reduced. For this purpose, one method is to investigate relevant decentralized strategies based only on local information: that is, the relevant strategies are designed only upon the individual state of the given agent, together with some mass-effect quantities, which are computed off-line.\\

In this context, motivated by the theory of the propagation of chaos, Lasry and Lions  \cite{LL-2006,LL-20062,ll} proposed distributed closed-loop strategies for each agent to solve the limiting problem, which were formulated as a coupled nonlinear forward-backward system consisting of a Hamilton-Jacobi-Bellman (HJB) equation and a Fokker-Planck equation. Moreover, the limiting problem enabled the design of approximate Nash equilibrium strategies. Independently,  Caines, Huang and Malham\'{e} \cite{hmc06} developed a similar program called the Nash certainty equivalence (NCE) principle, which was motivated by the analysis of large communications networks. In principle, the MFGs procedure consists of the following four main steps (see \cite{BFY,CD,hcm06,H-C-M 2007,ll}): (1) A limiting mass-effect term is introduced, which comes from the asymptotic mass-effect behavior when the agent number $N$ tends to infinity. This limiting term should be treated as an exogenous and undetermined ``frozen" term at this moment; (2) Through replacing the mass-effect term with the frozen limiting term, related limiting optimization problem can be introduced. Thus, the initial highly-coupled problem can be decoupled and only parameterized by this generic frozen limit. Subsequently, using standard control techniques (see \cite{yz}), an HJB equation can be obtained because of the dynamic programming principle (DPP), or a Hamiltonian system can be obtained because of the stochastic maximum principle (SMP); the obtained equations can characterize the decentralized optimal strategies; (3) A consistency condition is established to guarantee that the set of decentralized optimal strategies collectively replicates the mass-effect; (4) the derived decentralized strategies are shown to be $\varepsilon$-Nash equilibrium, which justifies the aforementioned scheme for finding the approximate Nash equilibrium.\\

For further analysis details of MFGs, readers are referred to \cite{A2015,BFY,CD,G-L-L 2011,hcm06,T-Z-B 2014,W-Z 2012}. That substantial literatures exists studying linear-quadratic-Gaussian (LQG)-MFG is remarkable. For example, \cite{HCM2007} studied LQG MFGs using the common Riccati equation approach; \cite{BSYY} adopted SMP with the help of adjoint equations; \cite{B2012} and \cite{L-Z 2008} studied
ergodic LQG MFGs; \cite{H-C-M 2007} studied LQG MFGs with nonuniform agents through state-aggregation by empirical distribution. For further research, readers are referred to \cite{BP2014, BSY2013,P2014} and the references therein.\\

All aforementioned mentioned works examine standard MFGs, which (except \cite{hmc06}) require that all the agents be statistically identical and that the individual influence on the overall population of a single agent be negligible as the number of the agents tends to infinity. However, in the real world, some models exist in which a major agent has a significant influence on other agents (called minor agents) no matter how numerous the minor-agents may be. Such interactions exists in numerous socio-economic problems (e.g., \cite{KMC2012,Y2002}). This type of games involving agents with different power levels is usually called a mixed type games. Compared with MFGs with only minor agents, a distinctive feature of mixed type MFGs is that the mean field behavior of the minor agents is affected by the major agent; thus, it is a random process, and the influence of the major agent on the minor agents is not negligible in the limiting problem. To deal with such new features, conditional distribution with respect to the major agent's information flow is introduced (see \cite{NC-2013,CZ2016}).\\

We now discuss some works on MFGs with a major agent and minor agents that are related to our paper. To the best of our knowledge, the first LQG MFG with a major agent and minor agents was studied by \cite{H 2010}, in which the minor agents were from a total $K$ classes. In the study of \cite{N-H 2012}, the authors examined mean-field LQG mixed games with continuum-parameterized minor agents.  \cite{NC-2013} investigated nonlinear stochastic dynamic systems with major and minor agents. \cite{BLP2014} studied nonlinear stochastic differential games, involving a major agent and a large number of collectively acting minor agents, as two-person zero-sum stochastic differential games of the type feedback control against feedback control, the limiting behaviors of the saddle point controls were also studied. For further research, readers are referred to \cite{BCY2016,CZ2016} and the references therein. \\

In this study, we investigated a class of LQG MFGs with major agents and minor agents in the presence of control constraints.  In all of the aforementioned papers about linear quadratic (LQ) control problems, the control was unconstrained, and the (feedback) control was constructed from DPP or SMP, which is automatically admissible. Whereas, if we impose constraints on the admissible control, the whole LQ approach fails to apply, (see e.g., \cite{CZ2004,HZ2005}). We emphasizes that the LQ control problems with control constraints have wide applications in finance and economics. For example, the mean-variance  problem with prohibiting short-selling can be transferred to LQ control problems with positive control, (see e.g., \cite{BWYY2014,HZ2005}). The optimal investment problems where the agents have relative performance, (i.e., their portfolio constraints take different forms), can also be tackled through approach of LQ control problems with input constraint, (see e.g, \cite{HIM2005,ET2015}). Remark \ref{remark for constraint} of the current paper provides several other constraint sets $\Gamma\subset\mathbb{R}^m$ as well as their applications. For an investigation of LQ problems with positive controls or a more general study where the control is constrained in a given convex cone, readers are referred to \cite{B1972} for a deterministic case and \cite{CZ2004,HZ2005,LZL2002} for a stochastic case.\\

To the best of our knowledge, the current study is the first to examine constrained LQG MFGs with major agents and large numbers of minor agents. In addition to the control constraint being fully new, our study also has other novel features when compared with other relevant studies: in \cite{H 2010,N-H 2012}, the diffusion term simply takes a constant; however, this study considers the mean-field LQG mixed games in which the diffusion term depends on the major agent's and the minor agent's states, as well as the individual control strategy. This creates additional difficulties, especially when applying the general SMP; in \cite{NC-2013,BLP2014}, a nonlinear stochastic differential games was studied; however, in this study, we put ourselves in a LQ mean-field framework with individual controls constrained in a closed convex set; thus, we can explicitly present the optimal strategies through a projection operator. Moreover, we use SMP to obtain the optimal strategies through Hamiltonian systems that are fully coupled forward-backward stochastic differential equations (FBSDEs); this is different from \cite{NC-2013}, which used DPP and a verification theorem to characterize the optimal strategies. Here, we connect the consistency condition to  a new type of conditional mean field forward-backward stochastic differential equations (MF-FBSDEs) involving projection operators. We establish its well-posedness under suitable conditions using a fixed point theorem, both in local cases and global cases. Unlike our previous paper \cite{HHL2016}, we now focus on the mixed game, which is more realistic and more difficult. In this situation, the consistency condition is a conditional MF-FBSDE that does not satisfy the usual monotonicity condition of \cite{HP1995}. Moreover, we require an additional subtle analysis to analyze the major agent's influence and establish the approximate Nash equilibrium. Finally, motivated by \cite{ET2015}, we believe that our results can be applied to solve the optimal investment problems caused by major agent and $N$ minor agents. \\

The reminder of this paper is structured as follows: In Section 2, we formulate the LQG MFGs with a control constraint involving a major agent and minor agents. Decentralized strategies are derived through a FBSDE with projection operators. A consistency condition is also established using some fully coupled FBSDEs that come from the SMP. In Section 3, we prove the well-posedness of fully coupled conditional MF-FBSDEs, which characterize the consistency condition in the local time horizon case. In Section 4, we ascertain the wellposedness of the global time case. In Section 5, we verify the $\varepsilon-$Nash equilibrium of the decentralized strategies. \\

The main contributions of this paper can be summarized as follows:

\begin{itemize}
\item To introduce and analyze a new class of LQ mixed MFGs using SMP. In our setting, both the major agent and minor agents are constrained in their control inputs. \vspace{-0.2cm}
\item The diffusion terms of major and minor agents are dependent on their states and control variables.\vspace{-0.2cm}
\item The consistency condition system or NCE is represented through a new type of conditional mean-filed type FBSDE with projection operators.\vspace{-0.2cm}
\item The existence and uniqueness of such an NCE system are established in a local case (i.e., small time horizon) using the contraction mapping method, and in a global case (i.e., arbitrary time horizon) using the discounting method.
\end{itemize}
\bigskip

\section{Notations and terminology}
Consider a finite time horizon $[0,T]$ for fixed $T>0$. We assume $(\Omega,
\mathcal F, \{\mathcal{F}_t\}_{0 \leq t \leq T}, \mathbb{P})$ is a complete, filtered probability space satisfying usual conditions and $\{W_i(t),\ 0\le i\leq N\}_{0 \leq t \leq T}$ is a $(N+1)$-dimensional Brownian motion on this space. Let $\mathcal{F}_t$ be the natural filtration generated by  $\{W_i(s), 0\leq i\leq N, 0\leq s\leq t\}$ and augmented by $\mathcal{N}_{\mathbb{P}}$ (the class of $\mathbb{P}$-null sets of $\mathcal{F}$). Let $\mathcal F^{W_0}_t$, $\mathcal F^{W_i}_t$, $\mathcal F^{i}_t$ be respectively the augmentation of $\sigma\{W_0(s), 0\leq s\leq t\}$,  $\sigma\{W_i(s), 0\leq s\leq t\}$, $\sigma\{W_0(s),W_i(s), 0\leq s\leq t\}$ by $\mathcal{N}_{\mathbb{P}}$. Here, $\{\mathcal F^{W_0}_t\}_{0\leq t\leq T}$ stands for the information of the major agent, while $\{\mathcal F^{W_i}_t\}_{0\leq t\leq T}$ represents the individual information of $i$-th minor agent.

Throughout the paper, $x^{\prime}$ denotes the transpose of a vector or a matrix $x$, $\mathcal{S}^n$ denotes the set of symmetric $n\times n$ matrices with real elements. For a matrix $M\in\mathbb{R}^{n\times d}$, we define the norm $|M|:=\sqrt {Tr(M^{\prime}M)}$. If $M\in \mathcal{S}^n$ is positive (semi) definite, we write $M>(\geq) 0$. Let $\mathcal{H}$ be a given Hilbert space, the set of $\mathcal{H}$-valued continuous functions is denoted by $C(0,T;\mathcal{H})$. If $N(\cdot)\in C(0,T;\mathcal{S}^n)$ and $N(t)>(\geq) 0$ for every $t\in[0,T]$, we say that $N(\cdot)$ is positive (semi) definite, which is denoted by $N(\cdot)>(\geq) 0$.  Now, for a given Hilbert space $\mathcal{H}$ and a filtration $\{\mathcal{G}_t\}_{0\leq t\leq T}$, we also introduce the following spaces which will be used in this paper:
\begin{flalign*}
&L^2_{\mathcal G_T}(\Omega;\mathcal{H}):=\left\{\xi:\Omega\rightarrow\mathcal{H}\ |\  \xi \text{ is } \mathcal{G}_T\text{-measurable such that } \mathbb{E}|\xi|^2<\infty\right\},&\medskip\\
&L^2_{\mathcal G}(0,T;\mathcal{H}):=\Big\{x(\cdot): [0,T]\times\Omega\rightarrow\mathcal{H}\ |\ x(\cdot)\text{ is } \mathcal{G}_t\text{-progressively measurable process such that } &\medskip\\
&\qquad\qquad\qquad\qquad\qquad\qquad\qquad\qquad\qquad\qquad\qquad\qquad\mathbb{E}\int_0^T|x(t)|^2dt<\infty\Big\},&\medskip\\
&L^2_{\mathcal G}(\Omega;C(0,T;\mathcal{H})):=\Big\{x(\cdot): [0,T]\times\Omega\rightarrow\mathcal{H}\ |\ x(\cdot)\text{ is } \mathcal{G}\text{-adapted continuous process such that } &\medskip\\
&\qquad\qquad\qquad\qquad\qquad\qquad\qquad\qquad\qquad\qquad\qquad\qquad\mathbb{E}\left[\sup\limits_{0\leq t\leq T}|x(t)|^2\right]<\infty\Big\}.&
\end{flalign*}
\section{LQG mixed games with control constraint}

We consider a linear-quadratic-Gaussian mixed mean-field game involving a major agent $\mathcal{A}_0$ and a heterogeneous large-population with $N$ individual minor agents $\{\mathcal{A}_{i}: 1 \leq i \leq N\}$. Unlike other works of LQG mixed games, our control domain is constrained in a closed convex set of the full Euclidean space, and more details of constraints will be given later. The states  $x_0$ for major agent $\mathcal{A}_{0}$ and $x_i$ for each minor agent $\mathcal{A}_{i}$ are modeled by the following controlled linear stochastic differential equations (SDEs) with mean-field coupling:
\begin{equation}\label{major player state}\begin{aligned}
dx_0(t)&=[A_{0}(t)x_0(t)+B_{0}(t)u_0(t)+F_0^{1}(t) x^{(N)}(t)+b_{0}(t)]dt\\&+[C_{0}(t)x_0(t)+D_{0}(t)u_0(t)
+F_0^{2}(t)x^{(N)}(t)+\sigma_{0}(t)]dW_0(t), \quad x_0(0)=x_0\in \mathbb{R}^{n},
\end{aligned}\end{equation}
and
\begin{equation}\label{minor player state}\begin{aligned}
dx_i(t)&=[A_{\theta_i}(t)x_i(t)+B(t)u_i(t)+F_1(t) x^{(N)}(t)+b(t)]dt\\&+[C(t)x_i(t)+D_{\theta_i}(t)u_i(t)
+F_2(t)x^{(N)}(t)+Hx_0(t)+\sigma(t)]dW_i(t), \quad x_i(0)=x \in \mathbb{R}^{n},
\end{aligned}\end{equation}
where $x^{(N)}(\cdot)=\frac{1}{N}\sum_{i=1}^{N}x_{i}(\cdot)$ is the state-average of minor agents. Recalling that $\mathcal F^{i}_t$ is the individual decentralized information while $\mathcal{F}_t$ is the centralized information driven by all Brownian motion components. We point out that the heterogeneous noise $W_i$ is specific for individual agent $\mathcal{A}_{i}$ whereas $x_i(t)$ is adapted to $\mathcal{F}_t$ instead of $\mathcal{F}^{i}_t$ due to the coupling state-average $x^{(N)}.$ The coefficients of (\ref{major player state}) and (\ref{minor player state}) are deterministic matrix-valued functions with appropriate dimensions. The number $\theta_i$ is a parameter of agent $\mathcal{A}_i$ to model a heterogeneous population of minor agents, for more explanations, see \cite{H 2010}. For sake of notations, in (\ref{minor player state}), we only set the coefficients $A(\cdot)$ and $D(\cdot)$ (see also $R(\cdot)$ in (\ref{minor agent cost functional})) to be dependent on  $\theta_i$. Similar analysis can be proceeded in case that all other coefficients depend also on $\theta_i$. In this paper, we assume that $\theta_i$ takes values from a finite set $\Theta:=\{1,2,\ldots,K\}$ which means that totally $K$ types of minor agents are considered. We call $\mathcal{A}_i$ a $k$-type minor agent if $\theta_i=k \in \Theta$.

In this paper, we are interested in the asymptotic behavior as $N$ tends to infinity which is essentially to consider a family of games with an increasing number of minor agents. To describe the related large population system, let us first define
\[
\mathcal{I}_k=\{i\ |\ \theta_i=k, 1 \leq i \leq  N\}, \quad \quad N_k=|\mathcal{I}_k|,
\]
where $N_k$ is the cardinality of index set $\mathcal{I}_k, 1 \leq k \leq K.$ Let $\pi_k^{(N)}=\frac{N_k}{N}$ for $k \in \{1, \cdots, K\}$ then $\pi^{(N)}=(\pi_1^{(N)}, \cdots, \pi_K^{(N)})$ is a probability vector to represent the empirical distribution of $\theta_1, \cdots, \theta_N.$ The following assumption gives some statistical properties for $\theta_i.$ For more details, the reader is referred to \cite{H 2010}:
\begin{description}
  \item[(A1)] There exists a probability mass vector $\pi=(\pi_1,\pi_2, \cdots, \pi_K)$ such that $\lim_{N\rightarrow+\infty}\pi^{(N)}=\pi$ and $\min_{1 \leq k \leq K}\pi_{k}>0.$
  \end{description}
From $\mathbf{(A1)}$ we know that when $N\rightarrow+\infty,$ the proportion of $k$-type agents becomes stable for each $k$ and that the number of each type agents tends to infinity. Otherwise, the agents in given type with bounded size should be excluded from consideration when analyzing asymptotic behavior as $N \rightarrow+\infty.$ Throughout the paper we make the convention that $N$ is suitable large such that $\min_{1 \leq k \leq K}N_{k}\ge1$.

Now let us specify the admissible control set and cost functionals of our linear-quadratic-Gaussian mixed game with input control constraint. We call $u_0$ an admissible control for the major agent if $u_0\in \mathcal{U}^{0}_{ad}$, where $\mathcal{U}^{0}_{ad}:=\{u(\cdot)\ | \ u(\cdot)\in L^{2}_{\mathcal{F}}(0, T; \Gamma_{0})\}.$
Here $\Gamma_{0} \subset \mathbb{R}^{m}$ is a nonempty closed convex set. Moreover, for each $1\leq i\leq N$, we also define decentralized admissible control $u_i$ for the minor agent $\mathcal{A}_i$ as $u_i\in \mathcal{U}_{ad}^{i}$, where for a nonempty closed convex set $\Gamma_{\theta_i}\subset\mathbb{R}^m$,
$
\mathcal{U}_{ad}^{i}:=\{u_i(\cdot)\ |\ u_i(\cdot)\in L^{2}_{\mathcal{F}^{i}}(0, T; \Gamma_{\theta_i}) \}.
$
\begin{remark}\label{remark for constraint}
We give the following typical examples for the closed convex constraint set $\Gamma$: $\Gamma^{1}=\mathbb{R}^{m}_{+}$ represents that the control can only take positive values, it connects with the mean-variance portfolio selection problem with no-shorting, see \cite{BWYY2014,HZ2005}. The linear subspace $\Gamma^{2}=(\mathbb{R}e_i)^{\bot}$ (where $(e_1,e_2,\ldots,e_m)$ is the canonical basis of $\mathbb{R}^m$) represents that the control can only take from a hyperplane, it is used to deal with that in the investment theory, each manager has access to the whole market except some fixed firm who has private information. For more examples of linear constraints and their economic meaning, the reader is referred to \cite{ET2015}. $\Gamma$ can also be some closed cone (i.e. $\Gamma$ is closed and if $u\in\Gamma$, then $\alpha u\in\Gamma$, for all $\alpha\ge0$), e.g. $\Gamma^{3}=\{u\in\mathbb{R}^m: \Upsilon u=0\}$ or $\Gamma^{4}=\{u\in\mathbb{R}^m: \Upsilon u\leq 0\}$, where $\Upsilon \in \mathbb{R}^{n\times m}$.  
For investigations on the stochastic LQ problems with conic control constraint, the readers are refereed to \cite{CZ2004,HZ2005}.
\end{remark}

Let $u=(u_0,u_1, \cdots, u_{N})$  be the set of strategies of all $N+1$ agents, $u_{-0}=(u_1,u_2,\ldots,u_N)$ be the control strategies except $\mathcal{A}_0$ and
$u_{-i}=(u_{0},u_1, \cdots, u_{i-1}, u_{i+1}, \cdots, u_{N})$
be the set of strategies except the $i$-th agent $\mathcal{A}_i.$ We introduce the cost functional of the major agent as
\begin{equation}\label{major agent cost functional}
\begin{aligned}
    \mathcal {J}_0(u_0,u_{-0})&=\frac{1}{2}\mathbb{E}\bigg [\int_0^T \Big <Q_{0}(t)\big(x_0(t)-\rho_0 x^{(N)}(t)\big),x_0(t)-\rho_0 x^{(N)}(t)\Big>\\
    &+\big<R_{0}(t)u_0(t), u_0(t)\big>dt+\Big<G_{0}\big(x_0(T)-\rho_0 x^{(N)}(T)\big),x_0(T)-\rho_0 x^{(N)}(T)\Big>\bigg]
\end{aligned}
\end{equation}
and the cost functional of the minor agent $\mathcal{A}_i$ as
\begin{equation}\label{minor agent cost functional}
\begin{aligned}
    \mathcal {J}_i(u_i,u_{-i})=&\frac{1}{2}\mathbb{E}\bigg [\int_0^T \Bigg(\Big <Q\big(x_i\!-\!\rho x^{(N)}\!-\!(1\!-\!\rho)x_0(t)\big),x_i\!-\!\rho x^{(N)}\!-\!(1\!-\!\rho)x_0(t)\Big>\!+\!\big<R_{\theta_i}u_i,u_i\big>\Bigg)dt\\
    &+\Big<G\big(x_i(T)\!-\!\rho x^{(N)}(T)\!-\!(1\!-\!\rho)x_0(T)\big),x_i(T)\!-\!\rho x^{(N)}(T)\!-\!(1\!-\!\rho)x_0(T)\Big>\bigg].
\end{aligned}
\end{equation}
We mention that for notational brevity, the time argument is suppressed in above equation as
well as in the sequel when necessary.

We impose the following assumptions:
\begin{description}
  \item[(A2)] The coefficients of the states satisfy that, for $1\leq i\leq N$,
  \[
  \begin{aligned}
  &A_0(\cdot), A_{\theta_i}(\cdot), C_0(\cdot), C(\cdot), F_0^1(\cdot), F_1(\cdot), F_0^2(\cdot), F_2(\cdot), H(\cdot)\in L^\infty(0,T;\mathbb{R}^{n \times n}),\medskip\\
  &B_0(\cdot),B(\cdot),D_0(\cdot), D_{\theta_i}(\cdot) \in L^\infty(0,T;\mathbb{R}^{n \times m}), b_0(\cdot), b(\cdot), \sigma_0(\cdot), \sigma(\cdot) \in L^\infty(0,T;\mathbb{R}^n).
  \end{aligned}
  \]
  \item[(A3)] The coefficients of cost functionals satisfy that, for $1\leq i\leq N$,
   \[
     \begin{aligned}
   &Q_0(\cdot), Q(\cdot)\in L^\infty(0,T;\mathcal{S}^n), R_0(\cdot), R_{\theta_i}(\cdot)\in L^\infty(0,T;\mathcal{S}^m),G_0,G\in \mathcal{S}^n,\medskip\\
   & Q_0(\cdot)\geq0, Q(\cdot)\geq0, R_0(\cdot)>0, R_{\theta_i}(\cdot)>0,G_0\geq0, G\geq0, \rho_0,\rho\in [0,1].
     \end{aligned}
   \]
\end{description}
Here $L^\infty(0,T;\mathcal{H})$ denotes the space of uniformly bounded functions mapping from $[0,T]$ to $\mathcal{H}$.
It follows that, under assumptions (\textbf{A2}) and  (\textbf{A3}), the system (\ref{major player state}) and  (\ref{minor player state}) admits a unique solution $x_0, x_i(\cdot) \in L^{2}_{\mathcal{F}}(\Omega;C(0,T;\mathbb{R}^n))$ for given admissible control $u_0$ and $u_i$. Now, let us formulate the LQG mixed games with control constraint.\\

\textbf{Problem (CC).}
Find a strategies set $\bar{u}=(\bar{u}_0,\bar{u}_1,\cdots,\bar{u}_N)$ where $\bar u_i(\cdot)\in \mathcal{U}^{i}_{ad},\ 0 \leq i\leq N$,  such that
\[
\mathcal{J}_i(\bar{u}_i(\cdot),\bar{u}_{-i}(\cdot))=\inf_{u_i(\cdot)\in \mathcal{U}_{ad}^{i}}\mathcal{J}_i(u_i(\cdot),\bar{u}_{-i}(\cdot)), \quad 0\leq i\leq N.
\]
We call $\bar{u}$ an optimal strategies set for Problem \textbf{(CC).}
\medskip

For comparison, we present also the definition of $\varepsilon$-Nash equilibrium.
\begin{definition}\label{epsilon-Nash equilibrium def}
A set of strategies $\bar u_i(\cdot)\in \mathcal{U}^{i}_{ad},\ 0 \leq i\leq N,$ is called  an $\varepsilon$-Nash equilibrium with respect to costs  $\mathcal J^i,\ 0\leq i\leq N,$ if there exists a $\varepsilon=\varepsilon(N)\geq0, \displaystyle\lim_{N \rightarrow +\infty}\varepsilon(N)=0,$
such that for any $0 \leq i\leq N$, we have
\[\label{epsilon-Nash equilibrium}
\mathcal J_i(\bar u_i(\cdot),\bar u_{-i}(\cdot))\leq \mathcal J_i(u_i(\cdot),\bar{u}_{-i}(\cdot))+\varepsilon,
\]
when any alternative strategy $u_i\in \mathcal{U}^{i}_{ad}$ is applied by $\mathcal{A}_i$.
\end{definition}
\begin{remark}
If $\varepsilon=0,$ Definition \ref{epsilon-Nash equilibrium def} reduces to the usual exact Nash equilibrium.
\end{remark}

\subsection{Stochastic optimal control problem of the major agent}

As explained in introduction, the centralized optimization strategies to Problem (\textbf{CC}), are rather complicate and infeasible to be applied when the number of the agents tends to infinity. Alternatively, we investigate the decentralized strategies via the limiting problem with the help of frozen limiting state-average. To this end, we first figure out the representation of limiting process using heuristic arguments. Based on it, we can find the decentralized strategies by consistency condition, and then rigorously verify the derived strategies satisfy the approximate Nash equilibrium. We formalize the auxiliary limiting mixed game via the approximation of the average state $x^{(N)}$. Since $\pi_k^{(N)}\approx\pi_k$ for large $N$ and
\[
x^{(N)}=\frac{1}{N}\sum_{k=1}^{K}\sum_{i\in\mathcal{I}_k}x_i=\sum_{k=1}^{K}\pi_k^{(N)}\frac{1}{N_k}\sum_{i\in\mathcal{I}_k}x_i,
\]
we may approximate $x^{(N)}$ by $\sum_{k=1}^{K}\pi_km_k$, where $m_k\in\mathbb{R}^n$ is used to approximate $\frac{1}{N_k}\sum_{i\in\mathcal{I}_k}x_i$. Denote $m=(m_1^\prime,m_2^\prime,\ldots m_k^\prime)^{\prime}$, which is called the set of aggregate quantities. Replacing $x^{(N)}$ of (\ref{major player state}) and (\ref{major agent cost functional}) by $\sum_{k=1}^{K}\pi_km_k$, the major agent's dynamics is given by
\begin{equation}\label{major agent state limit}
\begin{aligned}
dz_0(t)&=[A_0(t)z_0(t)+B_0(t)u_0(t)+F_0^1(t) \sum_{k=1}^{K}\pi_km_k(t)+b_0(t)]dt\medskip\\
&+[C_0(t)z_0(t)+D_0(t)u_0(t)
+F_0^2(t)\sum_{k=1}^{K}\pi_km_k(t)+\sigma_0(t)]dW_0(t), \quad z_0(0)=x_0 \in \mathbb{R}^{n},
\end{aligned}
\end{equation}
and the limiting cost functional is
\begin{equation}\label{major agent cost functional limit}
\begin{aligned}
    &{J}_0(u_0)=\frac{1}{2}\mathbb{E}\bigg [ \int_0^T \Big <Q_0(t)\big(z_0(t)-\rho_0\sum_{k=1}^{K}\pi_km_k(t)\big),z_0(t)-\rho_0\sum_{k=1}^{K}\pi_km_k(t)\Big>\medskip\\
    &\quad+\big<R_0(t)u_0(t),u_0(t)\big>dt
    +\Big<G_0\big(z_0(T)-\rho_0\sum_{k=1}^{K}\pi_km_k(T)\big),z_0(T)-\rho_0\sum_{k=1}^{K}\pi_km_k(T)\Big>\bigg].
\end{aligned}
\end{equation}
For simplicity, let $\otimes$ be the Kronecker product of two matrix (see \cite{Graham 1981}) and we denote $F_0^{1,\pi}:=\pi\otimes F_0^1$, $F_0^{2,\pi}:=\pi\otimes F_0^2$, $\rho_0^{\pi}:=\pi\otimes \rho_0I_{n\times n}$. Then (\ref{major agent state limit}) and (\ref{major agent cost functional limit}) become respectively to
\begin{equation}\label{major agent state limit 1}
\begin{aligned}
dz_0(t)&=[A_0(t)z_0(t)+B_0(t)u_0(t)+F_0^{1,\pi}(t)m(t)+b_0(t)]dt\medskip\\
&+[C_0(t)z_0(t)+D_0(t)u_0(t)+F_0^{2,\pi}(t)m(t)+\sigma_0(t)]dW_0(t), \quad z_0(0)=x_0 \in \mathbb{R}^{n},
\end{aligned}
\end{equation}
and
\[\label{major agent cost functional limit 1}
\begin{aligned}
    {J}_0(u_0)&=\frac{1}{2}\mathbb{E}\bigg [ \int_0^T \Big <Q_0(t)\big(z_0(t)-\rho_0^\pi m(t)\big),z_0(t)-\rho_0^\pi m(t)\Big>\medskip\\
    &+\big<R_0(t)u_0(t),u_0(t)\big>dt+\Big<G_0\big(z_0(T)-\rho_0^\pi m(T)\big),z_0(T)-\rho_0^\pi m(T)\Big>\bigg].
\end{aligned}
\]

We define the following auxiliary stochastic optimal control problem for major agent with infinite population:
\medskip

\textbf{Problem (LCC-Major).} For major agent $\mathcal{A}_0,$ find $u^\ast_0(\cdot)\in \mathcal{U}_{ad}^{0}$ satisfying
\[J_0(u^\ast_0(\cdot))=\inf_{u_0(\cdot)\in \mathcal{U}_{ad}^{0}}J_0(u_0(\cdot)).
\]
Then $u^\ast_0(\cdot)$ is called a decentralized optimal control for Problem (\textbf{LCC-Major}).
\medskip

Now, similar to \cite{HHL2016}, we would like to apply SMP to above limiting LQG problem (\textbf{LCC-Major}) with input constraint. We introduce the following first order adjoint equation:
\begin{equation}\label{first order adjoint equation}
    \left\{
    \begin{aligned}
      dp_0(t)&=-\big[A^{\prime}_0(t)p_0(t)-Q_0(t)(z_0(t)-\rho_0^\pi m(t))+C^{\prime}_0(t)q_0(t)\big]dt+q_0(t)dW_0(t),\\
      p_0(T)&=-G_0\big(z_0(T)-\rho_0^\pi m(T)\big),
    \end{aligned}
    \right.
\end{equation}
as well as the Hamiltonian function
\[\label{Hamiltonian function}
\begin{aligned}
    H_0(t,p,q,x,u)=&\big<p,A_0x+B_0u+F_0^{1,\pi}m+b_0\big>+\big<q,C_0x+D_0u+F_0^{2,\pi}m+\sigma_0\big>\medskip\\
    &-\frac{1}{2}\big<Q_0(x-\rho_0^{\pi}m),x-\rho_0^{\pi}m\big>-\frac{1}{2}\big<R_0u,u\big>.
\end{aligned}
\]
Since $\Gamma_0$ is a closed convex set, for optimal control $u_0^\ast$, related optimal state $z_0^\ast$ and related solution $(p_0^\ast,q_0^\ast)$ to (\ref{first order adjoint equation}),  the SMP reads as
the following local form
\begin{equation}\label{convex maximum principle}
  \left \langle \frac{\partial H_0}{\partial u}(t,p_0^\ast,q_0^\ast,z_0^\ast,u_0^\ast),u-u_0^\ast\right\rangle\leq 0, \quad \text{ for all } u\in\Gamma_0, \text{ a.e. } t\in[0,T],\ \mathbb{P}-a.s.
\end{equation}
Similar to the argument in page 5 of \cite{HHL2016}, using the well-known results of convex analysis (see Theorem 5.2 of \cite{Brezis 2010} or Theorem 4.1 of \cite{HHL2016}), we obtain that (\ref{convex maximum principle}) is equivalent to
\begin{equation}\label{optimal control}
   u_0^\ast(t)=\mathbf{P}_{\Gamma_0}[R^{-1}_0(t)(B_0^{\prime}(t)p_0^\ast(t)+D_0^{\prime}(t)q_0^\ast)(t)], \quad \text{ a.e. } t\in[0,T], \mathbb{P}-a.s.
\end{equation}
where $\mathbf{P}_{\Gamma_0}[\cdot]$ is the projection mapping from $\mathbb{R}^m$ to its closed convex subset $\Gamma_0$ under the norm $\|\cdot\|_{R_0}$ (where $
 \|x\|^2_{R_0}=\left\langle \left\langle x,x\right\rangle\right\rangle:=\big< R_0^{\frac{1}{2}}x,R_0^{\frac{1}{2}}x\big>).$
Finally, by substituting (\ref{optimal control}) in (\ref{major agent state limit 1}) and (\ref{first order adjoint equation}), we get the following Hamiltonian system for the major agent:
\begin{equation}\label{major Hamiltonian system}
\left\{
    \begin{aligned}
      dz_0=&\bigg( A_0z_0+B_0\mathbf{P}_{\Gamma_0}\big[R_0^{-1}\big(B_0^{\prime}p_0+D_0^{\prime}q_0\big)\big]
      +F_0^{1,\pi}m+b_{0}\bigg)dt\medskip\\
      &+\bigg(C_0z_0+D_0\mathbf{P}_{\Gamma_0}\big[R_0^{-1}\big(B_0^{\prime}p_0+D_0^{\prime}q_0\big)\big]
      +F_0^{2,\pi}m+\sigma_0\bigg)dW_0(t),\medskip\\
      dp_0&=-\big(A^{\prime}_0p_0-Q_0(z_0-\rho_0^{\pi}m)+C_0^{\prime}q_0\big)dt+q_0dW_0(t),\medskip\\
      z_0(0)&=x_0,\quad \quad p_0(T)=-G_0\big(z_0(T)-\rho_0^{\pi}m(T)\big).
    \end{aligned}
    \right.
\end{equation}

\subsection{Stochastic optimal control problem for minor agent}

By denoting $F^{1,\pi}:=\pi\otimes F_1$, $F^{2,\pi}:=\pi\otimes F_2$, $\rho^{\pi}:=\pi\otimes \rho I_{n\times n}$, the limiting state of minor agent $\mathcal{A}_i$ is
\[
\left\{
    \begin{aligned}
      dz_i=&\bigg( A_{\theta_i}z_i+Bu_i+F^{1,\pi}m+b\bigg)dt+\bigg(Cz_i+D_{\theta_i}u_i+F^{2,\pi}m+Hz_0+\sigma\bigg)dW_i(t),\medskip\\
     z_i(0)=&x.
    \end{aligned}
    \right.
\]
The limiting cost functional is given by:
\begin{equation}\label{minor agent cost functional limit}
\begin{aligned}
    &{J}_i(u_i)=\frac{1}{2}\mathbb{E}\bigg [\int_0^T \Bigg(\Big <Q\big(z_i\!-\!\rho\sum_{k=1}^{K}\pi_km_k\!-\!(1\!-\!\rho)z_0\big),
    z_i\!-\!\rho\sum_{k=1}^{K}\pi_km_k\!-\!(1\!-\!\rho)z_0\Big>
    \!+\!\big<R_{\theta_i}u_i,u_i\big>\Bigg)dt\\
    &+\Big<G\big(z_i(T)-\rho\sum_{k=1}^{K}\pi_km_k(T)-(1-\rho)z_0(T)\big),
    z_i(T)-\rho\sum_{k=1}^{K}\pi_km_k(T)-(1-\rho)z_0(T)\Big>\bigg],
\end{aligned}
\end{equation}
or equivalently by
\[\label{minor agent cost functional limit 1}
\begin{aligned}
    {J}_i(u_i)=&\frac{1}{2}\mathbb{E}\bigg [ \int_0^T \Bigg(\Big <Q\big(z_i-\rho^\pi m-\!(1\!-\!\rho)z_0\big),z_i-\rho^\pi m-\!(1\!-\!\rho)z_0\Big>+\big<R_{\theta_i}u_i,u_i\big>\Bigg)dt\medskip\\
    &+\Big<G\big(z_i(T)-\rho^\pi m(T)-\!(1\!-\!\rho)z_0(T)\big),z_i(T)-\rho^\pi m(T)-\!(1\!-\!\rho)z_0(T)\Big>\bigg],
\end{aligned}
\]
and the related limiting stochastic optimal control problem for the minor agents is:

\textbf{Problem (LCC-Minor).} For each minor agent $\mathcal{A}_i,$ find $u^\ast_i(\cdot)\in \mathcal{U}_{ad}^{i}$ satisfying
\[J_i(u^\ast_i(\cdot))=\inf_{u_i(\cdot)\in \mathcal{U}_{ad}^{i}}J_i(u_i(\cdot)).
\]
Then $u^{\ast}_i(\cdot)$ is called a decentralized optimal control for Problem (\textbf{LCC-Minor}).
\medskip

Similar to the major agent, we obtain the following Hamiltonian system for minor agent $\mathcal{A}_i$,
\begin{equation}\label{minor Hamiltonian system}
\left\{
    \begin{aligned}
      dz_i=&\bigg( A_{\theta_i}z_i+B\mathbf{P}_{\Gamma_{\theta_i}}\big[R_{\theta_i}^{-1}\big(B^{\prime}p_i
      +D_{\theta_i}^{\prime}q_i\big)\big]+F_1^{\pi}m+b\bigg)dt\medskip\\
      &+\bigg(Cz_i+D_{\theta_i}\mathbf{P}_{\Gamma_{\theta_i}}\big[R_{\theta_i}^{-1}\big(B^{\prime}p_i
      +D_{\theta_i}^{\prime}q_i\big)\big]+F_2^{\pi}m+Hz_0+\sigma\bigg)dW_i(t),\medskip\\
      dp_i=&-\big(A_{\theta_i}^{\prime}p_i-Q(z_i-\rho^{\pi}m-(1-\rho)z_0)
      +C^{\prime}q_i\big)dt+q_idW_i(t)+q_{i,0}dW_0(t),\medskip\\
     z_i(0)=&x,\quad \quad p_i(T)=-G\big(z_i(T)-\rho^{\pi}m(T)-(1-\rho)z_0(T)\big).
    \end{aligned}
    \right.
\end{equation}
Here, $\mathbf{P}_{\Gamma_{\theta_i}}[\cdot]$ is the  projection mapping from $\mathbb{R}^m$ to its closed convex subset $\Gamma_{\theta_i}$ under the norm $\|\cdot\|_{R_{\theta_i}}$. We mention that the limiting minor agent's state $z_i$ depends also on the limiting major agent's state $z_0$, it makes that $z_i$ is $\mathcal{F}^i$-adapted, thus $q_{i,0}dW_0(t)$ appears in the adjoint equation.

\subsection{Consistency condition system for mixed game}

Let us first focus on the $k$-type minor agent. When $i\in\mathcal{I}_k=\left\{i\ |\ \theta_i=k\right\}$, we denote $A_{\theta_i}=A_k$, $D_{\theta_i}=D_k$, $R_{\theta_i}=R_k$ and $\Gamma_{\theta_i}=\Gamma_k$.
We would like to approximate $x_i$ by $z_i$ when  $N\rightarrow+\infty$, thus $m_k$ should satisfy the consistency condition (noticing that Assumption \textbf{(A1)} implies that $N_k\rightarrow\infty$ if $N\rightarrow\infty$)
\[
 m_k(\cdot)=\lim_{N\rightarrow+\infty}\frac{1}{N_k}\sum_{i\in\mathcal{I}_k}z_i(\cdot).
\]
Recall that for $i,j\in\mathcal{I}_k$, $z_i$ and $z_j$ are identically distributed, and conditional independent (under $\mathbb{E}(\cdot\ |\mathcal{F}_{\cdot}^{W_0})$). Thus by conditional strong law of large number, we have (the convergence is in the sense of almost surely, see e.g. \cite{MNZ-2005})
\begin{equation}\label{limit average state}
    m_k(\cdot)=\lim_{N\rightarrow+\infty}\frac{1}{N_k}\sum_{i\in\mathcal{I}_k}z_i(\cdot)
    =\mathbb{E}(z_i(\cdot)|\mathcal{F}_{\cdot}^{W_0}),
\end{equation}where $z_i$ is given by (\ref{minor Hamiltonian system}) with $A_{\theta_i}=A_k, D_{\theta_i}=D_k, R_{\theta_i}=R_k, \Gamma_{\theta_i}=\Gamma_k.$
By combining (\ref{major Hamiltonian system}), (\ref{minor Hamiltonian system}) and (\ref{limit average state}), we get the following consistency condition system or Nash certainty equivalence principle of $k$-type minor agent, for $1\leq k\leq K$: (As mentioned before, for notational brevity, the time argument is suppressed in following equations except $\mathbb{E}(\alpha_k(t)|\mathcal{F}_t^{W_0})$  to emphasise its dependence on conditional expectation under $\mathcal{F}_t^{W_0}$)
\begin{equation}\label{minor cc}
\left\{
    \begin{aligned}
      &d\alpha_k=\bigg( A_{k}\alpha_k\!+\!B\mathbf{P}_{\Gamma_{k}}\big[R_{k}^{-1}\big(B^{\prime}\beta_k
\!+\!D_{k}^{\prime}\gamma_k\big)\big]\!+\!F_1\sum_{i=1}^{K}\pi_i\mathbb{E}(\alpha_i(t)|\mathcal{F}_t^{W_0})
\!+\!b\bigg)dt\medskip\\
      &\!+\!\bigg(C\alpha_k+D_{k}
      \mathbf{P}_{\Gamma_{k}}\big[R_{k}^{-1}\big(B^{\prime}\beta_k+D_{k}^{\prime}\gamma_k\big)\big]
      \!+\!F_2\sum_{i=1}^{K}\pi_i\mathbb{E}(\alpha_i(t)|\mathcal{F}_t^{W_0})\!+\!H\alpha_0\!+\!\sigma\bigg)dW_k(t),
      \medskip\\
      &d\beta_k=\!-\!\big(A_{k}^{\prime}\beta_k\!-\!Q(\alpha_k\!-\!\rho\sum_{i=1}^{K}\pi_i\mathbb{E}
      (\alpha_i(t)|\mathcal{F}_t^{W_0})
      \!-\!(1\!-\!\rho)\alpha_0)
      \!+\!C^{\prime}\gamma_k\big)dt\!+\!\gamma_kdW_k(t)\!+\!\gamma_{k,0}dW_0(t),\medskip\\
     &\alpha_k(0)=x,\quad \quad \beta_k(T)=-G\big(\alpha_k(T)-\rho\sum_{i=1}^{K}\pi_i\mathbb{E}(\alpha_i(T)|\mathcal{F}_T^{W_0})-
     (1-\rho)\alpha_0(T)\big),
    \end{aligned}
    \right.
\end{equation}
where $\alpha_0$ satisfies the following FBSDE which is coupled with all $k$-type minor agents:
\begin{equation}\label{major cc}
\left\{
    \begin{aligned}
    &d\alpha_0=\bigg( A_0\alpha_0\!+\!B_0\mathbf{P}_{\Gamma_0}\big[R_0^{-1}\big(B^{\prime}_0\beta_0\!+\!D^{\prime}_0\gamma_0\big)\big]
    \!+\!F_0^1\sum_{i=1}^{K}\pi_i\mathbb{E}(\alpha_i(t)|\mathcal{F}_t^{W_0})\!+\!b_{0}\bigg)dt\medskip\\
      &\!+\!\bigg(C_0\alpha_0+D_0\mathbf{P}_{\Gamma_0}\big[R_0^{-1}\big(B^{\prime}_0\beta_0\!+\!D^{\prime}_0\gamma_0\big)\big]
      \!+\!F_0^2\sum_{i=1}^{K}\pi_i\mathbb{E}(\alpha_i(t)|\mathcal{F}_t^{W_0})\!+\!\sigma_0\bigg)dW_0(t),\medskip\\
      &d\beta_0=-\big(A^{\prime}_0\beta_0-Q_0(\alpha_0-\rho_0\sum_{i=1}^{K}\pi_i\mathbb{E}(\alpha_i(t)|\mathcal{F}_t^{W_0}))
      +C^{\prime}_0\gamma_0\big)dt+\gamma_0dW_0(t),\medskip\\
      &\alpha_0(0)=x_0,\quad \quad \beta_0(T)=-G_0\big(\alpha_0(T)-\rho_0\sum_{i=1}^{K}\pi_i\mathbb{E}(\alpha_i(T)|\mathcal{F}_T^{W_0})\big).
    \end{aligned}
    \right.
\end{equation}
We consider together the major agent and all kinds of minor agents, i.e. (\ref{major cc}) and (\ref{minor cc}) for all $1\leq k\leq K$, then there arise $2K+2$ fully coupled equations including $K+1$ forward equations and $K+1$ backward equations. Such fully coupled equations are called consistency condition system. Once we can solve it, then $m_k=\mathbb{E}(\alpha_k(t)|\mathcal{F}_t^{W_0})$ which depends on the conditional distribution of $\alpha_k$, this allows us, in  (\ref{minor cc}), to use arbitrary Brownian motion $W_k$ which is independent of $W_0$. Finally, let us introduce the following notation which will be used in the following sections
\begin{equation}\label{sum of mk}
\Phi(t):=\sum_{i=1}^{K}\pi_im_i=\sum_{i=1}^{K}\pi_i\mathbb{E}(\alpha_i(t)|\mathcal{F}_t^{W_0}).
\end{equation}

\section{Existence and uniqueness of consistency condition system: local time horizon case}
This section aims to establish the well-posedness of consistency condition system (\ref{minor cc})-(\ref{major cc}) in small time duration using the method of contraction mapping. Similar to the classical results on FBSDEs, see for example Chapter 1 Section 5 of Ma and Yong \cite{my}, we need introduce the following additional assumption:
\begin{description}
  \item[(A4)] We suppose $R_0^{-1}(\cdot), R_{\theta_i}^{-1}(\cdot)\in L^\infty(0,T;\mathcal{S}^m)$ and
$M_0|D|^2< 1$, where  $|D|:=\max_{0\leq k\leq K}|D_k|$ and $M_0:=\max\left\{|G_0|^2(1+\rho_0^2),\ |G|^2(1+\rho^2+(1-\rho)^2)\right\}$.

\end{description}

For simplicity, we denote
\[
\varphi_0(p,q):=
\mathbf{P}_{\Gamma_0}\big[R_0^{-1}\big(B^{\prime}_0p+D^{\prime}_0q\big)\big], \quad
\varphi_{\theta_i}(p,q):=\mathbf{P}_{\Gamma_{\theta_i}}\big[R_{\theta_i}^{-1}\big(B^{\prime}p
\!+\!D_{\theta_i}^{\prime}q\big)\big].
\]
We have the following theorem:
\begin{theorem}\label{wellposedness theorem}
Assume \emph{\textbf{(A1)-(A4)}}, then there exists a $T_0>0$, such that for any $T\in(0,T_0]$, the system  (\ref{minor cc})-(\ref{major cc}) has a unique solution $(\alpha_0,\beta_0,\gamma_0,\alpha_k,\beta_k,\gamma_k,\gamma_{k,0})$, $1\leq k\leq K$, satisfying
\begin{equation}\label{solution space}
\begin{aligned}
& \alpha_0,\beta_0\in L^{2}_{\mathcal{F}^{W_0}}(\Omega;C(0,T;\mathbb{R}^{n})), \quad
\alpha_k,\beta_k\in L^{2}_{\mathcal{F}^{k}}(\Omega;C(0,T;\mathbb{R}^{n})), \\
&\gamma_0\in L^{2}_{\mathcal{F}^{W_0}}(0,T;\mathbb{R}^{n}), \quad
\gamma_k,\gamma_{k,0} \in L^{2}_{\mathcal{F}^k}(0,T;\mathbb{R}^{n}), \quad 1\leq k\leq K.
\end{aligned}
\end{equation}
\end{theorem}
\begin{proof}
Let $T_0\in(0,1]$ be undetermined and $0<T\leq T_0$. We denote
\[
\begin{aligned}
\mathcal{N}[0,T]:=&L^{2}_{\mathcal{F}^{W_0}}(\Omega;C(0,T;\mathbb{R}^{n}))\times L^{2}_{\mathcal{F}^{1}}(\Omega;C(0,T;\mathbb{R}^{n})) \times\ldots \times L^{2}_{\mathcal{F}^{K}}(\Omega;C(0,T;\mathbb{R}^{n}))\\
&\times L^{2}_{\mathcal{F}^{W_0}}(0,T;\mathbb{R}^{n}) \times L^{2}_{\mathcal{F}^{1}}(0,T;\mathbb{R}^{n})\times
\ldots \times L^{2}_{\mathcal{F}^{K}}(0,T;\mathbb{R}^{n})\\
& \times L^{2}_{\mathcal{F}^{1}}(0,T;\mathbb{R}^{n})\times
\ldots \times L^{2}_{\mathcal{F}^{K}}(0,T;\mathbb{R}^{n}).
\end{aligned}
\]
For $(Y_0,Y_1,\ldots,Y_K,Z_0, Z_1,\ldots,Z_K,\Upsilon_{1,0},\ldots,\Upsilon_{K,0})\in \mathcal{N}[0,T]$, we introduce the following norm:
\begin{equation}\label{norm}
\begin{aligned}
&\|(Y_0,Y_1,\ldots,Y_K,Z_0, Z_1,\ldots,Z_K,\Upsilon_{1,0},\ldots,\Upsilon_{K,0})\|_{{\mathcal{N}}[0,T]}^2\\
:=&
\sup_{t\in[0,T]}\mathbb{E}\bigg\{\sum_{k=0}^K|Y_k(t)|^2
+\sum_{k=0}^K\int_0^T|Z_k(s)|^2ds+\sum_{k=1}^K\int_0^T|\Upsilon_{k,0}(s)|^2ds\bigg\}.
\end{aligned}
\end{equation}
Let $\overline{\mathcal{N}}[0,T]$ be the completion of $\mathcal{N}[0,T]$ in $L^{2}_{\mathcal{F}^{W_0}}(0,T;\mathbb{R}^{n})\times L^{2}_{\mathcal{F}^{1}}(0,T;\mathbb{R}^{n})\times \ldots\times L^{2}_{\mathcal{F}^{k}}(0,T;\mathbb{R}^{n}) \times L^{2}_{\mathcal{F}^{W_0}}(0,T;\mathbb{R}^{n}) \times L^{2}_{\mathcal{F}^{1}}(0,T;\mathbb{R}^{n})\times \ldots\times L^{2}_{\mathcal{F}^{k}}(0,T;\mathbb{R}^{n})\times L^{2}_{\mathcal{F}^{1}}(0,T;\mathbb{R}^{n})\times
\ldots \times L^{2}_{\mathcal{F}^{K}}(0,T;\mathbb{R}^{n})$ under norm (\ref{norm}). Now for any
\[(Y_0^j,Y_1^j,\ldots,Y_K^j,Z_0^j, Z_1^j,\ldots,Z_K^j,\Upsilon_{1,0}^j,\ldots,\Upsilon_{k,0}^j)\in \mathcal{N}[0,T], \quad j=1,2,
\]
we solve respectively the following system including $K+1$ SDEs, for $1\leq k\leq K$:
\begin{equation}\label{conditioanl mean field SDE}
\left\{
\begin{aligned}
&d\alpha_0^j=\Big(A_0\alpha_0^j+B_0\varphi_0(Y_0^j,Z_0^j)
+F_0^1\sum_{i=1}^{K}\pi_i\mathbb{E}[\alpha_i^j(t)|\mathcal{F}_t^{W_0}]
+b_0\Big)dt\\
&\qquad+\Big(C_0\alpha_0^j+D_0\varphi_0(Y_0^j,Z_0^j)
+F_0^2\sum_{i=1}^{K}\pi_i\mathbb{E}[\alpha_i^j(t)|\mathcal{F}_t^{W_0}]
+\sigma_0\Big)dW_0(t)\\
&d\alpha_k^j=\Big(A_k\alpha_k^j+B\varphi_k(Y_k^j,Z_k^j)
+F_1\sum_{i=1}^{K}\pi_i\mathbb{E}[\alpha_i^j(t)|\mathcal{F}_t^{W_0}]
+b\Big)dt\\
&\qquad+\Big(C\alpha_k^j+D_k\varphi_k(Y_k^j,Z_k^j)
+F_2\sum_{i=1}^{K}\pi_i\mathbb{E}[\alpha_i^j(t)|\mathcal{F}_t^{W_0}]
+H\alpha_0^j+\sigma\Big)dW_k(t)\\
&\alpha_0^j(0)=x_0, \quad \alpha_k^j(0)=x.
\end{aligned}
\right.
\end{equation}
Then (\ref{conditioanl mean field SDE}) admits a unique solution  for $j=1,2$,
\[
(\alpha_0^j,\alpha_1^j,\ldots,\alpha_K^j)\in L^{2}_{\mathcal{F}^{W_0}}(\Omega;C(0,T;\mathbb{R}^{n}))\times L^{2}_{\mathcal{F}^{1}}(\Omega;C(0,T;\mathbb{R}^{n})) \times\ldots \times L^{2}_{\mathcal{F}^{K}}(\Omega;C(0,T;\mathbb{R}^{n})).
\]
Indeed, (\ref{conditioanl mean field SDE}) is a $n(K+1)$-dimensional SDE with the mean-field term $\sum_{i=1}^{K}\pi_i\mathbb{E}[\alpha_i(t)|\mathcal{F}_t^{W_0}]$. We can prove the well-posedness of such SDEs system by noticing $\mathbb{E}|\mathbb{E}[\alpha_i(t)|\mathcal{F}_t^{W_0}]|^2\leq \mathbb{E}|\alpha_i(t)|^2$ and by constructing a fixed point using the classical contraction mapping method, we omit the proof here. Now let us denote for $0\leq k\leq K$,
\[
\begin{aligned}
&\hat{\alpha}_k:=\alpha_k^1-\alpha_k^2, \quad
\hat{\varphi}_k:=\varphi_k(Y_k^1,Z_k^1)-\varphi_k(Y_k^2,Z_k^2)\\
&\hat{Y}_k:=Y_k^1-Y_k^2, \quad \hat{Z}_k:=Z_k^1-Z_k^2, \quad \hat{\Upsilon}_{k,0}:=\Upsilon_{k,0}^1-\Upsilon_{k,0}^2.
\end{aligned}
\]
Applying It\^{o}'s formula, we obtain
\[
\begin{aligned}
d|\hat{\alpha}_0|^2=&2\Big<\hat{\alpha}_0,A_0\hat{\alpha}_0+B_0\hat{\varphi}_0
+F_0^1\sum_{i=1}^{K}\pi_i\mathbb{E}[\hat{\alpha}_i(t)|\mathcal{F}_t^{W_0}]\Big>dt\\
&+\Big|C_0\hat{\alpha_0}+D_0\hat{\varphi}_0
+F_0^2\sum_{i=1}^{K}\pi_i\mathbb{E}[\hat{\alpha}_i(t)|\mathcal{F}_t^{W_0}]\Big|^2dt\\
&+2\Big<\hat{\alpha}_0,C_0\hat{\alpha}_0+D_0\hat{\varphi}_0
+F_0^2\sum_{i=1}^{K}\pi_i\mathbb{E}[\hat{\alpha}_i(t)|\mathcal{F}_t^{W_0}]\Big>dW_0(t)\\
\end{aligned}
\]
and
\[
\begin{aligned}
d|\hat{\alpha}_k|^2=&2\Big<\hat{\alpha}_k,A_k\hat{\alpha}_k+B\hat{\varphi}_k
+F_1\sum_{i=1}^{K}\pi_i\mathbb{E}[\hat{\alpha}_i(t)|\mathcal{F}_t^{W_0}]\Big>dt\\
&+\Big|C\hat{\alpha_k}+D_k\hat{\varphi}_k
+F_2\sum_{i=1}^{K}\pi_i\mathbb{E}[\hat{\alpha}_i(t)|\mathcal{F}_t^{W_0}]+H\hat{\alpha}_0\Big|^2\\
&+2\Big<\hat{\alpha}_k,C\hat{\alpha}_k+D_k\hat{\varphi}_k
+F_2\sum_{i=1}^{K}\pi_i\mathbb{E}[\hat{\alpha}_i(t)|\mathcal{F}_t^{W_0}]+H\hat{\alpha}_0\Big>dW_k(t).\\
\end{aligned}
\]
Thus by using  \textbf{(A2)-(A3)}, $\mathbb{E}\left|\mathbb{E}[\hat{\alpha}_i(s)|\mathcal{F}_s^{W_0}]\right|\leq \mathbb{E}|\hat{\alpha}_i(s)|$, $\mathbb{E}\left|\mathbb{E}[\hat{\alpha}_i(s)|\mathcal{F}_s^{W_0}]\right|^2\leq \mathbb{E}|\hat{\alpha}_i(s)|^2$ as well as that $\varphi_k$ is Lipschitz with Lipschitz constant $1$(see Proposition 4.2 of \cite{HHL2016}), we have
\begin{equation}\label{SDE estimate 1}
\begin{aligned}
\mathbb{E}|\hat{\alpha}_0(t)|^2\leq & 2\mathbb{E}\int_0^t\Big(|A_0||\hat{\alpha}_0|^2+|B_0||\hat{\alpha}_0||\hat{\varphi}_0|
+|F_0^1||\hat{\alpha}_0|\sum_{i=1}^{K}\left|\mathbb{E}[\hat{\alpha}_i(s)|\mathcal{F}_s^{W_0}]\right|\Big)ds\\
&+\mathbb{E}\int_0^t\left|C_0\hat{\alpha}_0+D_0\hat{\varphi}_0
+F_0^2\sum_{i=1}^{K}\pi_i\mathbb{E}[\hat{\alpha}_i(s)|\mathcal{F}_s^{W_0}]\right|^2ds\\
\leq & C_{\varepsilon}\mathbb{E}\int_0^t
\sum_{i=0}^{K}|\hat{\alpha}_i|^2ds+\mathbb{E}\int_0^t(|D_0|^2+\varepsilon)(|\hat{Y}_0|^2+|\hat{Z}_0|^2)ds\\
\end{aligned}
\end{equation}
and
\begin{equation}\label{SDE estimate 2}
\begin{aligned}
\mathbb{E}|\hat{\alpha}_k(t)|^2\leq&2\mathbb{E}\int_0^t\Big(|A_k||\hat{\alpha}_k|^2
+|B||\hat{\alpha}_k||\hat{\varphi}_k|
+|F_1||\hat{\alpha}_k|\sum_{i=1}^{K}\left|\mathbb{E}[\hat{\alpha}_i(s)|\mathcal{F}_s^{W_0}]\right|\Big)ds\\
&+\mathbb{E}\int_0^t\left|C\hat{\alpha}_k+D_k\hat{\varphi}_k
+\sum_{i=1}^{K}F_2\mathbb{E}[\hat{\alpha}_i(s)|\mathcal{F}_s^{W_0}]+H\hat{\alpha}_0\right|^2ds\\
\leq & C_{\varepsilon}\mathbb{E}\int_0^t\sum_{i=0}^{K}|\hat{\alpha}_i|^2ds
+\mathbb{E}\int_0^t(|D_k|^2+\varepsilon)(|\hat{Y}_k|^2+|\hat{Z}_k|^2)ds,\\
\end{aligned}
\end{equation}
where $C_\varepsilon$ is a constant independent of $T$, which may vary line by line. Adding up (\ref{SDE estimate 1}) and (\ref{SDE estimate 2}) for $1\leq k\leq K$, we have
\[
\begin{aligned}
\mathbb{E}\sum_{i=0}^{K}|\hat{\alpha}_i(t)|^2\leq
C_{\varepsilon}\mathbb{E}\int_0^t\sum_{i=0}^{K}|\hat{\alpha}_i(s)|^2ds
+\mathbb{E}\int_0^t\sum_{i=0}^{K}(|D_i|^2+\varepsilon)(|\hat{Y}_i|^2+|\hat{Z}_i|^2)ds,
\end{aligned}
\]
and the Gronwall's inequality yields
\begin{equation}\label{SDE estimate 3}
\mathbb{E}\sum_{i=0}^{K}|\hat{\alpha}_i(t)|^2\leq
e^{C_{\varepsilon}T}\mathbb{E}\int_0^T\sum_{i=0}^{K}(|D_i|^2+\varepsilon)(|\hat{Y}_i|^2+|\hat{Z}_i|^2)ds.
\end{equation}
Next, we solve the following BSDEs, for $j=1,2$,
\begin{equation}\label{BSDE}
\left\{
\begin{aligned}
&d\beta_0^j=-\Big[A_0'Y_0^j-Q_0\big(\alpha_0^j
-\rho_0\sum_{i=1}^{K}\pi_i\mathbb{E}[\alpha_i^j(t)|\mathcal{F}_t^{W_0}]\big)+C_0'Z_0^j\Big]dt+\gamma_0^jdW_0(t),\\
&d\beta_k^j=-\Big[A_k'Y_k^j-Q\big(\alpha_k^j
-\rho\sum_{i=1}^{K}\pi_i\mathbb{E}[\alpha_i^j(t)|\mathcal{F}_t^{W_0}]\big)-(1-\rho)\alpha_0^j
+C'Z_k^j\Big]dt\\
&\qquad\qquad\qquad\qquad\qquad\qquad\qquad\qquad+\gamma_k^jdW_k(t)+\gamma_{k,0}^jdW_0(t),\\
&\beta_0^j(T)=-G_0\Big(\alpha_0^j(T)-\rho_0\sum_{i=1}^{K}\pi_i\mathbb{E}[\alpha_i^j(T)|\mathcal{F}_T^{W_0}]\Big),\\
&\beta_k^j(T)=-G\Big(\alpha_k^j(T)-\rho\sum_{i=1}^{K}\pi_i\mathbb{E}[\alpha_i^j(T)|\mathcal{F}_T^{W_0}]
-(1-\rho)\alpha_0^j(T)\Big).
\end{aligned}
\right.
\end{equation}
Since \textbf{(A2)-(A3)} hold and $\alpha_i$, $0\leq i\leq K$ have been solved from (\ref{conditioanl mean field SDE}), the classical result of BSDEs yields that (\ref{BSDE}) admits a unique solution
\[
(\beta_0^j,\beta_1^j,\ldots,\beta_K^j,\gamma_0^j, \gamma_1^j,\ldots,\gamma_K^j, \gamma_{1,0}^j,\ldots,\gamma_{K,0}^j)\in \mathcal{N}[0,T]\subseteq \overline{\mathcal{N}}[0,T].
\]
Thus we have defined a mapping  through (\ref{conditioanl mean field SDE}) and (\ref{BSDE})\footnotesize
\[
\begin{aligned}\footnotesize
\mathcal{T}:&\overline{\mathcal{N}}[0,T]\rightarrow \overline{\mathcal{N}}[0,T], \\
 &(Y_0^j,Y_1^j,\ldots,Y_K^j,Z_0^j, Z_1^j,\ldots,Z_K^j,\Upsilon_{1,0}^j,\ldots,\Upsilon_{K,0}^j)\mapsto (\beta_0^j,\beta_1^j,\ldots,\beta_K^j,\gamma_0^j, \gamma_1^j,\ldots,\gamma_K^j, \gamma_{1,0}^j,\ldots,\gamma_{K,0}^j).
\end{aligned}
\]
\normalsize
Similarly, we denote
\[
\begin{aligned}
\hat{\beta}_k:=\beta_k^1-\beta_k^2, \quad \hat{\gamma}_k:=\gamma_k^1-\gamma_k^2, \text{ for }0\leq k\leq K \text{ and } \hat{\gamma}_{k,0}:=\gamma_{k,0}^1-\gamma_{k,0}^2, \text{ for } 1\leq k\leq K.
\end{aligned}
\]
Applying It\^{o}'s formula to $|\hat{\beta_0}(t)|^2$, and noticing $\mathbb{E}\left|\mathbb{E}[\hat{\alpha}_i(s)|\mathcal{F}_s^{W_0}]\right|^2\leq \mathbb{E}|\hat{\alpha}_i(s)|^2$ , we obtain
\[
\begin{aligned}
&\mathbb{E}\left(|\hat{\beta_0}(t)|^2+\int_t^T|\hat{\gamma}_0|^2ds\right)\\
=&\mathbb{E}|\hat{\beta}_0(T)|^2+2\mathbb{E}\int_t^T\left<\hat{\beta_0},
A_0'\hat{Y}_0-Q_0\Big(\hat{\alpha}_0-\rho_0\sum_{i=1}^{K}\pi_i\mathbb{E}[\hat{\alpha}_i(s)|\mathcal{F}_s^{W_0}]\Big)
+C_0'\hat{Z}_0\right>ds\\
\leq& |G_0|^2(1+\rho_0^2)\mathbb{E}\sum_{i=0}^{K}|\hat{\alpha}_i(T)|^2
\!+\!C_\varepsilon\mathbb{E}\int_t^T|\hat{\beta}_0|^2ds+\mathbb{E}\int_t^T\sum_{i=0}^{K}|\hat{\alpha}_i|^2ds
\!+\!\varepsilon \int_t^T(|\hat{Y}_0|^2\!+\!|\hat{Z}_0|^2)ds.
\end{aligned}
\]
Substituting (\ref{SDE estimate 3}) into above inequality, we have
\begin{equation}\label{BSDE estimate 1}
\begin{aligned}
\mathbb{E}\left(|\hat{\beta_0}(t)|^2\!+\!\int_t^T|\hat{\gamma}_0|^2ds\right)\leq &
\left(|G_0|^2(1+\rho_0^2)\!+\!T\right)e^{C_{\varepsilon}T}\mathbb{E}\sum_{i=0}^{K}\int_0^T
(|D_i|^2+\varepsilon)(|\hat{Y}_i|^2+|\hat{Z}_i|^2)ds\\
&+C_\varepsilon\mathbb{E}\int_t^T|\hat{\beta}_0|^2ds
+\varepsilon \int_t^T(|\hat{Y}_0|^2+|\hat{Z}_0|^2)ds.
\end{aligned}
\end{equation}
Similarly, by applying It\^{o}'s formula to $|\hat{\beta_k}(t)|^2$, $1\leq k\leq K$, we have
\begin{equation}\label{BSDE estimate 2}
\begin{aligned}
&\mathbb{E}\left | \hat{\beta}_k(t)\right |^2+\mathbb{E}\int_{t}^{T}\left | \hat{\gamma}_k\right |^2ds+\mathbb{E}\int_{t}^{T}\left | \hat{\gamma}_{k,0}\right |^2ds\\
=&\mathbb{E}\left |\hat{\beta}_k(T) \right |^2\!+\!2\mathbb{E}\int_{t}^{T}\left \langle \hat{\beta}_k,A_k'\hat{Y}_k\!-\!Q\left(\hat{\alpha}_k\!-\!\rho\sum_{i=1}^{K}\pi_i\mathbb{E}\left [ \hat{\alpha}_i(s)|\mathcal{F}_s^{W_0} \right ]\! -\!(1\!-\!\rho)\hat{\alpha}_0\right)\!-\!C'\hat{Z}_k\right \rangle ds.
\end{aligned}
\end{equation}
Noticing that $\rho\in[0,1]$, we have
\[
\begin{aligned}
\mathbb{E}\left |\hat{\beta}_k(T) \right |^2\leq&|G|^2\Big|\hat{\alpha}_k(T)-\rho\sum_{i=1}^{K}\pi_i\mathbb{E}[\hat{\alpha}_i(T)|\mathcal{F}_T^{W_0}]
-(1-\rho)\hat{\alpha}_0(T)\Big|^2\\
\leq &|G|^2(1+\rho^2+(1-\rho)^2)\left(\mathbb{E}\left|\hat{\alpha}_k(T)-\rho\mathbb{E}[\hat{\alpha}_k(T)|\mathcal{F}_T^{W_0}]\right|^2
+\sum_{i=0,i\neq k}^{K}\mathbb{E}|\hat{\alpha}_i(T)|^2\right)\\
\leq &|G|^2(1+\rho^2+(1-\rho)^2)\mathbb{E}\sum_{i=0}^{K}|\hat{\alpha}_i(T)|^2,
\end{aligned}
\]
where we have used the fact that
\[
\begin{aligned}
&\mathbb{E}\left|\hat{\alpha}_k(T)-\rho\mathbb{E}[\hat{\alpha}_k(T)|\mathcal{F}_T^{W_0}]\right|^2\\
=&\mathbb{E}|\hat{\alpha}_k(T)|^2+\rho^2\mathbb{E}|\mathbb{E}[\hat{\alpha}_k(T)|\mathcal{F}_T^{W_0}]|^2
-2\rho\mathbb{E}\left[\hat{\alpha}_k(T)\mathbb{E}[\hat{\alpha}_k(T)|\mathcal{F}_T^{W_0}]\right]\\
=&\mathbb{E}|\hat{\alpha}_k(T)|^2+\rho^2\mathbb{E}|\mathbb{E}[\hat{\alpha}_k(T)|\mathcal{F}_T^{W_0}]|^2
-2\rho\mathbb{E}\left(\mathbb{E}\left[\hat{\alpha}_k(T)\mathbb{E}[\hat{\alpha}_k(T)|\mathcal{F}_T^{W_0}]\right]
|\mathcal{F}_T^{W_0}\right)\\
=&\mathbb{E}|\hat{\alpha}_k(T)|^2+(\rho^2-2\rho)\mathbb{E}|\mathbb{E}[\hat{\alpha}_k(T)|\mathcal{F}_T^{W_0}]|^2
\leq\mathbb{E}\left|\hat{\alpha}_k(T)\right|^2.
\end{aligned}
\]
Thus, (\ref{BSDE estimate 2}) yields that
\[
\begin{aligned}
&\mathbb{E}\left | \hat{\beta}_k(t)\right |^2+\mathbb{E}\int_{t}^{T}\left | \hat{\gamma}_k\right |^2ds+\mathbb{E}\int_{t}^{T}\left | \hat{\gamma}_{k,0}\right |^2ds\\
=&\mathbb{E}\left |\hat{\beta}_k(T) \right |^2+2\mathbb{E}\int_{t}^{T}\left \langle \hat{\beta}_k,A_k'\hat{Y}_k-Q\left(\hat{\alpha}_k-\rho\sum_{i=1}^{K}\pi_i\mathbb{E}\left [ \hat{\alpha}_i(s)|\mathcal{F}_s^{W_0} \right ] -(1-\rho)\hat{\alpha}_0\right)-C'\hat{Z}_k\right \rangle ds\\
\leq&|G|^2(1+\rho^2+(1-\rho)^2)\mathbb{E}\sum_{i=0}^{K}|\hat{\alpha}_i(T)|^2+C_\varepsilon\mathbb{E}\int_{t}^{T} \left|\beta_k \right |^2ds\\
&\qquad\qquad\qquad+\mathbb{E}\int_{t}^{T} \sum_{i=0}^{K} \left|\alpha_i \right|^2ds+\varepsilon\mathbb{E}\int_t^T(|\hat{Y}_k|^2+|\hat{Z}_k|^2)ds.
\end{aligned}
\]
Substituting (\ref{SDE estimate 3}) into above inequality, we have
\begin{equation}\label{BSDE estimate 3}
\begin{aligned}
&\mathbb{E}\left | \hat{\beta}_k(t)\right |^2+\mathbb{E}\int_{t}^{T}\left | \hat{\gamma}_k\right |^2ds+\mathbb{E}\int_{t}^{T}\left | \hat{\gamma}_{k,0}\right |^2ds\\
\leq & \left(|G|^2(1+\rho^2+(1-\rho)^2)\!+\!T\right)e^{C_{\varepsilon}T}\mathbb{E}\sum_{i=0}^{K}\int_0^T
(|D_i|^2\!+\!\varepsilon)(|\hat{Y}_i|^2\!+\!|\hat{Z}_i|^2)ds\\
&+C_\varepsilon\mathbb{E}\int_t^T|\hat{\beta}_k|^2ds
+\varepsilon\mathbb{E}\int_t^T(|\hat{Y}_k|^2+|\hat{Z}_k|^2)ds.
\end{aligned}
\end{equation}
Adding up (\ref{BSDE estimate 1}) and (\ref{BSDE estimate 3}) for $1\leq k\leq K$, we obtain (recall $|D|^2:=\max_{0\leq k\leq K}|D_k|^2$ and $M_0:=\max\left\{|G_0|^2(1+\rho_0^2),\ |G|^2(1+\rho^2+(1-\rho)^2)\right\}$)
\[
\begin{aligned}
&\mathbb{E}\sum_{i=0}^{K}\left |\hat{\beta}_i\right|^2+\mathbb{E}\sum_{i=0}^{K}\int_{t}^{T}\left | \hat{\gamma}_i \right|^2ds+\mathbb{E}\sum_{i=1}^{K}\int_{t}^{T}\left | \hat{\gamma}_{i,0}\right |^2ds\\
\leq & \left(M_0\!+\!T\right)e^{C_{\varepsilon}T}\mathbb{E}\sum_{i=0}^{K}\int_0^T
(|D_i|^2\!+\!\varepsilon)(|\hat{Y}_i|^2\!+\!|\hat{Z}_i|^2)ds\\
&\qquad\qquad\qquad\!+\!C_\varepsilon\mathbb{E}\int_{t}^{T}\sum_{i=0}^{K}\left |\hat{\beta}_i\right|^2ds+\varepsilon\mathbb{E}\int_{0}^{T}\sum_{i=0}^{K}( | \hat{Y}_i |^2\!+\! | \hat{Z}_i |^2) ds\\
\leq & C_\varepsilon\mathbb{E}\int_{t}^{T}\sum_{i=0}^{K}\left |\hat{\beta}_i\right|^2ds+\left[\left(M_0\!+\!T\right)e^{C_{\varepsilon}T}(|D|^2+\varepsilon)\!+\!\varepsilon\right]
\mathbb{E}\sum_{i=0}^{K}\int_0^T(|\hat{Y}_i|^2+|\hat{Z}_i|^2)ds.
\end{aligned}
\]
The Gronwall's inequality yields that
\begin{equation}\label{contraction 1}
\begin{aligned}
&\mathbb{E}\sum_{i=0}^{K}\left |\hat{\beta}_i\right|^2+\mathbb{E}\sum_{i=0}^{K}\int_{t}^{T}\left | \hat{\gamma}_i \right|^2ds+\mathbb{E}\sum_{i=1}^{K}\int_{t}^{T}\left | \hat{\gamma}_{i,0}\right |^2ds\\
\leq
&e^{C_{\varepsilon}T}
\left[\left(M_0\!+\!T\right)e^{C_{\varepsilon}T}(|D|^2+\varepsilon)+\varepsilon\right]
\mathbb{E}\sum_{i=0}^{K}\int_0^T(|\hat{Y}_i|^2+|\hat{Z}_i|^2)ds\\
\leq & e^{C_{\varepsilon}T}T
\left[\left(M_0\!+\!T\right)e^{C_{\varepsilon}T}(|D|^2+\varepsilon)+\varepsilon\right]
\sup_{0\leq t\leq T}\mathbb{E}\sum_{i=0}^{K}|\hat{Y}_i(t)|^2\\
&+e^{C_{\varepsilon}T}
\left[\left(M_0\!+\!T\right)e^{C_{\varepsilon}T}(|D|^2+\varepsilon)+\varepsilon\right]
\mathbb{E}\sum_{i=0}^{K}\int_0^T|\hat{Z}_i|^2ds\\
\leq & e^{C_{\varepsilon}T}(T+1)
\left[\left(M_0\!+\!T\right)e^{C_{\varepsilon}T}(|D|^2+\varepsilon)+\varepsilon\right]\\
&\qquad\qquad
\cdot\|(\hat{Y}_0,\hat{Y}_1,\ldots,\hat{Y}_K,\hat{Z}_0,\hat{Z}_1,\ldots,\hat{Z}_K,
\hat{\Upsilon}_{1,0},\ldots,\hat{\Upsilon}_{K,0})\|_{\overline{\mathcal{N}}[0,T]}\\
=& e^{C_{\varepsilon}T}(T+1)
\left[M_0e^{C_{\varepsilon}T}(|D|^2+\varepsilon)+\varepsilon+Te^{C_{\varepsilon}T}(|D|^2+\varepsilon)\right]\\
&\qquad\qquad
\cdot\|(\hat{Y}_0,\hat{Y}_1,\ldots,\hat{Y}_K,\hat{Z}_0,\hat{Z}_1,\ldots,\hat{Z}_K,
\hat{\Upsilon}_{1,0},\ldots,\hat{\Upsilon}_{K,0})\|_{\overline{\mathcal{N}}[0,T]}
\end{aligned}
\end{equation}
Noticing that assumption \textbf{(A4)} holds, by first choosing $\varepsilon>0$ small enough such that
$M_0(|D|^2+\varepsilon)+\varepsilon<1$,
then choosing $T>0$ small enough, we obtain from (\ref{contraction 1}) that
\[
\begin{aligned}
&\|(\hat{\beta}_0,\hat{\beta}_1,\ldots,\hat{\beta}_K,\hat{\gamma}_0,\hat{\gamma}_1,
\ldots,\hat{\gamma}_K,\hat{\gamma}_{1,0},\ldots,\hat{\gamma}_{K,0})\|_{\overline{\mathcal{N}}[0,T]}\\
\leq &\delta \|(\hat{Y}_0,\hat{Y}_1,\ldots,\hat{Y}_K,\hat{Z}_0,\hat{Z}_1,\ldots,\hat{Z}_K,
\hat{\Upsilon}_{1,0},\ldots,\hat{\Upsilon}_{K,0})\|_{\overline{\mathcal{N}}[0,T]},
\end{aligned}
\]
for some $0< \delta< 1$. This means that the mapping $\mathcal{T}:\overline{\mathcal{N}}[0,T]\rightarrow \overline{\mathcal{N}}[0,T]$ is contractive. By the contraction mapping theorem, there exists a unique fixed point \[(\beta_0,\beta_1,\ldots,\beta_K,\gamma_0, \gamma_1,\ldots,\gamma_K,\gamma_{1,0},\ldots,\gamma_{K,0})\in\overline{ \mathcal{N}}[0,T].\]
Moreover, classical BSDE theory allows us to show that
\[(\beta_0,\beta_1,\ldots,\beta_K,\gamma_0, \gamma_1,\ldots,\gamma_K,\gamma_{1,0},\ldots,\gamma_{K,0})\in \mathcal{N}[0,T].\]
Let $\alpha_k$, $0\leq k\leq K$, be the corresponding solution of (\ref{conditioanl mean field SDE}). Then, one can obtain that  the system  (\ref{minor cc})-(\ref{major cc}) has a unique solution $(\alpha_0,\beta_0,\gamma_0,\alpha_k,\beta_k,\gamma_k,\gamma_{k,0})$, $1\leq k\leq K$, such that (\ref{solution space}) holds.
\end{proof}

\section{Existence and uniqueness of consistency condition system-global time horizon case}
The section aims to establish the well-posedness of consistency condition system (\ref{minor cc})-(\ref{major cc}) for arbitrary $T$, we first study one general kind of conditional mean-field forward-backward stochastic differential equations (MF-FBSDE) by using the discounting method of Pardoux and Tang \cite{PT-1999}.

Let $(\Omega,
\mathcal F, \mathbb{P})$ is a complete, filtered probability space satisfying usual conditions and $\{W_i(t),\ 0\le i\leq d\}_{0 \leq t \leq T}$ is a $d+1$-dimensional Brownian motion on this space. Let $\mathcal{F}_t$ be the natural filtration generated by  $\{W_i(s), 0\leq i\leq d, 0\leq s\leq t\}$ and augmented by $\mathcal{N}_{\mathbb{P}}$ (the class of $\mathbb{P}$-null sets of $\mathcal{F}$). Let $\mathcal F^{W_0}_t$ be  the augmentation of $\sigma\{W_0(s), 0\leq s\leq t\}$ by $\mathcal{N}_{\mathbb{P}}$. We consider the following general conditional MF-FBSDE:
\begin{equation}\label{General MF-FBSDE}
\left\{
\begin{array}
[c]{l}%
dX(s)= b(s,X(s),\mathbb{E}[X(s)|\mathcal{F}^{W_0}_s],Y(s),Z(s)])ds
+\sigma(s,X(s),\mathbb{E}[X(s)|\mathcal{F}^{W_0}_s],Y(s),Z(s))dW(s),\medskip\\
-dY(s)=f(s,X(s),\mathbb{E}[X(s)|\mathcal{F}^{W_0}_s],Y(s),Z(s))ds-Z(s)dW(s),\quad s\in[0,T],\medskip\\
X(0)=x, \quad Y(T)=g(X(T),\mathbb{E}[X(T)|\mathcal{F}^{W_0}_T]),
\end{array}
\right.
\end{equation}
where the adapted processes $X,Y,Z$ take their values in $\mathbb{R}^{n},\mathbb{R}^{l}$ and $\mathbb{R}^{l\times (d+1)}$, respectively. The coefficients $b,\sigma$ and $f$ are defined on $\Omega\times[0,T]\times\mathbb{R}^{n}\times\mathbb{R}^{n}\times\mathbb{R}^{l}\times\mathbb{R}^{l\times (d+1)}$, such that $b(\cdot,\cdot,x,m,y,z)$, $\sigma(\cdot,\cdot,x,m,y,z)$ and $f(\cdot,\cdot,x,m,y,z)$ are $\{\mathcal{F}_{t}\}$-progressively measurable processes, for all fixed $(x,m,y,z)\in\mathbb{R}^{n}\times\mathbb{R}^{n}\times\mathbb{R}^{l}\times\mathbb{R}^{l\times (d+1)}$. The coefficient $g$ is defined on $\Omega\times\mathbb{R}^{n}\times\mathbb{R}^{n}$ and $g(\cdot,x,m)$ is $\mathcal{F}_{T}$-measurable, for all fixed $(x,m)\in\mathbb{R}^{n}\times\mathbb{R}^{n}$. Moreover, the functions $b,\sigma,f$ and $g$ are continuous w.r.t.  $(x,m,y,z)\in\mathbb{R}^{n}\times\mathbb{R}^{n}\times\mathbb{R}^{l}\times\mathbb{R}^{l\times (d+1)}$ and satisfy the following assumptions:
\begin{itemize}
\item[$\left(H_{1}\right)  $] There exist $\lambda_1,\lambda_2\in\mathbb{R}$ and positive constants $k_{0},k_{i}$, $i=1,2,\ldots,12$ such that for all $t$, $x$, $x_{1}$, $x_{2}$, $m$, $m_{1}$, $m_{2}$, $y$, $y_{1}$, $y_{2}$, $z$, $z_{1}$, $z_{2}$ a.s.
\[
\begin{array}
[c]{rl}
(i) & \langle b(t,x_{1},m,y,z)-b(t,x_{2},m,y,z),x_1-x_2\rangle \leq
\lambda_1|x_{1}-x_{2}|^2,\medskip\\
(ii) & | b(t,x,m_1,y_{1},z_{1})-b(t,x,m_2,y_{2},z_{2})| \leq
k_1|m_1-m_2|+k_2| y_{1}-y_{2}|+k_3|z_{1}-z_{2}|,\medskip\\
(iii) & | b(t,x,m,y,z)| \leq |b(t,0,m,y,z)|+k_0(1+|x|),\medskip\\
(iv) & \langle f(t,x,m,y_1,z)-f(t,x,m,y_2,z),y_1-y_2\rangle \leq
\lambda_2|y_{1}-y_{2}|^2,\medskip\\
(v) & | f(t,x_1,m_1,y,z_{1})-f(t,x_1,m_2,y,z_{2})| \leq
k_4| x_{1}-x_{2}|+k_5|m_1-m_2|
+k_6|z_{1}-z_{2}|,\medskip\\
(vi) & | f(t,x,m,y,z)| \leq |f(t,x,m,0,z)|+k_0(1+|y|),\medskip\\
(vii)& |\sigma(t,x_{1},m_1,y_{1},z_{1})-\sigma(t,x_{2},m_2,y_{2},z_{2})|^{2} \medskip\\ &
\leq k_7^2| x_{1}-x_{2}|^{2}+k_8^2 |m_1-m_2|^2+k_9^2| y_{1}-y_{2}|^{2}
+k_{10}^{2}|z_{1}-z_{2}|^{2},\medskip\\
(viii)& |g(x_{1},m_1)-g(x_{2},m_2)|^2\leq
k_{11}^2| x_{1}-x_{2}|+k_{12}^2| m_{1}-m_{2}|.
\end{array}
\]
\item[$\left(  H_{2}\right)  $] It holds that
\[
\mathbb{E}\int_0^T\left(|b(s,0,0,0,0)|^2+|\sigma(s,0,0,0,0)|^2+|f(s,0,0,0,0)|^2\right)ds+
\mathbb{E}|g(0,0)|^2<+\infty.
\]
\end{itemize}
\begin{remark} \label{additional conditional}
From the mean-field structure of (\ref{General MF-FBSDE}), sometimes the following conditions holds:
$\left(H_{1}\right)-(i'): $ There exist $\lambda_1,\widehat{k}_1\in\mathbb{R}$, such that for all $t$, $y$, $z$  and process  $X_{1}$, $X_{2}$, a.s.
\[
\begin{array}
[c]{rl}
& \mathbb{E}\langle b(t,X_1,\mathbb{E}[X_1(t)|\mathcal{F}^{W_0}_t],y,z)
-b(t,X_{2},\mathbb{E}[X_2(t)|\mathcal{F}^{W_0}_t],y,z),X_1-X_2\rangle \\
&\leq
(\lambda_1+\widehat{k}_1)\mathbb{E}|X_{1}-X_{2}|^2.
\end{array}
\]
For example, if $b(t,x,m,y,z)=\lambda_1 x+\widehat{k}_1m+b_1(y,z)$, then it  obviously satisfies above assumption.
Indeed, our mean-filed FBSDE satisfies this assumption.
\end{remark}
Let $\mathcal{H}$ be an Euclidean space. Recall that $L^2_{\mathcal F}(0,T;\mathcal{H})$, denotes the Hilbert space of $\mathcal{H}$-valued $\{\mathcal{F}_{s}\}$-progressively measurable processes $\{u(s), ~s\in[0,T]\}$ such that
\[\|u\|:=\left(E\int_{0}^{T}|u(s)|^{2}ds\right)^{1/2}<\infty.
\]
For $\lambda\in\mathbb{R}$, we define an equivalent norm on $L^2_{\mathcal F}(0,T;\mathcal{H})$:
\[\|u\|_{\lambda}:=\left(E\int_{0}^{T}e^{-\lambda s}|u(s)|^{2}ds\right)^{1/2}.
\]

Now let us consider MF-FBSDE (\ref{General MF-FBSDE}), its fully-coupled structure arises difficulties for establishing its wellposedness. Similar to \cite{PT-1999}, when the coupling is weak enough, MF-FBSDE (\ref{General MF-FBSDE}) should be solvable. The following is the main results on MF-FBSDE (\ref{General MF-FBSDE}) and the proof is give in the appendix.
\begin{theorem}\label{wellposedness MF-FBSDE PT}
Suppose that assumption $(H_1)$ and $(H_2)$ hold.  Then there exists a $\delta_0>0$, which depends on $k_i,\lambda_1,\lambda_2,T$, for $i=1,4,5,6,7,8,11,12$ such that when $k_2,k_3,k_9,k_{10}\in[0,\delta_0)$, there exists a unique adapted solution $(X,Y,Z)\in L^2_{\mathcal F}(0,T;\mathbb{R}^n)\times L^2_{\mathcal F}(0,T;\mathbb{R}^m)\times L^2_{\mathcal F}(0,T;\mathbb{R}^{m\times (d+1)})$ to MF-FBSDE (\ref{General MF-FBSDE}). Further, if $2(\lambda_1+\lambda_2)<-2k_1-k_6^2-k_7^2-k_8^2$, there exists a $\delta_1>0$, which depends on $k_i,\lambda_1,\lambda_2$, for $i=1,4,5,6,7,8,11,12$  and is independent of  $T$, such that when $k_2,k_3,k_9,k_{10}\in[0,\delta_1)$, there exists a unique adapted solution $(X,Y,Z)$ to MF-FBSDE (\ref{General MF-FBSDE}).
\end{theorem}

\begin{remark} \label{additional result}
If in additional $(H_1)-(i')$ holds (see Remark \ref{additional conditional}), by repeating the above discussion, one can show that if $2(\lambda_1+\lambda_2)<-2\widehat{k}_1-k_6^2-k_7^2-k_8^2$, there exists a $\delta_1>0$, which depends on $\widehat{k}_1,k_i,\lambda_1,\lambda_2$, for $i=4,5,6,7,8,11,12$  and is independent of  $T$, such that when $k_2,k_3,k_9,k_{10}\in[0,\delta_1)$, there exists a unique adapted solution $(X,Y,Z)$ to MF-FBSDE (\ref{General MF-FBSDE}).
\end{remark}
Now let us apply Theorem \ref{wellposedness MF-FBSDE PT} to obtain the well-posedness of consistency condition system (\ref{minor cc})-(\ref{major cc}) which is
\[
\left\{
    \begin{aligned}
      &d\alpha_k=\bigg( A_{k}\alpha_k\!+\!B\mathbf{P}_{\Gamma_{k}}\big[R_{k}^{-1}\big(B^{\prime}\beta_k
\!+\!D_{k}^{\prime}\gamma_k\big)\big]\!+\!F_1\sum_{i=1}^{K}\pi_i\mathbb{E}(\alpha_i(t)|\mathcal{F}_t^{W_0})
\!+\!b\bigg)dt\medskip\\
      &\!+\!\bigg(C\alpha_k+D_{k}
      \mathbf{P}_{\Gamma_{k}}\big[R_{k}^{-1}\big(B^{\prime}\beta_k+D_{k}^{\prime}\gamma_k\big)\big]
      \!+\!F_2\sum_{i=1}^{K}\pi_i\mathbb{E}(\alpha_i(t)|\mathcal{F}_t^{W_0})\!+\!H\alpha_0\!+\!\sigma\bigg)dW_k(t),
      \medskip\\
      &d\beta_k=\!-\!\big(A_{k}^{\prime}\beta_k\!-\!Q(\alpha_k\!-\!\rho\sum_{i=1}^{K}\pi_i\mathbb{E}
      (\alpha_i(t)|\mathcal{F}_t^{W_0})
      \!-\!(1\!-\!\rho)\alpha_0)
      \!+\!C^{\prime}\gamma_k\big)dt\!+\!\gamma_kdW_k(t)\!+\!\gamma_{k,0}dW_0(t),\medskip\\
     &\alpha_k(0)=x,\quad \quad \beta_k(T)=-G\big(\alpha_k(T)-\rho\sum_{i=1}^{K}\pi_i\mathbb{E}(\alpha_i(T)|\mathcal{F}_T^{W_0})-
     (1-\rho)\alpha_0(T)\big),
    \end{aligned}
    \right.
\]
where $\alpha_0$ satisfies the following FBSDE which is coupled with all $k$-type minor agents:
\[
\left\{
    \begin{aligned}
    &d\alpha_0=\bigg( A_0\alpha_0\!+\!B_0\mathbf{P}_{\Gamma_0}\big[R_0^{-1}\big(B^{\prime}_0\beta_0\!+\!D^{\prime}_0\gamma_0\big)\big]
    \!+\!F_0^1\sum_{i=1}^{K}\pi_i\mathbb{E}(\alpha_i(t)|\mathcal{F}_t^{W_0})\!+\!b_{0}\bigg)dt\medskip\\
      &\!+\!\bigg(C_0\alpha_0+D_0\mathbf{P}_{\Gamma_0}\big[R_0^{-1}\big(B^{\prime}_0\beta_0\!+\!D^{\prime}_0\gamma_0\big)\big]
      \!+\!F_0^2\sum_{i=1}^{K}\pi_i\mathbb{E}(\alpha_i(t)|\mathcal{F}_t^{W_0})\!+\!\sigma_0\bigg)dW_0(t),\medskip\\
      &d\beta_0=-\big(A^{\prime}_0\beta_0-Q_0(\alpha_0-\rho_0\sum_{i=1}^{K}\pi_i\mathbb{E}(\alpha_i(t)|\mathcal{F}_t^{W_0}))
      +C^{\prime}_0\gamma_0\big)dt+\gamma_0dW_0(t),\medskip\\
      &\alpha_0(0)=x_0,\quad \quad \beta_0(T)=-G_0\big(\alpha_0(T)-\rho_0\sum_{i=1}^{K}\pi_i\mathbb{E}(\alpha_i(T)|\mathcal{F}_T^{W_0})\big).
    \end{aligned}
    \right.
\]
Recall that
\[
\varphi_0(p,q)=
\mathbf{P}_{\Gamma_0}\big[R_0^{-1}\big(B_0^{\prime}p+D_0^{\prime}q\big)\big], \quad
\varphi_{k}(p,q)=\mathbf{P}_{\Gamma_{k}}\big[R_{k}^{-1}\big(B^{\prime}p
\!+\!D_{k}^{\prime}q\big)\big].
\]
If we denote
\[
\begin{aligned}
&W=(W_0,W_1,\ldots,W_K)^{'},\quad \Pi=(0,\pi_1,\ldots,\pi_K),\quad\alpha=(\alpha_0',\alpha_1',\ldots,\alpha_K')^{'},\quad
\beta=(\beta_0',\beta_1',\ldots,\beta_K')^{'},\medskip\\
&\mathbb{X}=(x_0',x',\ldots,x')^{'},\quad \mathbb{E}(\alpha(t)|\mathcal{F}_t^{W_0})
=(\mathbb{E}(\alpha_0(t)|\mathcal{F}_t^{W_0})',\mathbb{E}(\alpha_1(t)|\mathcal{F}_t^{W_0})',\ldots,
\mathbb{E}(\alpha_K(t)|\mathcal{F}_t^{W_0})')^{'},\medskip\\
&\Phi(\beta,\gamma)=(\varphi_0(\beta_0,\gamma_0),\varphi_1(\beta_1,\gamma_1)
,\ldots,\varphi_K(\beta_K,\gamma_K))^{'},\quad \rho_0^{\Pi}:=\Pi\otimes \rho I_{n\times n},
\quad \rho^{\Pi}:=\Pi\otimes \rho I_{n\times n},\medskip\\
& F_0^{1,\Pi}=\Pi\otimes F^1_0, \quad F_0^{2,\Pi}:=\Pi\otimes F^2_0,\quad
F^{1,\Pi}:=\Pi\otimes F_1, \quad F^{2,\Pi}:=\Pi\otimes F_2,\medskip\\
&\gamma=
\begin{pmatrix}
\gamma_{0} & 0 & \ldots & 0 \\
\gamma_{1,0} & \gamma_{1}  &\ldots & 0 \\
\vdots & \vdots & \ddots  &\vdots \\
\gamma_{K,0} & 0 & \ldots & \gamma_{K}
\end{pmatrix},\quad
\mathbb{B}_0=
\begin{pmatrix}
b_0 \\
b  \\
\vdots   \\
b
\end{pmatrix},\quad
\mathbb{D}_0=
\begin{pmatrix}
\sigma_0 & 0  &\ldots & 0 \\
0 & \sigma  &\ldots  & 0\\
\vdots & \vdots & \ddots & \vdots  \\
0 & 0  &\ldots & \sigma
\end{pmatrix},
\medskip\\
&\mathbb{A}=
\begin{pmatrix}
A_0 & 0  &\ldots & 0 \\
0 & A_{1}  &\ldots & 0 \\
\vdots & \vdots & \ddots & \vdots  \\
0 & 0  &\ldots & A_{K}
\end{pmatrix},\quad
\mathbb{B}=
\begin{pmatrix}
B_0 & 0  &\ldots & 0  \\
0 & B &\ldots & 0  \\
\vdots & \vdots & \ddots & \vdots  \\
0 & 0  &\ldots & B
\end{pmatrix},\quad
\mathbb{R}^{-1}=
\begin{pmatrix}
R_0 & 0  &\ldots & 0  \\
0 & R_1  &\ldots & 0 \\
\vdots & \vdots & \ddots & \vdots  \\
0 & 0  &\ldots & R_K
\end{pmatrix},\medskip\\
&\mathbb{Q}=
\begin{pmatrix}
Q_0 & 0  &\ldots & 0 \\
Q(1-\rho) & Q  &\ldots & 0 \\
\vdots & \vdots & \ddots & \vdots  \\
Q(1-\rho) & 0  &\ldots & Q
\end{pmatrix},
\mathbb{G}=
\begin{pmatrix}
G_0 & 0  &\ldots & 0 \\
G(1-\rho) & G  &\ldots & 0  \\
\vdots & \vdots & \ddots & \vdots  \\
G(1-\rho) & 0  &\ldots & G
\end{pmatrix},
\mathbb{F}_1^{\Pi}=
\begin{pmatrix}
F_0^{1,\Pi}   \\
F^{1,\Pi}  \\
\vdots   \\
 F^{1,\Pi}
\end{pmatrix},
\mathbb{F}_2^{\Pi}=
\begin{pmatrix}
F_0^{2,\Pi}   \\
F^{2,\Pi}  \\
\vdots   \\
 F^{2,\Pi}
\end{pmatrix},\quad
\medskip\\
&
\mathbb{Q}^{\Pi}=
\begin{pmatrix}
Q_0\rho_0^{\Pi}  \\
Q\rho^{\Pi}   \\
\vdots   \\
Q\rho^{\Pi}
\end{pmatrix},\quad
\mathbb{G}^{\Pi}=
\begin{pmatrix}
G_0\rho_0^{\Pi}  \\
G\rho^{\Pi}   \\
\vdots   \\
G\rho^{\Pi}
\end{pmatrix},\quad
\mathbb{H}=
\begin{pmatrix}
0  \\
H  \\
\vdots \\
 H
\end{pmatrix},\quad
\mathbb{H}(\alpha)=
\begin{pmatrix}
0 & 0  &\ldots & 0 \\
0 & H\alpha_0 &\ldots & 0  \\
\vdots & \vdots & \ddots & \vdots  \\
0 & 0  &\ldots & H\alpha_0
\end{pmatrix},
\medskip\\
&
\mathbb{F}_2^{\Pi}(\mathbb{E}(\alpha(t)|\mathcal{F}_t^{W_0}))=
\begin{pmatrix}
F_0^{2,\Pi}\mathbb{E}(\alpha(t)|\mathcal{F}_t^{W_0})& 0  &\ldots & 0   \\
0& F^{2,\Pi}\mathbb{E}(\alpha(t)|\mathcal{F}_t^{W_0})&\ldots  &  0    \\
\vdots & \vdots & \ddots & \vdots   \\
0& 0  &\ldots &F^{2,\Pi}\mathbb{E}(\alpha(t)|\mathcal{F}_t^{W_0})
\end{pmatrix},\medskip\\
\medskip\\
\end{aligned}
\]
\[
\begin{aligned}
&\mathbb{C}=
\begin{pmatrix}
C_0 & 0 &\ldots  & 0  \\
0 & C  &\ldots & 0 \\
\vdots & \vdots & \ddots & \vdots  \\
0 & 0  &\ldots & C
\end{pmatrix},\quad
\mathbb{C}(\alpha)=
\begin{pmatrix}
C_0\alpha_0 & 0  &\ldots & 0 \\
0 & C\alpha  &\ldots  & 0\\
\vdots & \vdots & \ddots & \vdots  \\
0 & 0  &\ldots & C\alpha
\end{pmatrix},\quad
\mathbb{C}(\gamma)=
\begin{pmatrix}
C_0'\gamma_0   \\
C'\gamma_1  \\
\vdots   \\
C'\gamma_K
\end{pmatrix},
\medskip\\
&\mathbb{D}=
\begin{pmatrix}
D_0& 0  &\ldots & 0 \\
0 & D_1  &\ldots & 0  \\
\vdots & \vdots & \ddots & \vdots  \\
0 & 0  &\ldots & D_K
\end{pmatrix},\quad
\mathbb{D}(\beta,\gamma)=
\begin{pmatrix}
D_0\varphi_0(\beta_0,\gamma_0) & 0  &\ldots & 0 \\
0 & D_1\varphi_1(\beta_1,\gamma_1) &\ldots & 0  \\
\vdots & \vdots & \ddots & \vdots  \\
0 & 0  &\ldots & D_K\varphi_K(\beta_K,\gamma_K)
\end{pmatrix}.\quad
\end{aligned}
\]
Using above notations, the system (\ref{minor cc})-(\ref{major cc}) can be written as
\begin{equation}\label{system simple}
\left\{
    \begin{aligned}
    d\alpha=&\bigg( \mathbb{A}\alpha+\mathbb{B}\Phi(\beta,\gamma)
    +\mathbb{F}_1^{\Pi}\mathbb{E}(\alpha(t)|\mathcal{F}_t^{W_0})+\mathbb{B}_{0}\bigg)dt\medskip\\
      &+\bigg(\mathbb{C}(\alpha)+\mathbb{D}(\beta,\gamma)
      +\mathbb{F}_{2}^{\Pi}\Big(\mathbb{E}(\alpha(t)|\mathcal{F}_t^{W_0})\Big)
      +\mathbb{H}(\alpha)+\mathbb{D}_0\bigg)dW(t),\medskip\\
      d\beta=&-\left(\mathbb{A}^{\prime}\beta-\mathbb{Q}\alpha+
      \mathbb{Q}^{\Pi}\mathbb{E}(\alpha(t)|\mathcal{F}_t^{W_0}))
      +\mathbb{C}(\gamma)
      \right)dt+\gamma dW(t),\medskip\\
      \alpha(0)=&\mathbb{X},\quad \quad \beta(T)=-\mathbb{G}\alpha(T)+\mathbb{G}^{\Pi}\mathbb{E}(\alpha(T)|\mathcal{F}_T^{W_0}).
    \end{aligned}
    \right.
\end{equation}
Now let $\lambda^{\ast}$ be the largest eigenvalue of the symmetric matrix $\frac{1}{2}(\mathbb{A}+\mathbb{A}')$. Recalling that the projection operator is Lipschitz continuous with Lipschitz constant $1$, then
by comparing (\ref{system simple}) with (\ref{General MF-FBSDE}), one can check that the coefficients of Assumption $(H_1)$  can be chosen as following
\[
\begin{aligned}
&\lambda_1=\lambda_2=\lambda^{\ast}, \quad k_0=\|\mathbb{A}\|,\quad k_1=\|\mathbb{F}_1^{\Pi}\|, \quad
k_2=k_3=\|\mathbb{R}^{-1}\|\|\mathbb{B}\|(\|\mathbb{B}\|+\|\mathbb{D}\|),\medskip\\
& k_4=\|\mathbb{Q}\|, \quad k_5=\|\mathbb{Q}^{\Pi}\|, \quad k_6=\|\mathbb{C}\|,\quad
k_7^2=4(\|\mathbb{C}\|+\|\mathbb{H}\|)^2, \quad k_8^2=4\|\mathbb{F}_2^{\Pi}\|^2, \medskip\\ &k_9=k_{10}=\|\mathbb{R}^{-1}\|\|\mathbb{D}\|(\|\mathbb{B}\|+\|\mathbb{D}\|),\quad k_{11}^2=2\|\mathbb{G}\|^2,\quad
k_{12}^2=2\|\mathbb{G}^{\Pi}\|^2.
\end{aligned}
\]
Thus by applying Theorem \ref{wellposedness MF-FBSDE PT}, we obtain the following global wellposedness of (\ref{system simple}).
\begin{theorem}\label{wellposedness MF-FBSDE main result 1}
Suppose that
\[
4\lambda^{\ast}<-2\|\mathbb{F}_1^{\Pi}\|-\|\mathbb{C}\|^2-4(\|\mathbb{C}\|+\|\mathbb{H}\|)^2-4\|\mathbb{F}_2^{\Pi}\|^2,
\]
then there exists a $\delta_1>0$, which depends on $\lambda^{\ast},\|\mathbb{F}_1^{\Pi}\|,
\|\mathbb{Q}\|, \|\mathbb{Q}^{\Pi}\|, \|\mathbb{C}\|,
\|\mathbb{H}\|, \|\mathbb{F}_2^{\Pi}\|, \|\mathbb{G}\|,
\|\mathbb{G}^{\Pi}\|$,  and is independent of  $T$, such that when $\|\mathbb{R}^{-1}\|, \|\mathbb{B}\|,\|\mathbb{D}\|\in[0,\delta_1)$, there exists a unique adapted solution $(\alpha,\beta,\gamma)$ to consistency condition system (\ref{minor cc})-(\ref{major cc}).
\end{theorem}
\begin{remark} \label{wellposedness MF-FBSDE main result 2}
Let $\lambda^{\ast}_{\mathbb{F}_1^{\Pi}}$ be the largest eigenvalue of $\frac{1}{2}(\mathbb{F}_1^{\Pi}+(\mathbb{F}_1^{\Pi})')$. Noticing  Remark \ref{additional conditional}, one can check that
$(H_1)-(i')$ holds with $\widehat{k}_1=\lambda^{\ast}_{\mathbb{F}_1^{\Pi}}$. Thus, from Remark \ref{additional result}, we have that if
\[
4\lambda^{\ast}<-2\lambda^{\ast}_{\mathbb{F}_1^{\Pi}}-\|\mathbb{C}\|^2-4(\|\mathbb{C}\|+\|\mathbb{H}\|)^2-4\|\mathbb{F}_2^{\Pi}\|^2,
\]
then there exists a $\delta_1>0$, which depends on $\lambda^{\ast},\lambda^{\ast}_{\mathbb{F}_1^{\Pi}},
\|\mathbb{Q}\|, \|\mathbb{Q}^{\Pi}\|, \|\mathbb{C}\|,
\|\mathbb{H}\|, \|\mathbb{F}_2^{\Pi}\|, \|\mathbb{G}\|,
\|\mathbb{G}^{\Pi}\|$,  and is independent of  $T$, such that when $\|\mathbb{R}^{-1}\|, \|\mathbb{B}\|,\|\mathbb{D}\|\in[0,\delta_1)$, there exists a unique adapted solution $(\alpha,\beta,\gamma)$ to consistency condition system (\ref{minor cc})-(\ref{major cc}).
\end{remark}

\section{$\varepsilon$-Nash Equilibrium for Problem (\textbf{CC})}
In Section 2, we characterized the decentralized strategies $\{\bar{u}^i_t, 0\le i\le N\}$ of \textbf{Problem (CC)}
through the auxiliary \textbf{Problem (LCC)} and consistency condition system. Now, we turn to verify the $\varepsilon$-Nash equilibrium of these decentralized strategies. Here, we proceed our verification based on the assumptions of local time horizon case (Section 4). Note that it can also be verified based on global horizon case (Section 5) without essential difficulties. For major agent $\mathcal{A}_0$ and minor agent $\mathcal{A}_i$,
the decentralized states $\breve{x}_t^{0}$ and $\breve{x}_t^{i}$ are given respectively by
\begin{equation}\label{decentralized state}
\left\{
        \begin{aligned}d\breve{x}_{0}&=\left[A_0\breve{x}_{0}\!+\!B_0\varphi_0(\bar{p}_0,\bar{q}_0)\!+\!F_0^1\breve{x}^{(N)}\!+\!b_0\right]dt
        \!+\!\left[C_0\breve{x}_{0}\!+\!D_0\varphi_0(\bar{p}_0,\bar{q}_0)
      +F_0^2\breve{x}^{(N)}\!+\!\sigma_0\right]dW_0(t),\\
      d\breve{x}_{i}&=\left[A_{\theta_i}\breve{x}_{i}\!+\!B\varphi_{\theta_i}(\bar{p}_i,\bar{q}_i)
      \!+\!F_1\breve{x}^{(N)}
      \!+\!b\right]dt\!+\!\left[C\breve{x}_{i}\!+\!D_{\theta_i}\varphi_{\theta_i}(\bar{p}_i,\bar{q}_i)
      +F_2\breve{x}^{(N)}\!+\!H\breve{x}_0\!+\!\sigma\right]dW_i(t),\\
      \breve{x}_{0}(0)&=x_0, \quad \breve{x}_{i}(0)=x,
    \end{aligned}
    \right.
\end{equation}
where $\breve{x}^{(N)}=\frac{1}{N}\sum_{i=1}^{N}\breve{x}^{i}$ and the processes $(\bar{p}_0,\bar{q}_0,\bar{p}_i,\bar{q}_i)$ are solved by
\begin{equation}\label{limiting state under cc condition}
\left\{
    \begin{aligned}
    &d\bar{x}_0=\bigg( A_0\bar{x}_0\!+\!B_0\varphi_0(\bar{p}_0,\bar{q}_0)
    \!+\!F_0^1\sum_{k=1}^{K}\pi_k\mathbb{E}(\alpha_k(t)|\mathcal{F}_t^{W_0})\!+\!b_{0}\bigg)dt\medskip\\
      &\qquad\!+\!\bigg(C_0\bar{x}_0+D_0\varphi_0(\bar{p}_0,q\bar{}_0)
      \!+\!F_0^2\sum_{k=1}^{K}\pi_k\mathbb{E}(\alpha_k(t)|\mathcal{F}_t^{W_0})\!+\!\sigma_0\bigg)dW_0(t),\medskip\\
      &d\bar{p}_0=-\big(A^{\prime}\bar{p}_0
      -Q_0(\bar{x}_0-\rho_0\sum_{i=1}^{K}\pi_k\mathbb{E}(\alpha_k(t)|\mathcal{F}_t^{W_0}))
      +C^{\prime}\bar{q}_0\big)dt+\bar{q}_0dW_0(t),\medskip\\
      &d\bar{x}_i=\bigg( A_{\theta_i}\bar{x}_i\!+\!B\varphi_{\theta_i}(\bar{p}_i,\bar{q}_i)
      \!+\!F_1\sum_{k=1}^{K}\pi_k\mathbb{E}(\alpha_k(t)|\mathcal{F}_t^{W_0})
      \!+\!b\bigg)dt\medskip\\
      &\qquad\!+\!\bigg(C\bar{x}_i+D_{\theta_i}\varphi_{\theta_i}(\bar{p}_i,\bar{q}_i)
      \!+\!F_2\sum_{k=1}^{K}\pi_k\mathbb{E}(\alpha_k(t)|\mathcal{F}_t^{W_0})\!+\!H\bar{x}_0\!+\!\sigma\bigg)dW_i(t),\medskip\\
      &d\bar{p}_i=\!-\!\big(A_{\theta_i}^{\prime}\bar{p}_i
      \!-\!Q(\bar{x}_i\!-\!\rho\sum_{k=1}^{K}\pi_k\mathbb{E}(\alpha_k(t)|\mathcal{F}_t^{W_0})\!-\!(1\!-\!\rho)\bar{x}_0)
      \!+\!C^{\prime}\bar{q}_i\big)dt\!+\!\bar{q}_idW_i(t)\!+\!\bar{q}_{i,0}dW_0(t),\medskip\\
        &\bar{x}_0(0)=x_0,\quad \quad \bar{p}_0(T)=-G_0\big(\bar{x}_0(T)-\rho_0\sum_{k=1}^{K}\pi_k\mathbb{E}(\alpha_k(T)|\mathcal{F}_T^{W_0})\big),\\
     &\bar{x}_i(0)=x,\quad \quad \bar{p}_i(T)=-G\big(\bar{x}_i(T)-\rho\sum_{k=1}^{K}\pi_k\mathbb{E}(\alpha_k(T)|\mathcal{F}_T^{W_0})
     -(1-\rho)\bar{x}_0(T)\big).
    \end{aligned}
    \right.
\end{equation}
Here we recall that
\[
\varphi_0(p,q):=
\mathbf{P}_{\Gamma_0}\big[R_0^{-1}\big(B^{\prime}_0p+D^{\prime}_0q\big)\big], \quad
\varphi_{\theta_i}(p,q):=\mathbf{P}_{\Gamma_{\theta_i}}\big[R_{\theta_i}^{-1}\big(B^{\prime}p
\!+\!D_{\theta_i}^{\prime}q\big)\big],
\]
and $\alpha_k$, $1\leq k\leq K$ are given by (\ref{minor cc}) and (\ref{major cc}). We mention that (\ref{limiting state under cc condition}) gives also the dynamics of the limiting state $(\bar{x}_0,\bar{x}_{1},\ldots,\bar{x}_{N})$ and one can check easily that $(\bar{x}_0,\bar{p}_0,\bar{q}_0)=(\alpha_0,\beta_0,\gamma_0)$.
Now, we would like to show that for $\bar{u}_0=\varphi_{0}(\bar{p}_0,\bar{q}_0)$ and  $\bar{u}_i=\varphi_{\theta_i}(\bar{p}_i,\bar{q}_i)$, $1\leq i\leq N$,
$(\bar{u}_0,\bar{u}_1,\cdots,\bar{u}_N)$ is an $\varepsilon$-Nash equilibrium of \textbf{Problem (CC)}.
Let us first present following several lemmas.
\begin{lemma}\label{lemma for xN}
Under  \emph{\textbf{(A1)-(A4)}}, there exists a constant $M$ independent of $N$, such that
\[
\sup_{0\leq i\leq N}\mathbb{E}\sup_{0\leq t\leq T}\Big|\breve{x}^{i}(t)\Big|^2\leq M.
\]
\end{lemma}
\begin{proof}
From  Theorem \ref{wellposedness theorem}, we know that on a small time interval the system of fully coupled FBSDE (\ref{minor cc})-(\ref{major cc}) has a unique solution (for the global case, see Theorem \ref{wellposedness MF-FBSDE main result 1} and Remark \ref{wellposedness MF-FBSDE main result 2})
\[
(\alpha_0,\beta_0,\gamma_0)\in L^{2}_{\mathcal{F}^{W_0}}(0,T;\mathbb{R}^{n\times 3}) \text{ and }
(\alpha_k,\beta_k,\gamma_k,\gamma_{k,0})\in L^{2}_{\mathcal{F}^k}(0,T;\mathbb{R}^{n\times 4}), \quad 1\leq k\leq K.
\]
Then, the classical results on FBSDEs yields that (\ref{limiting state under cc condition}) also has a unique solution
\[
(\bar{x}_0,\bar{p}_0,\bar{q}_0)\in L^{2}_{\mathcal{F}^{W_0}}(0,T;\mathbb{R}^{n\times 3}) \text{ and }
(\bar{x}_i,\bar{p}_i,\bar{q}_i,\bar{q}_{i,0})\in L^{2}_{\mathcal{F}^i}(0,T;\mathbb{R}^{n\times 4}), \quad 1\leq i\leq N.
\]
(Indeed, FBSDEs (\ref{limiting state under cc condition}) has a unique solution for arbitrary $T$ by using the Theorem 2.1 of \cite{HHL2016} and the results of \cite{HP1995,PW1999}).
Thus, SDEs system (\ref{decentralized state})
has also a unique solution
\[
(\breve{x}_0,\breve{x}_1\ldots,\breve{x}_N)\in L^{2}_{\mathcal{F}}(0,T;\mathbb{R}^n)\times L^{2}_{\mathcal{F}}(0,T;\mathbb{R}^n)\times\ldots\times L^{2}_{\mathcal{F}}(0,T;\mathbb{R}^n).
\]
Moreover, since $\{W_i\}_{i=1}^N$ is $N$-dimensional Brownian motion whose components are independent identically distributed, we have that for the $k$-type minor agents, under the conditional expectation $\mathbb{E}(\cdot|\mathcal{F}^{W_0}))$, for each $1\leq k\leq K$, the processes $(\bar{x}_i,\bar{p}_i,\bar{q}_i)$, $i\in\mathcal{I}_k$, are independent identically distributed. We also note that for each $1\leq k\leq K$,  $\breve{x}_i$,  $i\in\mathcal{I}_k$, are identically distributed. Noticing that $(\bar{p}_0,\bar{q}_0)\in L^{2}_{\mathcal{F}^{W_0}}(0,T;\mathbb{R}^n)\times L^{2}_{\mathcal{F}^{W_0}}(0,T;\mathbb{R}^n)$, and $(\bar{p}_i,\bar{q}_i)\in L^{2}_{\mathcal{F}^i}(0,T;\mathbb{R}^n)\times L^{2}_{\mathcal{F}^i}(0,T;\mathbb{R}^n)$, $1\leq i\leq N$, then the Lipschitz property of the projection onto the convex set
yields that  $\varphi_{\theta_0}(\bar{p}_0,\bar{q}_0):=\varphi_{0}(\bar{p}_0,\bar{q}_0)=\mathbf{P}_{\Gamma_0}\Big(R_{0}^{-1}
\big(B^T\bar{p}_0+D^T\bar{q}_0\big)\Big)\in L^{2}_{\mathcal{F}^{W_0}}(0,T;\Gamma_0)$ and $\varphi_{\theta_i}(\bar{p}_i,\bar{q}_i):=\mathbf{P}_{\Gamma_{\theta_i}}\Big(R_{\theta_i}^{-1}
\big(B^T\bar{p}_i+D^T\bar{q}_i\big)\Big)\in L^{2}_{\mathcal{F}^i}(0,T;\Gamma_{\theta_i})$, $1\leq i\leq N$. Moreover, there exists a constant $M$ independent of $N$ which may very line by line in the following, such that for all $0\leq i\leq N$, $0\leq k\leq K$,
\begin{equation}\label{bounded}
\begin{aligned}
 &\mathbb{E}\sup_{0\leq t\leq T}(|\alpha_k(t)|^2+|\beta_k(t)|^2+|\bar{x}_i(t)|^2+|\bar{p}_i(t)|^2)\\
 &\qquad\qquad+\mathbb{E}\int_0^T(|\gamma_k(t)|^2+|\bar{q}_i(t)|^2+|\varphi_{\theta_i}(\bar{p}_i(t),\bar{q}_i(t))|^2)
 \leq M.
\end{aligned}
\end{equation}
From (\ref{decentralized state}),  by using Burkholder-Davis-Gundy (BDG) inequality,
it follows that for any $t\in[0,T]$,
\[
  \begin{aligned}
\mathbb{E}\sup_{0\leq s\leq t}|\breve{x}_{0}(s)|^2\leq & M+M\mathbb{E}\int_0^t\Big[|\breve{x}_{0}(s)|^2\!+\!|\breve{x}^{(N)}(s)|^2\Big]ds
\leq  M+M\mathbb{E}\int_0^t\Big[|\breve{x}_{0}(s)|^2\!+\!\frac{1}{N}\sum_{i=1}^N|\breve{x}_{i}(s)|^2\Big]ds,
\end{aligned}
\]
and by Gronwall's inequality, we obtain
\begin{equation}\label{ee80}
\mathbb{E}\sup_{0\leq s\leq t}|\breve{x}_{0}(s)|^2
\leq  M+M\mathbb{E}\int_0^t\frac{1}{N}\sum_{i=1}^N|\breve{x}_{i}(s)|^2ds,
\end{equation}
Similarly, from (\ref{decentralized state}) again, we have
\[
  \begin{aligned}
\mathbb{E}\sup_{0\leq s\leq t}|\breve{x}_{i}(s)|^2
\leq & M+M\mathbb{E}\int_0^t\Big[|\breve{x}_{0}(s)|^2\!+\!|\breve{x}_{i}(s)|^2
\!+\!\frac{1}{N}\sum_{i=1}^N|\breve{x}_{i}(s)|^2\Big]ds.
\end{aligned}
\]
Then using (\ref{ee80}), we get
\begin{equation}\label{ee8}
  \begin{aligned}
\mathbb{E}\sup_{0\leq s\leq t}|\breve{x}_{i}(s)|^2
\leq & M+M\mathbb{E}\int_0^t\Big[|\breve{x}_{i}(s)|^2\!+\!\frac{1}{N}\sum_{i=1}^N|\breve{x}_{i}(s)|^2\Big]ds.
\end{aligned}
\end{equation}
Thus
\[
\mathbb{E}\sup_{0\leq s\leq t}\sum_{i=1}^N|\breve{x}_{i}(s)|^2\leq \mathbb{E}\sum_{i=1}^N\sup_{0\leq s\leq t}|\breve{x}_{i}(s)|^2
\leq MN+2M\mathbb{E}\int_0^t\Big[\sum_{i=1}^N|\breve{x}_{i}(s)|^2\Big]ds.
\]
By Gronwall's inequality, it is easy to obtain
$\mathbb{E}\sup\limits_{0\leq t\leq T}\sum_{i=1}^N|\breve{x}_{i}(t)|^2=O(N)$, for any $1\leq i\leq N$.
Then, substituting this estimate to (\ref{ee8}) and using Gronwall's inequality once again, we have
$
\mathbb{E}\sup_{0\leq t\leq T}\Big|\breve{x}_{i}(t)\Big|^2\leq M, \text{ for any } 1\leq i\leq N.
$
By applying this estimate to (\ref{ee80}), we get
$\mathbb{E}\sup\limits_{0\leq t\leq T}\Big|\breve{x}_{0}(t)\Big|^2\leq M.$
We complete the proof.
\end{proof}
\bigskip

Now, we recall that
\[
\breve{x}^{(N)}=\frac{1}{N}\sum_{k=1}^{K}\sum_{i\in\mathcal{I}_k}\breve{x}_i
=\sum_{k=1}^{K}\pi_k^{(N)}\frac{1}{N_k}\sum_{i\in\mathcal{I}_k}\breve{x}_i, \quad \text{ and } \quad
\Phi(t)=\sum_{k=1}^{K}\pi_k\mathbb{E}(\alpha_k(t)|\mathcal{F}_t^{W_0}),
\]
then we have
\begin{lemma}\label{lemma for xN and m}
Under \emph{\textbf{(A1)-(A4)}}, there exists a constant $M$ independent of $N$ such that
\[
\mathbb{E}\sup_{0\leq t\leq T}\Big|\breve{x}^{(N)}(t)-\Phi(t)\Big|^2\leq M\Big(\frac{1}{N}+\varepsilon_N^2\Big), \text { where } \varepsilon_N=\sup\limits_{1\leq k\leq K}|\pi_k^{(N)}-\pi_k|.
\]
\end{lemma}
\begin{proof}
For each fixed $1\leq k\leq K$, we consider the $k$-type minor agents.
We denote $\breve{x}^{(k)}:=\frac{1}{N_k}\sum_{i\in\mathcal{I}_k}\breve{x}_i$.
Let us add up $N_k$ states of all $k$-type minor agents and then divide by $N_k$, we have
\begin{equation}\label{k type sum}
   \begin{aligned}
      d\breve{x}^{(k)}&=\left[A_{k}\breve{x}^{(k)}\!+\!\frac{B}{N_k}\sum_{i\in\mathcal{I}_k}\varphi_k(\bar{p}_i,\bar{q}_i)
      \!+\!F^1\breve{x}^{(N)}\!+\!b_0\right]dt\\
        &\qquad\!+\!\frac{1}{N_k}\sum_{i\in\mathcal{I}_k}\left[C\breve{x}_{i}\!+\!D_{k}\varphi_k(\bar{p}_i,\bar{q}_i)
      +F_2\breve{x}^{(N)}\!+\!H\breve{x}_0\!+\!\sigma_0\right]dW_i(t), \qquad
      \breve{x}^{(k)}(0)=x.
    \end{aligned}
\end{equation}
Recall $m_k=\mathbb{E}(\alpha_k(t)|\mathcal{F}_t^{W_0})$, and taking conditional expectation $\mathbb{E}(\cdot|\mathcal{F}_t^{W_0})$ on the first equation of (\ref{minor cc}) and noticing that $\mathbb{E}(\varphi_k(\bar{p}_i,\bar{q}_i)|\mathcal{F}_t^{W_0})
=\mathbb{E}(\varphi_k(\beta_k,\gamma_k)|\mathcal{F}_t^{W_0})$, for any $i\in\mathcal{I}_k$ and $\alpha_k(s)$ is $\mathcal{F}_s^k$-adapted, then we have
\begin{equation}\label{mk}
    \begin{aligned}
    dm_k=\bigg( A_{k}m_k\!+\!B\mathbb{E}(\varphi_k(\bar{p}_i,\bar{q}_i)|\mathcal{F}_t^{W_0})
      \!+\!F_1\Phi
      \!+\!b\bigg)dt, \qquad m_k(0)=x.
    \end{aligned}
\end{equation}
From (\ref{k type sum}) and  (\ref{mk}), by denoting  $\Delta_k(t):=\breve{x}^{(k)}(t)-m_k(t)$, we have
\[
\begin{aligned}
d\Delta_k&=\Bigg[A_k\Delta_k\!+\!F_1(\breve{x}^{(N)}\!-\!\Phi)
\!+\!\frac{B}{N_k}\Big(\sum_{i\in\mathcal{I}_k}\varphi_k(\bar{p}_i,\bar{q}_i)
\!-\!\mathbb{E}(\varphi_k(\bar{p}_i,\bar{q}_i)|\mathcal{F}_t^{W_0})\Big)\Bigg]dt\\
&\quad+\frac{1}{N_k}\sum_{i\in\mathcal{I}_k}\left[C\breve{x}_{i}\!+\!D_{k}\varphi_k(\bar{p}_i,\bar{q}_i)
      +F_2\breve{x}^{(N)}\!+\!H\breve{x}_0\!+\!\sigma_0\right]dW_i(t),\qquad
\Delta(0)=0.
\end{aligned}
\]
The inequality $(x+y)^2\leq 2x^2+2y^2$ and Cauchy-Schwartz inequality yield that
\[
\begin{aligned}
&\mathbb{E}\sup_{0\leq s\leq t}|\Delta_k(s)|^2\leq M\mathbb{E}\int_0^t\Big[|\Delta_k(s)|^2+|\breve{x}^{(N)}(s)-\Phi(s)|^2\Big]ds\\
&\qquad+M\mathbb{E}\int_0^t\left|\frac{1}{N_k}\sum_{i\in\mathcal{I}_k}\varphi_k(\bar{p}_i(s),\bar{q}_i(s))
-\mathbb{E}(\varphi_k(\bar{p}_i(s),\bar{q}_i(s))|\mathcal{F}_s^{W_0})\right|^2ds\\
&\qquad+\frac{2}{N_k^2}\mathbb{E}\sup_{0\leq r\leq t}\left|\int_0^r\sum_{i\in\mathcal{I}_k}\left[C\breve{x}_{i}\!+\!D_{k}\varphi_k(\bar{p}_i,\bar{q}_i)
      +F_2\breve{x}^{(N)}\!+\!H\breve{x}_0\!+\!\sigma_0\right]dW_i(s)\right|^2.
\end{aligned}
\]
From BDG inequality, we obtain
\begin{equation}\label{ee3}
\begin{aligned}
&\mathbb{E}\sup_{0\leq s\leq t}|\Delta_k(s)|^2\leq M\mathbb{E}\int_0^t\Big[|\Delta_k(s)|^2+|\breve{x}^{(N)}(s)-\Phi(s)|^2\Big]ds\\
&\qquad +M\mathbb{E}\int_0^t\Bigg|\frac{1}{N_k}\sum_{i\in\mathcal{I}_k}\varphi_k(\bar{p}_i(s),\bar{q}_i(s))
-\mathbb{E}(\varphi_k(\bar{p}_i(s),\bar{q}_i(s))|\mathcal{F}_s^{W_0})\Bigg|^2ds\\
&\qquad+\frac{M}{N_k^2}\mathbb{E}\sum_{i\in\mathcal{I}_k}\int_0^t\left|C\breve{x}_{i}
\!+\!D_{k}\varphi_k(\bar{p}_i,\bar{q}_i)
      +F_2(\breve{x}^{(N)}-\Phi)\!+\!F_2\Phi\!+\!H\breve{x}_0\!+\!\sigma_0\right|^2ds.
\end{aligned}
\end{equation}
Let us first focus on the second term of the right hand-side of (\ref{ee3}). Since for each fixed $s\in[0,T]$, under the conditional expectation $\mathbb{E}(\cdot|\mathcal{F}_s^{W_0}))$, for each $1\leq k\leq K$, the processes $(\bar{x}_i(s),\bar{p}_i(s),\bar{q}_i(s))$, $i\in\mathcal{I}_k$, are independent identically distributed, if we denote $\mu(s)=\mathbb{E}(\varphi_k(\bar{p}_i(s),\bar{q}_i(s))|\mathcal{F}_s^{W_0}))$, then $\mu$ does not depend on $i$ and moreover we have
\[
\begin{aligned}
&\mathbb{E}\left|\frac{1}{N_k}\sum_{i\in\mathcal{I}_k}\varphi_k(\bar{p}_i(s),\bar{q}_i(s))
-\mu(s)\right|^2
=\frac{1}{N^2_k}\mathbb{E}\left|\sum_{i\in\mathcal{I}_k}\big[\varphi_k(\bar{p}_i(s),\bar{q}_i(s))
-\mu(s)\big]\right|^2\\
=&\frac{1}{N^2_k}\mathbb{E}\left(\sum_{i\in\mathcal{I}_k}\left|\varphi_k(\bar{p}_i(s),\bar{q}_i(s))\!-\!\mu(s)\right|^2
\!+\!\sum_{i,j\in\mathcal{I}_k,j\neq i}\left\langle\varphi_k(\bar{p}_i(s),\bar{q}_i(s))
\!-\!\mu(s),\varphi_k(\bar{p}_j(s),\bar{q}_j(s))\!-\!\mu(s)\right\rangle\right).
\end{aligned}
\]
Since $(\bar{p}_i(s),\bar{q}_i(s))$, $i\in\mathcal{I}_k$, are independent under $\mathbb{E}(\cdot|\mathcal{F}_s^{W_0}))$, we have
\[
\begin{aligned}
&\mathbb{E}\sum_{i,j\in\mathcal{I}_k,j\neq i}\left\langle\varphi_k(\bar{p}_i(s),\bar{q}_i(s))\!-\!\mu(s),
\varphi_k(\bar{p}_j(s),\bar{q}_j(s))\!-\!\mu(s)\right\rangle\\
=&\mathbb{E}\left[\sum_{i,j\in\mathcal{I}_k,j\neq i}\mathbb{E}\left(\left\langle\varphi_k(\bar{p}_i(s),\bar{q}_i(s))\!
-\!\mu(s),\varphi_k(\bar{p}_j(s),\bar{q}_j(s))\!-\!\mu(s)\right\rangle\Big|\mathcal{F}_s^{W_0}\right)\right]\\
=&\mathbb{E}\left[\sum_{i,j\in\mathcal{I}_k,j\neq i}\left\langle\mathbb{E}\left(\varphi_k(\bar{p}_i(s),\bar{q}_i(s))\!
-\!\mu(s)\Big|\mathcal{F}_s^{W_0}\right),
\mathbb{E}\left(\varphi_k(\bar{p}_j(s),\bar{q}_j(s))\!-\!\mu(s)\Big|\mathcal{F}_s^{W_0}\right)\right\rangle\right]
=0.
\end{aligned}
\]
Then, due to the fact that $(\bar{p}_i,\bar{q}_i)$, $1\leq i\leq N$  are identically distributed, we obtain
\[
\begin{aligned}
&\mathbb{E}\int_0^t\Bigg|\frac{1}{N_k}\sum_{i\in\mathcal{I}_k}\varphi_k(\bar{p}_i(s),\bar{q}_i(s))
-\mathbb{E}(\varphi_k(\bar{p}_i(s),\bar{q}_i(s))|\mathcal{F}_s^{W_0})\Bigg|^2ds\\
=& \frac{1}{N^2_k}\int_0^t\mathbb{E}\sum_{i\in\mathcal{I}_k}
\left|\varphi_k(\bar{p}_i(s),\bar{q}_i(s))\!-\!\mu(s)\right|^2ds
\!=\!\frac{1}{N_k}\int_0^t\mathbb{E}\left|\varphi_k(\bar{p}_i(s),\bar{q}_i(s))\!-\!\mu(s)\right|^2ds
\!\leq\!\frac{M}{N_k},
\end{aligned}
\]
where the last equality due to (\ref{bounded}). Now we focus on the third term of the right hand-side of (\ref{ee3}), using (\ref{bounded}), Lemma \ref{lemma for xN} and that $(\breve{x}_i(s),\bar{p}_i(s),\bar{q}_i(s))$, $i\in\mathcal{I}_k$, are identically distributed, it follows
\[
\begin{aligned}
&\frac{M}{N^2_k}\sum_{i\in\mathcal{I}_k}\mathbb{E}\int_0^t\left|C\breve{x}_{i}\!+\!D_{k}\varphi_k(\bar{p}_i,\bar{q}_i)
      +F_2(\breve{x}^{(N)}-\Phi)\!+\!F_2\Phi\!+\!H\breve{x}_0\!+\!\sigma_0\right|^2ds\\
\leq & \frac{M}{N^2_k}\sum_{i\in\mathcal{I}_k}\mathbb{E}\int_0^t\left(|\breve{x}_{i}(s)|^2
\!+\!|\varphi_k(\bar{p}_i,\bar{q}_i)|^2
\!+\!|\breve{x}^{(N)}(s)-\Phi(s)|^2\!+\!|\Phi(s)|^2\!+\!|\breve{x}_0(s)|^2\!+\!|\sigma(s)|^2\right)ds\\
\leq &\frac{M}{N_k}\mathbb{E}\int_0^t\mathbb{E}|\breve{x}^{(N)}(s)-\Phi(s)|^2ds+\frac{M}{N_k}.
\end{aligned}
\]
Therefore, from above analysis, we get from (\ref{ee3}) that
\[
\mathbb{E}\sup_{0\leq s\leq t}|\Delta_k(s)|^2\leq M\mathbb{E}\int_0^t\Big[|\Delta_k(s)|^2+|\breve{x}^{(N)}(s)-\Phi(s)|^2\Big]ds
+\frac{M}{N_k},
\]
and Gronwall's inequality yields that
\begin{equation}\label{minor estimate 1}
\mathbb{E}\sup_{0\leq s\leq t}|\Delta_k(s)|^2\leq M\mathbb{E}\int_0^t|\breve{x}^{(N)}(s)-\Phi(s)|^2ds+\frac{M}{N_k}.
\end{equation}
Since
\[
\begin{aligned}
\breve{x}^{(N)}(s)-\Phi(s)=&\sum_{k=1}^{K}\left[\pi_k^{(N)}\frac{1}{N_k}\sum_{i\in\mathcal{I}_k}\breve{x}_i(s)
-\pi_k\mathbb{E}(\alpha_k(s)|\mathcal{F}_s^{W_0})\right]\\
=&\sum_{k=1}^{K}\left[\pi_k^{(N)}\left(\frac{1}{N_k}\sum_{i\in\mathcal{I}_k}\breve{x}_i(s)
-\mathbb{E}(\alpha_k(s)|\mathcal{F}_s^{W_0})\right)+
\left(\pi_k^{(N)}-\pi_k\right)\mathbb{E}(\alpha_k(s)|\mathcal{F}_s^{W_0})\right]\\
=&\sum_{k=1}^{K}\pi_k^{(N)}\Delta_k(s)+\sum_{k=1}^{K}
\left(\pi_k^{(N)}-\pi_k\right)\mathbb{E}(\alpha_k(s)|\mathcal{F}_s^{W_0}),
\end{aligned}
\]
by using (\ref{bounded}),(\ref{minor estimate 1}) and $\pi_k^{(N)}=\frac{N_k}{N}\leq 1$ we obtain that for any $t\in[0,T]$,
\[
\begin{aligned}
\mathbb{E}\sup_{0\leq s\leq t}|\breve{x}^{(N)}(s)-\Phi(s)|^2
\leq & \mathbb{E}\sum_{k=1}^{K}\sup_{0\leq s\leq t}\pi_k^{(N)}|\Delta_k(s)|^2+M\varepsilon_{N}^2\\
\leq & M\mathbb{E}\int_0^t|\breve{x}^{(N)}(s)-\Phi(s)|^2ds+\frac{M}{N}+M\varepsilon_{N}^2.
\end{aligned}
\]
Finally, by using Gronwall's inequality, we complete the proof.
\end{proof}

\begin{lemma}\label{lemma for x and xi}
Under the assumptions of \emph{\textbf{(A1)-(A4)}}, we have
\[
\sup_{0 \leq i \leq N}\mathbb{E}\sup_{0\leq t\leq T}\Big|\breve{x}_{i}(t)-\bar{x}_{i}(t)\Big|^2
\leq M\Big(\frac{1}{N}+\varepsilon_{N}^2\Big).
\]
\end{lemma}
\begin{proof}
On the one hand, from both the first equation of (\ref{decentralized state}) and  (\ref{limiting state under cc condition}), we have
\[\label{ee4}
\left\{
\begin{aligned}
d(\breve{x}_{0}-\bar{x}_{0})&=\Big[A_0(\breve{x}_{0}-\bar{x}_{0})\!+\!F_0^1(\breve{x}^{(N)}
\!-\!\Phi)\Big]dt
\!+\!\Big[C_0(\breve{x}_{0}-\bar{x}_{0})\!+\!F_0^2(\breve{x}^{(N)}\!-\!\Phi)\Big]dW_i(t),\\
\breve{x}_{0}(0)-\bar{x}_{0}(0)&=0.
\end{aligned}
\right.
\]
The classical estimate for the SDE yields that
\[
\mathbb{E}\sup_{0\leq t\leq T}\Big|\breve{x}_{0}(t)-\bar{x}_{0}(t)\Big|^2\leq
ME\int_0^T\left|\breve{x}^{(N)}(s)-\Phi(s)\right|^2ds,
\]
where $M$ is a constant independent of $N$. Noticing  Lemma \ref{lemma for xN and m},
we obtain
\begin{equation}\label{ee5}
\mathbb{E}\sup_{0\leq t\leq T}\Big|\breve{x}_{0}(t)-\bar{x}_{0}(t)\Big|^2
\leq M\Big(\frac{1}{N}+\varepsilon_{N}^2\Big).
\end{equation}
On the other hand, from the second equation of (\ref{decentralized state}) and
the third equation of  (\ref{limiting state under cc condition}), we have that for $1\leq i\leq N$,
\[
\left\{
\begin{aligned}
&d(\breve{x}_{i}-\bar{x}_{i})=\Big[A_{\theta_i}(\breve{x}_{i}-\bar{x}_{i})\!+\!F_1(\breve{x}^{(N)}
\!-\!\Phi)\Big]dt\!+\!\Big[C(\breve{x}_{i}-\bar{x}_{i})\!+\!F_2(\breve{x}^{(N)}\!-\!\Phi)
+H(\breve{x}_{0}-\bar{x}_{0})\Big]dW_i(t),\\
&\breve{x}_{i}(0)-\bar{x}_{i}(0)=0.
\end{aligned}
\right.
\]
The classical estimate for the SDE yields that
\[
\mathbb{E}\sup_{0\leq t\leq T}\Big|\breve{x}_{i}(t)-\bar{x}_{i}(t)\Big|^2\leq
M\mathbb{E}\int_0^T\left(\left|\breve{x}^{(N)}(s)-\Phi(s)\right|^2
+\left|\breve{x}_0(s)-\bar{x}_{0}(s)\right|^2\right)ds.
\]
Noticing Lemma \ref{lemma for xN and m} and (\ref{ee5}), we obtain
$
\mathbb{E}\sup\limits_{0\leq t\leq T}\Big|\breve{x}_{i}(t)-\bar{x}_{i}(t)\Big|^2
\leq M\Big(\frac{1}{N}+\varepsilon_{N}^2\Big).
$
Thus, considering also (\ref{ee5}) we complete the proof.
\end{proof}
\begin{lemma}\label{first lemma for cost}
Under the assumptions of \emph{\textbf{(A1)-(A4)}}, for all $ 0\leq i\leq N$, we have
\[
 \Big|\mathcal{J}_i(\bar{u}_i, \bar{u}_{-i})-J_i(\bar{u}_i)\Big|=O\Big(\frac{1}{\sqrt{N}}+\varepsilon_N\Big).
\]
\end{lemma}
\begin{proof}
Let us first consider the major agent. Recall (\ref{major agent cost functional}), (\ref{major agent cost functional limit}) and (\ref{sum of mk}), we have
\begin{equation}\label{ee7}
\begin{aligned}
&\mathcal{J}_0(\bar{u}_0, \bar{u}_{-0})-J_0(\bar{u}_0)=\frac{1}{2}\mathbb{E}\bigg[\int_0^T\Bigg( \Big <Q_0\big(\breve{x}_0\!-\!\rho_0\breve{x}^{(N)}\big),\breve{x}_0\!-\!\rho_0\breve{x}^{(N)}\Big>
\!-\!\Big <Q_0\big(\bar{x}_0\!-\!\rho_0\Phi\big),\bar{x}_0\!-\!\rho_0\Phi\Big>\Bigg)dt\\
&\!+\!\Big<G_0\big(\breve{x}_0(T)\!-\!\rho_0\breve{x}^{(N)}(T)\big),\breve{x}_0(T)\!-\!\rho_0\breve{x}^{(N)}(T)\Big>
\!-\!\Big<G_0\big(\bar{x}_0(T)\!-\!\rho_0\Phi(T)\big),\bar{x}_0(T)\!-\!\rho_0\Phi(T)\Big>\bigg].
\end{aligned}
\end{equation}
From (\ref{bounded}), we have $\mathbb{E}\sup\limits_{0\leq t\leq T}\left|\bar{x}_0(t)\right|^2\leq M$ and
$\mathbb{E}\sup\limits_{0\leq t\leq T}\left|\alpha_i(t)\right|^2\leq M$,  for any $0\leq i\leq N$.
Noticing that $\Phi(t)=\sum_{k=1}^{K}\pi_k\mathbb{E}(\alpha_k(t)|\mathcal{F}_t^{W_0})$, it is not hard to check that
$
\mathbb{E}\sup\limits_{0\leq t\leq T}\left|\Phi(t)\right|^2\leq M.
$
Now, from such estimates and Lemma \ref{lemma for xN and m}, Lemma \ref{lemma for x and xi} as well as
\[
\begin{aligned}
&\langle Q_0(a-b), a-b\rangle-\langle Q_0(c-d), c-d\rangle\\
=&\langle Q_0(a-b-(c-d)), a-b-(c-d)\rangle+2\langle Q_0(a-b-(c-d)), c-d\rangle,
\end{aligned}
\]
we have
\[
\begin{aligned}
&\Bigg|\mathbb{E}\bigg [\int_0^T\Bigg( \Big< Q_0\big(\breve{x}_0\!-\!\rho_0\breve{x}^{(N)}\big),\breve{x}_0\!-\!\rho_0\breve{x}^{(N)}\Big>\!-\!\Big< Q_0\big(\bar{x}_0\!-\!\rho_0\Phi\big),\bar{x}_0\!-\!\rho_0\Phi\Big>\Bigg)dt\Bigg|\\
\leq& M\int_0^T\mathbb{E}\left|\breve{x}_0-\rho_0\breve{x}^{(N)}-(\bar{x}_0-\rho_0\Phi)\right|^2dt
+M\int_0^T\mathbb{E}\left|\breve{x}_0-\rho_0\breve{x}^{(N)}
-(\bar{x}_0-\rho_0\Phi)\right|
\left|\bar{x}_0-\rho_0\Phi\right|dt\\
\leq &M\int_0^T\mathbb{E}\left|\breve{x}_0-\bar{x}_0\right|^2dt
+M\int_0^T\mathbb{E}\left|\breve{x}^{(N)}-\Phi\right|^2dt\\
&\qquad\qquad+M\int_0^T\left(\mathbb{E}
\left|\breve{x}_0-\rho_0\breve{x}^{(N)}-(\bar{x}_0-\rho_0\Phi)\right|^2\right)^{\frac{1}{2}}
\left(\mathbb{E}\left|\bar{x}_0-\rho_0\Phi\right|^2\right)^{\frac{1}{2}}dt\\
\leq &M\int_0^T\mathbb{E}\left|\breve{x}_0-\bar{x}_0\right|^2dt
+M\int_0^T\mathbb{E}\left|\breve{x}^{(N)}-\Phi\right|^2dt+M\int_0^T\left(\mathbb{E}\left|\breve{x}_0-\bar{x}_0\right|^2
+\mathbb{E}\left|\breve{x}^{(N)}-\Phi\right|^2\right)^{\frac{1}{2}}
dt\\
=& O\left(\frac{1}{\sqrt{N}}+\varepsilon_N\right).
\end{aligned}
\]
Similar argument allows us to show that
\[
\begin{aligned}
&\left|\mathbb{E}\bigg [\Big< G_0\big(\breve{x}_0(T)\!-\!\rho_0\breve{x}^{(N)}(T)\big),\breve{x}_0(T)\!-\!\rho_0\breve{x}^{(N)}(T)\Big>
\!-\!\Big< G_0\big(\bar{x}_0(T)\!-\!\rho_0\Phi(T)\big),\bar{x}_0(T)\!-\!\rho_0\Phi(T)\Big>\bigg]\right|\\
&=O\left(\frac{1}{\sqrt{N}}+\varepsilon_N\right).
\end{aligned}
\]
Thus, the proof for the major agent is completed by noticing (\ref{ee7}). Let us now focus on the minor agents, for $1\leq i\leq N$, recalling (\ref{minor agent cost functional}), (\ref{minor agent cost functional limit}) and (\ref{sum of mk}), we have
\[
\begin{aligned}
   \mathcal {J}_i(\bar{u}_i,\bar{u}_{-i})=&\frac{1}{2}\mathbb{E}
    \bigg [ \int_0^T \left(\Big <Q\big(\breve{x}_i-\rho\breve{x}^{(N)}-(1-\rho)\breve{x}_0\big),\breve{x}_i-\rho\breve{x}^{(N)}
    -(1-\rho)\breve{x}_0\Big>\!+\!\big<R_{\theta_i}\bar{u}_i,\bar{u}_i\big>\right)dt\\
    &\!+\!\Big<G\big(\breve{x}_i(T)\!-\!\rho\breve{x}^{(N)}(T)\!-\!(1\!-\!\rho)\breve{x}_0(T)\big),
    \breve{x}_i(T)\!-\!\rho\breve{x}^{(N)}(T)
    \!-\!(1\!-\!\rho)\breve{x}_0(T)\Big>\bigg]
\end{aligned}
\]
and
\[
\begin{aligned}
    {J}_i(\bar{u}_i)=&\frac{1}{2}\mathbb{E}
    \bigg [ \int_0^T \left(\Big <Q\big(\bar{x}_i-\rho \Phi-(1-\rho)\bar{x}_0\big),\bar{x}_i-\rho \Phi-(1-\rho)\bar{x}_0\Big>dt+\big<R_{\theta_i}\bar{u}_i,\bar{u}_i\big>\right)dt\\
    &+\Big<G\big(\bar{x}_i(T)-\rho \Phi(T)-(1-\rho)\bar{x}_0(T)\big),\bar{x}_i(T)-\rho \Phi(T)-(1-\rho)\bar{x}_0(T)\Big>\bigg].
\end{aligned}
\]
From (\ref{bounded}), we have $\mathbb{E}\sup\limits_{0\leq t\leq T}\left|\bar{x}_i(t)\right|^2\leq M$.  Using such estimate, $
\mathbb{E}\sup\limits_{0\leq t\leq T}\left|\Phi(t)\right|^2\leq M$, Lemmas \ref{lemma for xN and m}, \ref{lemma for x and xi}, similar to the major agent, it follows that
$
\Big|\mathcal{J}_i(\bar{u}_i, \bar{u}_{-i})-J_i(\bar{u}_i)\Big|=O\Big(\frac{1}{\sqrt{N}}+\varepsilon_N\Big).
$
\end{proof}

\medskip
\subsection{Major agent's perturbation}
In this subsection, we will prove that the control strategies set $(\bar{u}_1,\bar{u}_2,\ldots,\bar{u}_N)$ is an $\varepsilon$-Nash equilibrium of \textbf{Problem (CC)} for the major agent, i.e. $\exists\quad \varepsilon=\varepsilon(N)\geq0$, $\lim\limits_{N\rightarrow\infty}\varepsilon(N)=0$ s.t
\[
\mathcal J_0(\bar u_0(\cdot),\bar u_{-0}(\cdot))\leq \mathcal J_0(u_0(\cdot),\bar{u}_{-0}(\cdot))+\varepsilon, \quad
\text{ for any } u_0\in\mathcal{U}_{ad}^{0}.
\]

Let us consider that the major agent $\mathcal{A}_0$ uses an alternative strategy $u_0$ and each minor agent $\mathcal{A}_i$ uses the control $\bar{u}_i=\varphi_{\theta_i}(\bar{p}_i,\bar{q}_i)$, where $(\bar{p}_i,\bar{q}_i)$ are solved from (\ref{limiting state under cc condition}). Then the realized state system with major agent's perturbation is, for $1\leq i\leq N$,
  \begin{equation}\label{perturbed major}
\left\{
\begin{aligned}dy_{0}&=\left[A_0y_{0}\!+\!B_0u_0\!+\!F_0^1y^{(N)}\!+\!b_0\right]dt
        \!+\!\left[C_0y_{0}\!+\!D_0u_0
      +F_0^2y^{(N)}\!+\!\sigma_0\right]dW_0(t),\\
      dy_{i}&=\left[A_{\theta_i}y_{i}\!+\!B\varphi_{\theta_i}(\bar{p}_i,\bar{q}_i)\!+\!F_1y^{(N)}
      \!+\!b_0\right]dt\\
        &\qquad\!+\!\left[Cy_{i}\!+\!D_{\theta_i}\varphi_{\theta_i}(\bar{p}_i,\bar{q}_i)
      +F_2y^{(N)}\!+\!Hy_0\!+\!\sigma_0\right]dW_i(t),\\
      y_{0}(0)&=x_0, \quad y_{i}(0)=x,
    \end{aligned}
    \right.
\end{equation}
where $y^{(N)}=\frac{1}{N}\sum_{i=1}^{N}y_{i}$. The well-posedness of above SDEs system is easy to obtain.
To prove   $(\bar{u}_0,\bar{u}_1,\ldots,\bar{u}_N)$ is an $\varepsilon$-Nash equilibrium for the major agent, we need to show that for possible alternative control $u_0$,
$
\inf_{u_0\in\mathcal{U}_{ad}^0}\mathcal{J}_0(u_0,\bar{u}_{-0})\ge\mathcal{J}_0(\bar{u}_0,\bar{u}_{-0})-\varepsilon.
$
Then we only need to consider the perturbation $u_0\in\mathcal{U}_{ad}^{0}$
such that $\mathcal{J}_0(u_0,\bar{u}_{-0})\leq\mathcal{J}_0(\bar{u}_0,\bar{u}_{-0})$. Thus, noticing $Q_0\ge0$ and $G_0\ge0$, from Lemma \ref{first lemma for cost}, we have
\[
\mathbb{E}\int_0^T\langle R_0u_0(t),u_0(t)\rangle dt\leq
\mathcal{J}_0(u_0,\bar{u}_{-0})\leq\mathcal{J}_0(\bar{u}_0,\bar{u}_{-0})
\leq J_0(\bar{u}_0)+O(\frac{1}{\sqrt{N}}+\varepsilon_N),
\]
which implies that (noting \textbf{(A4)}),
$
\mathbb{E}\int_0^T|u_0(t)|^2dt\leq M,
$
where $M$ is a constant independent of $N$. Then similar to Lemma \ref{lemma for xN},
we can show that
\begin{equation}\label{boundedness of yi}
\sup_{0\leq i\leq N}\mathbb{E}\sup_{0\leq t\leq T}|y_i(t)|^2\leq M.
\end{equation}
\begin{lemma}\label{perturbed major lemma 1}
Under the assumptions of \emph{\textbf{(A1)-(A4)}}, we have
\[
\mathbb{E}\sup_{0\leq t\leq T}\Big|y^{(N)}(t)-\Phi(t)\Big|^2=O\Big(\frac{1}{N}+\varepsilon_N^2\Big).
\]
\end{lemma}
\begin{proof}
For each fixed $1\leq k\leq K$, we consider the $k$-type minor agents. We denote $y^{(k)}:=\frac{1}{N_k}\sum_{i\in\mathcal{I}_k}y_i$. 
As there are $N_k$  minor agents of the $k$-type, let us add up their states and then divided by $N_k$, it follows that
\[
   \begin{aligned}
      dy^{(k)}&=\left[A_{k}y^{(k)}\!+\!\frac{B}{N_k}\sum_{i\in\mathcal{I}_k}\varphi_k(\bar{p}_i,\bar{q}_i)
      \!+\!F_1y^{(N)}\!+\!b_0\right]dt\\
        &\qquad\!+\!\frac{1}{N_k}\sum_{i\in\mathcal{I}_k}\left[Cy_{i}\!+\!D_{k}\varphi_k(\bar{p}_i,\bar{q}_i)
      +F_2y^{(N)}\!+\!Hy_0\!+\!\sigma_0\right]dW_i(t),\qquad y^{(k)}(0)=x.
    \end{aligned}
\]
Recall (\ref{mk}) and if we denote $\tilde{\Delta}_k(t):=y^{(k)}(t)-m_k(t)$, it follows that
\[
\begin{aligned}
d\tilde{\Delta}_k&=\Bigg[A_k\tilde{\Delta}_k+\frac{B}{N_k}\left(\sum_{i\in\mathcal{I}_k}\varphi_k(\bar{p}_i,\bar{q}_i)
-\mathbb{E}(\varphi_k(\bar{p}_i,\bar{q}_i)|\mathcal{F}_t^{W_0})\right)+F_1(y^{(N)}-\Phi)\Bigg]dt\\
&\quad+\frac{1}{N_k}\sum_{i\in\mathcal{I}_k}\left[Cy_{i}\!+\!D_{k}\varphi_k(\bar{p}_i,\bar{q}_i)
      +F_2y^{(N)}\!+\!Hy_0\!+\!\sigma_0\right]dW_i(t),\qquad
\tilde{\Delta}(0)=0.
\end{aligned}
\]
Similar to the argument in the proof of Lemma \ref{lemma for xN and m},  we can show that
\begin{equation}\label{major ee3}
\begin{aligned}
&\mathbb{E}\sup_{0\leq s\leq t}|\tilde{\Delta}_k(s)|^2\leq M\mathbb{E}\int_0^t\Big[|\tilde{\Delta}_k(s)|^2+|y^{(N)}(s)-\Phi(s)|^2\Big]ds\\
&\qquad+M\mathbb{E}\int_0^t\Bigg|\frac{1}{N_k}\sum_{i\in\mathcal{I}_k}\varphi_k(\bar{p}_i(s),\bar{q}_i(s))
-\mathbb{E}(\varphi_k(\bar{p}_i(s),\bar{q}_i(s))|\mathcal{F}_s^{W_0})\Bigg|^2ds\\
&\qquad+\frac{M}{N_k^2}\mathbb{E}\sum_{i\in\mathcal{I}_k}\int_0^t\left|Cy_{i}
\!+\!D_{k}\varphi_k(\bar{p}_i,\bar{q}_i)
      +F_2(y^{(N)}-\Phi)\!+\!F_2\Phi\!+\!Hy_0\!+\!\sigma_0\right|^2ds,
\end{aligned}
\end{equation}
and
\[
\begin{aligned}
\mathbb{E}\int_0^t\Bigg|\frac{1}{N_k}\sum_{i\in\mathcal{I}_k}\varphi_k(\bar{p}_i(s),\bar{q}_i(s))
-\mathbb{E}(\varphi_k(\bar{p}_i(s),\bar{q}_i(s))|\mathcal{F}_s^{W_0})\Bigg|^2ds
\leq \frac{M}{N_k}.
\end{aligned}
\]
Using (\ref{bounded}), (\ref{boundedness of yi}) and the fact that $(y_i(s),\bar{p}_i(s),\bar{q}_i(s))$, $i\in\mathcal{I}_k$, are identically distributed, we have
\[
\begin{aligned}
&\frac{M}{N^2_k}\sum_{i\in\mathcal{I}_k}\mathbb{E}\int_0^t\left|Cy_{i}\!+\!D_{k}\varphi_k(\bar{p}_i,\bar{q}_i)
      +F^2(y^{(N)}-\Phi)\!+\!F_2\Phi\!+\!Hy_0\!+\!\sigma_0\right|^2ds\\
=&\frac{M}{N_k}\mathbb{E}\int_0^t\mathbb{E}|y^{(N)}(s)-\Phi(s)|^2ds+\frac{M}{N_k}.
\end{aligned}
\]
Therefore, we get from (\ref{major ee3}) that
\[
\mathbb{E}\sup_{0\leq s\leq t}|\tilde{\Delta}_k(s)|^2\leq M\mathbb{E}\int_0^t\Big[|\tilde{\Delta}_k(s)|^2+|y^{(N)}(s)-\Phi(s)|^2\Big]ds
+\frac{M}{N_k},
\]
and Gronwall's inequality yields that
\[
\mathbb{E}\sup_{0\leq s\leq t}|\tilde{\Delta}_k(s)|^2\leq
M\mathbb{E}\int_0^t\Big[|y^{(N)}(s)-\Phi(s)|^2\Big]ds.
\]
Similar to the proof of  Lemma \ref{lemma for xN and m} again, and using \textbf{(A1)}, we have for any $t\in[0,T]$,
\[
\begin{aligned}
\mathbb{E}\sup_{0\leq s\leq t}|y^{(N)}(s)-\Phi(s)|^2
\leq & \mathbb{E}\sum_{k=1}^{K}\sup_{0\leq s\leq t}\pi_k^{(N)}|\Delta_k(s)|^2+M\varepsilon_{N}^2\\
\leq & M\mathbb{E}\int_0^t\Big[|y^{(N)}(s)-\Phi(s)|^2\Big]ds+\frac{M}{N}+M\varepsilon_{N}^2.
\end{aligned}
\]
Finally, Gronwall's inequality allows us to complete the proof.
\end{proof}
\medskip

Now, we introduce the following system of the decentralized limiting state  with the major's perturbation control, for $1\leq i\leq N$:
  \begin{equation}\label{perturbed major limit}
\left\{
\begin{aligned}
d\bar{y}_{0}&=\left[A_0\bar{y}_{0}\!+\!B_0u_0\!+\!F_0^1\Phi\!+\!b_0\right]dt
        \!+\!\left[C_0\bar{y}_{0}\!+\!D_0u_0
      +F_0^2\Phi\!+\!\sigma_0\right]dW_0(t)\\
      d\bar{y}_{i}&=\left[A_{\theta_i}\bar{y}_{i}\!+\!B\varphi_{\theta_i}(\bar{p}_i,\bar{q}_i)
      \!+\!F_1\Phi
      \!+\!b_0\right]dt\\
        &\qquad\!+\!\left[C\bar{y}_{i}\!+\!D_{\theta_i}\varphi_{\theta_i}(\bar{p}_i,\bar{q}_i)
      +F_2\Phi\!+\!H\bar{y}_0\!+\!\sigma_0\right]dW_i(t)\\
      \bar{y}_{0}(0)&=x_0, \quad \bar{y}_{i}(0)=x.
    \end{aligned}
    \right.
\end{equation}
\begin{lemma}\label{perturbed major lemma 2}
Under the assumptions of \emph{\textbf{(A1)-(A4)}}, we have
\[\label{perturbed major lemma 2 estimate}
\sup_{0\leq i\leq N}\mathbb{E}\sup_{0\leq t\leq T}\Big|y_{i}(t)-\bar{y}_{i}(t)\Big|^2=O\Big(\frac{1}{N}+\varepsilon_N^2\Big).
\]
\end{lemma}
\begin{proof}
From both the first equation of (\ref{perturbed major}) and (\ref{perturbed major limit}), we obtain
\[
\left\{
\begin{aligned}
&d(y_0-\bar{y}_0)=\left[A(y_0-\bar{y}_0)+F_0^1(y^{(N)}-\Phi)\right]dt
+\left[C(y_0-\bar{y}_0)+F_0^2(y^{(N)}-\Phi)\right]dW_0(t),\\
&y_0(0)-\bar{y}_0(0)=0.
\end{aligned}
\right.
\]
With the help of classical estimates of SDE and  Lemma \ref{perturbed major lemma 1},
it is easy to obtain
\begin{equation}\label{perturbed major agent difference}
\mathbb{E}\sup_{0\leq t\leq T}\Big|y_0(t)-\bar{y}_0(t)\Big|^2=O\Big(\frac{1}{N}+\varepsilon_N^2\Big).
\end{equation}
Now, for any $1\leq i\leq N$, from both the second equation of (\ref{perturbed major}) and (\ref{perturbed major limit}), we get
\[
\begin{aligned}
d(y_i-\bar{y}_i)&=\left[A_{\theta_i}(y_i-\bar{y}_i)+F_1(y^{(N)}-\Phi)\right]dt\\
&\qquad+\left[C(y_i-\bar{y}_i)+F_2(y^{(N)}-\Phi)+H(y_0-\bar{y}_0)\right]dW_i(t),\qquad
y_i(0)-\bar{y}_i(0)=0.
\end{aligned}
\]
The classical estimates of SDE, Lemma \ref{perturbed major lemma 1} and (\ref{perturbed major agent difference}) allow us to complete the proof.
\end{proof}
\begin{lemma}\label{perturbed major lemma 3}
Under \emph{\textbf{(A1)-(A4)}}, for the major agent's perturbation control $u_{0}$, we have
\[
\Big|\mathcal{J}_0({u}_0, \bar{u}_{-0})-J_0({u}_0)\Big|=O\Big(\frac{1}{\sqrt{N}}+\varepsilon_N\Big).
\]
\end{lemma}
\begin{proof}
Recall (\ref{major agent cost functional}), (\ref{major agent cost functional limit}) and (\ref{sum of mk}), we have
\begin{equation}\label{major agent ee12}
\begin{aligned}
&\mathcal{J}_0({u}_0, \bar{u}_{-0})-J_0({u}_0)=\frac{1}{2}\mathbb{E}\bigg [ \int_0^T\Big( \Big <Q_0\big(y_0-\rho_0y^{(N)}\big),y_0-\rho_0y^{(N)}\Big>\!-\!\Big <Q_0\big(\bar{y}_0
-\rho_0\Phi\big),\bar{y}_0-\rho_0\Phi\Big>\Big)dt\\
&\!+\!\Big<G_0\big(y_0(T)\!-\!\rho_0y^{(N)}(T)\big),y_0(T)\!-\!\rho_0y^{(N)}(T)\Big>\!-\!\Big<G_0\big(\bar{y}_0(T)
\!-\!\rho_0\Phi(T)\big),\bar{y}_0(T)\!-\!\rho_0\Phi(T)\Big>\bigg].
\end{aligned}
\end{equation}
Similar to Lemma \ref{first lemma for cost}, by using Lemmas \ref{perturbed major lemma 1}, \ref{perturbed major lemma 2} and $\mathbb{E}\left(\left|\bar{y}_0(t)\right|^2+\left|\Phi(t)\right|^2\right)\leq M$, we have
\[
\begin{aligned}
&\Bigg|\mathbb{E} \int_0^T\Big( \Big <Q_0\big(y_0-\rho_0y^{(N)}\big),y_0-\rho_0y^{(N)}\Big>-\Big <Q_0\big(\bar{y}_0
-\rho_0\Phi\big),\bar{y}_0-\rho_0\Phi\Big>\Big)dt\Bigg|\\
\leq &M\int_0^T\mathbb{E}\left|y_0-\bar{y}_0\right|^2dt
+M\int_0^T\mathbb{E}\left|y^{(N)}-\Phi\right|^2dt+M\int_0^T\left(\mathbb{E}\left|y_0-\bar{y}_0\right|^2
+\mathbb{E}\left|y^{(N)}-\Phi\right|^2\right)^{\frac{1}{2}}dt\\
=&O\Big(\frac{1}{\sqrt{N}}+\varepsilon_N\Big),
\end{aligned}
\]
and
\[
\begin{aligned}
&\Bigg|\mathbb{E}\bigg[\Big<G_0\big(y_0(T)-\rho_0y^{(N)}(T)\big),y_0(T)-\rho_0y^{(N)}(T)\Big>\\
&\qquad\qquad-\Big<G_0\big(\bar{y}_0(T)
-\rho_0\Phi(T)\big),\bar{y}_0(T)-\rho_0\Phi(T)\Big>\bigg]\Bigg|
=O\Big(\frac{1}{\sqrt{N}}+\varepsilon_N\Big).
\end{aligned}
\]
The proof is completed by noticing (\ref{major agent ee12}).
\end{proof}
\begin{theorem}\label{major Nash equilibrium theorem}
Under the assumptions of \emph{\textbf{(A1)-(A4)}},
then the strategies set $(\bar{u}_0,\bar{u}_1,\cdots,\bar{u}_N)$ is an $\varepsilon$-Nash equilibrium of \emph{\textbf{Problem (CC)}} for the major agent.
\end{theorem}
\begin{proof}
Combining Lemma \ref{first lemma for cost} and Lemma \ref{perturbed major lemma 3}, we have
\[
\begin{aligned}
\mathcal{J}_0(\bar{u}_0, \bar{u}_{-0})
\!\leq\!J_0(\bar{u}_0)\!+\!O\Big(\frac{1}{\sqrt{N}}\!+\!\varepsilon_N\Big)
\leq J_0({u}_0)\!+\!O\Big(\frac{1}{\sqrt{N}}+\varepsilon_N\Big)
\!\leq\!\mathcal{J}_0({u}_0, \bar{u}_{-0})\!+\!O\Big(\frac{1}{\sqrt{N}}\!+\!\varepsilon_N\Big),
\end{aligned}
\]
where the second inequality comes from the fact that $J_0(\bar{u}_0)=\inf_{u_0\in\mathcal{U}_{ad}^0}J_0(u_0)$.
Consequently, Theorem \ref{major Nash equilibrium theorem} holds with $\varepsilon=O\Big(\frac{1}{\sqrt{N}}+\varepsilon_N\Big)$.
\end{proof}
\subsection{Minor agent's perturbation}
Now, let us consider the following case: a given minor agent $\mathcal{A}_i$ uses an alternative strategy $u_i\in\mathcal{U}_{ad}^{i},$ the major agent uses $\bar{u}_0=\varphi_{0}(\bar{p}_0,\bar{q}_0)$ while other minor agents $\mathcal{A}_j$ use the control $\bar{u}_j=\varphi_{\theta_j}(\bar{p}_j,\bar{q}_j)$, $j\neq i$, $1\leq j\leq N$, where $(\bar{p}_j,\bar{q}_j)$, $0\leq j\leq N$, $j\neq i$, are solved from (\ref{limiting state under cc condition}). Then the realized state system with the minor agent's perturbation is, for $1\leq j\leq N$, $j\neq i$,
  \begin{equation}\label{perturbed minor}
\left\{
\begin{aligned}
&dl_{0}=\left[A_0l_{0}\!+\!B_0\varphi_{0}(\bar{p}_0,\bar{q}_0)\!+\!F_0^1l^{(N)}\!+\!b_0\right]dt
        \!+\!\left[C_0l_{0}\!+\!D_0\varphi_{0}(\bar{p}_0,\bar{q}_0)
      +F_0^2l^{(N)}\!+\!\sigma_0\right]dW_0(t)\\
      &dl_{i}=\left[A_{\theta_i}l_{i}\!+\!Bu_i\!+\!F_1l^{(N)}
      \!+\!b_0\right]dt\!+\!\left[Cl_{i}\!+\!D_{\theta_i}u_i
      +F_2l^{(N)}\!+\!Hl_0\!+\!\sigma_0\right]dW_i(t),\\
      &dl_{j}=\left[A_{\theta_j}l_{j}\!+\!B\varphi_{\theta_j}(\bar{p}_j,\bar{q}_j)\!+\!F_1l^{(N)}
      \!+\!b_0\right]dt\!+\!\left[Cl_{j}\!+\!D_{\theta_j}\varphi_{\theta_j}(\bar{p}_j,\bar{q}_j)
      +F_2l^{(N)}\!+\!Hl_0\!+\!\sigma_0\right]\!dW_i(t),\\
      &l_{0}(0)=x_0, \quad l_{i}(0)=l_{j}(0)=x,
    \end{aligned}
    \right.
\end{equation}
where $l^{(N)}=\frac{1}{N}\sum_{i=1}^{N}l^{i}$. The well-posedness of above SDEs system is easily to obtain. Similar to the argument of major agent, to prove   $(\bar{u}_0,\bar{u}_1,\ldots,\bar{u}_N)$ is an $\varepsilon$-Nash equilibrium for the minor agent, noticing $Q\ge0$, $G\ge0$, $R_{\theta_i}>0$ and Lemma \ref{first lemma for cost}, we only need to consider the perturbation $u_i\in\mathcal{U}_{ad}^{i}$ satisfying
\begin{equation}\label{boundedness of minor control}
\mathbb{E}\int_0^T|u_i(t)|^2dt\leq M,
\end{equation}
where $M$ is a constant independent of $N$. Similar to Lemma \ref{lemma for xN},
we can show that
\begin{equation}\label{boundedness of li}
\sup_{0\leq i\leq N}\mathbb{E}\sup_{0\leq t\leq T}|l_i(t)|^2\leq M.
\end{equation}
We first present the following lemma
\begin{lemma}\label{perturbed minor lemma 1}
Under the assumptions of \emph{\textbf{(A1)-(A4)}}, we have
\[
\mathbb{E}\sup_{0\leq t\leq T}\Big|l^{(N)}(t)-\Phi(t)\Big|^2=O\Big(\frac{1}{N}+\varepsilon_N^2\Big).
\]
\end{lemma}
\begin{proof}
We know that for each fixed $i$, there exists a unique $1\leq \bar{k}\leq K$, such that $i\in\mathcal{I}_{\bar{k}}$. Let us denote $l^{(k)}:=\frac{1}{N_k}\sum_{i\in\mathcal{I}_k}l_i$, $1\leq k\leq K$. We first consider the $k$-type minor agents, where $k\neq \bar{k}$.  Adding up their states and then divided by $N_k$, we have for $k\neq \bar{k}$,
\[
   \begin{aligned}
      dl^{(k)}&=\left[A_{k}l^{(k)}\!+\!\frac{B}{N_k}\sum_{i\in\mathcal{I}_k}\varphi_k(\bar{p}_i,\bar{q}_i)
      \!+\!F_1l^{(N)}\!+\!b_0\right]dt\\
        &\qquad\!+\!\frac{1}{N_k}\sum_{i\in\mathcal{I}_k}\left[Cl_{i}\!+\!D_{k}\varphi_k(\bar{p}_i,\bar{q}_i)
      +F_2l^{(N)}\!+\!Hl_0\!+\!\sigma_0\right]dW_i(t),\qquad
     l^{(k)}(0)=x.
    \end{aligned}
\]
Similar to the proof of Lemma \ref{lemma for xN and m}, for $m_k=\mathbb{E}(\alpha_k(t)|\mathcal{F}_t^{W_0})$, we have
\begin{equation}\label{minor k other type sum}
\begin{aligned}
\mathbb{E}\sup_{0\leq s\leq t}|l^{(k)}(s)-m_k(s)|^2
\leq  M\mathbb{E}\int_0^t\Big[|l^{(N)}(s)-\Phi(s)|^2\Big]ds
+\frac{M}{N_k}.
 \end{aligned}
\end{equation}
Now let us focus on the $\bar{k}$-type minor agents, we have
\[
   \begin{aligned}
      dl^{(\bar{k})}&=\left[A_{{\bar{k}}}l^{({\bar{k}})}\!+\!\frac{B}{N_{\bar{k}}}u_i
      \!+\!\frac{B}{N_{\bar{k}}}\sum_{j\in\mathcal{I}_{\bar{k}},j\neq i}\varphi_{\bar{k}}(\bar{p}_j,\bar{q}_j)
      \!+\!F_1l^{(N)}\!+\!b_0\right]dt\\
        &\qquad\!+\!\frac{1}{N_{\bar{k}}}\sum_{j\in\mathcal{I}_{\bar{k}}}
        \left[Cl_{j}
      +F_2l^{(N)}\!+\!Hl_0\!+\!\sigma_0\right]dW_j(t)\\
      &\qquad\!+\!\frac{1}{N_{\bar{k}}}D_{{\bar{k}}}u_idW_i(t)
      \!+\!\frac{1}{N_{\bar{k}}}\sum_{j\in\mathcal{I}_{\bar{k}},j\neq i}D_{{\bar{k}}}\varphi_{\bar{k}}(\bar{p}_j,\bar{q}_j)
      dW_j(t),\qquad
     l^{({\bar{k}})}(0)=x.
    \end{aligned}
\]
Recalling (\ref{mk})
and if we denote  $\Xi:=l^{(\bar{k})}-m_{\bar{k}}$,  it follows that
\[
\begin{aligned}
d\Xi&=\Bigg[A\Xi\!+\!F_1(l^{(N)}-\Phi)\!+\!\frac{1}{N_{\bar{k}}}Bu_i
\!+\!\frac{B}{N_{\bar{k}}}\left(\sum_{j\in\mathcal{I}_{\bar{k}},j\neq i}
        \varphi_{\bar{k}}(\bar{p}_j,\bar{q}_j)
        -\mathbb{E}(\varphi_{\bar{k}}(\bar{p}_i,\bar{q}_i)|\mathcal{F}_t^{W_0})\right)\Bigg]dt\\
         &\qquad\!+\!\frac{1}{N_{\bar{k}}}\sum_{j\in\mathcal{I}_{\bar{k}}}
        \left[Cl_{j}
      +F_2l^{(N)}\!+\!Hl_0\!+\!\sigma_0\right]dW_j(t),\\
      &\qquad\!+\!\frac{1}{N_{\bar{k}}}D_{{\bar{k}}}u_idW_i(t)
      \!+\!\frac{1}{N_{\bar{k}}}\sum_{j\in\mathcal{I}_{\bar{k}},j\neq i}D_{{\bar{k}}}\varphi_{\bar{k}}(\bar{p}_j,\bar{q}_j)
      dW_j(t),\qquad
\Pi(0)=0.
\end{aligned}
\]
By Cauchy-Schwartz inequality and BDG inequality,
we obtain that for any $t\in[0,T]$,
\begin{equation}\label{minor ee11}
\begin{aligned}
\mathbb{E}\sup_{0\leq s\leq t}|\Xi(s)|^2&\leq M\mathbb{E}\int_0^t\left(|\Xi(s)|^2\!+\!\frac{1}{N_{\bar{k}}^2}|u_i(s)|^2\!+\!|l^{(N)}(s)-\Phi(s)|^2
\right)ds\\
&\!+\!M\mathbb{E}\int_0^t\left|\frac{1}{N_{\bar{k}}}\sum_{j\in\mathcal{I}_{\bar{k}},j\neq i}
        \varphi_{\bar{k}}(\bar{p}_j,\bar{q}_j)
        -\mathbb{E}(\varphi_{\bar{k}}(\bar{p}_i,\bar{q}_i)|\mathcal{F}_t^{W_0})\right|^2\Big]ds\\
&\!+\!\frac{M}{N_{\bar{k}}^2}\mathbb{E}\sum_{j\in\mathcal{I}_{\bar{k}}}\int_0^t\left|F_2(l^{(N)}(s)-\Phi(s))
\!+\!F_2\Phi(s)\!+\!
Cl_j(s)\!+\!Hl_0(s)\!+\!\sigma(s)\right|^2ds\\
&\!+\!\frac{M}{N_{\bar{k}}^2}\mathbb{E}\int_0^t|u_i(s)|^2ds
\!+\!\frac{M}{N_{\bar{k}}^2}\mathbb{E}\sum_{j\in\mathcal{I}_{\bar{k}},j\neq i}\int_0^t|\varphi_{\bar{k}}(\bar{p}_j(s),\bar{q}_j(s))|^2ds.
\end{aligned}
\end{equation}
On the one hand, since  for each fixed $s\in[0,T]$, under the conditional expectation $\mathbb{E}(\cdot|\mathcal{F}_s^{W_0})$, the processes $(\bar{p}_i(s),\bar{q}_i(s))$, $i\in\mathcal{I}_{\bar{k}}$, are independent identically distributed. If we denote $\mu(s)=\mathbb{E}(\varphi_{\bar{k}}(\bar{p}_i(s),\bar{q}_i(s))|\mathcal{F}_s^{W_0}))$, then $\mu$ does not depend on $i.$ Moreover,
\[
\begin{aligned}
&\mathbb{E}\left|\frac{1}{N_{\bar{k}}}\sum_{j\in\mathcal{I}_{\bar{k}},j\neq i}
        \varphi_{\bar{k}}(\bar{p}_j(s),\bar{q}_j(s))
-\mu(s)\right|^2\\
\leq&2\mathbb{E}\left|\frac{1}{N_{\bar{k}}}\sum_{j\in\mathcal{I}_{\bar{k}},j\neq i}
        \varphi_{\bar{k}}(\bar{p}_j(s),\bar{q}_j(s))
-\frac{N_{\bar{k}}-1}{N_{\bar{k}}}\mu(s)\right|^2
+2\mathbb{E}\left|\frac{1}{N_{\bar{k}}}\mu(s)\right|^2\\
=&2\frac{(N_{\bar{k}}-1)^2}{N_{\bar{k}}^2}\mathbb{E}\left|\frac{1}{N_{\bar{k}}-1}\sum_{j\in\mathcal{I}_{\bar{k}},j\neq i}\varphi_{\bar{k}}(\bar{p}_j(s),\bar{q}_j(s))
-\mu(s)\right|^2+\frac{2}{N_{\bar{k}}^2}\mathbb{E}|\mu(s)|^2.
\end{aligned}
\]
Then, due to (\ref{bounded}) and the fact that $(\bar{p}_i(s),\bar{q}_i(s))$, $i\in\mathcal{I}_{\bar{k}}$, are independent identically distributed under $\mathbb{E}(\cdot|\mathcal{F}_s^{W_0})$, similar to the proof of Lemma \ref{lemma for xN and m}, we can obtain
\[
\begin{aligned}
&\int_0^t\mathbb{E}\left|\frac{1}{N_{\bar{k}}}\sum_{j\in\mathcal{I}_{\bar{k}},j\neq i}
        \varphi_{\bar{k}}(\bar{p}_j(s),\bar{q}_j(s))
-\mu(s)\right|^2ds\\
\leq & 2\frac{(N_{\bar{k}}-1)^2}{N_{\bar{k}}^2}\int_0^t\mathbb{E}
\left|\frac{1}{N_{\bar{k}}-1}\sum_{j\in\mathcal{I}_{\bar{k}},j\neq i}\varphi_{\bar{k}}(\bar{p}_j(s),\bar{q}_j(s))
-\mu(s)\right|^2ds
+\frac{2}{N_{\bar{k}}^2}\int_0^t\mathbb{E}|\mu(s)|^2ds\\
=& 2\frac{N_{\bar{k}}-1}{N_{\bar{k}}^2}\int_0^t\mathbb{E}\left|\varphi_{\bar{k}}(\bar{p}_j(s),\bar{q}_j(s))
\!-\!\mu(s)\right|^2ds+\frac{2}{N_{\bar{k}}^2}\int_0^t\mathbb{E}|\mu(s)|^2ds
\leq \frac{M}{N_{\bar{k}}}.
\end{aligned}
\]
On the other hand, due to (\ref{boundedness of minor control}) and (\ref{boundedness of li}), we get
\[
\begin{aligned}
&\frac{M}{N_{\bar{k}}^2}\mathbb{E}\int_0^t|u_i(s)|^2ds
+\frac{M}{N_{\bar{k}}^2}\mathbb{E}\sum_{j=1}^{N}\int_0^t\left|F_2(l^{(N)}(s)-\Phi(s))\!+\!F_2\Phi(s)\!+\!
Cl_j(s)\!+\!Hl_0(s)\!+\!\sigma(s)\right|^2ds\\
\leq &\frac{M}{N_{\bar{k}}}\mathbb{E}\int_0^t|l^{(N)}(s)-\Phi(s)|^2ds\!+\!\frac{M}{N_{\bar{k}}}.
\end{aligned}
\]
Moreover, since  $(\bar{p}_i(s),\bar{q}_i(s))$, $i\in\mathcal{I}_{\bar{k}}$, are identically distributed under $\mathbb{E}(\cdot|\mathcal{F}_s^{W_0})$, we have
$
\frac{M}{N_{\bar{k}}^2}\mathbb{E}\sum\limits_{j\in\mathcal{I}_{\bar{k}},j\neq i}\int_0^t|\varphi_{\bar{k}}(\bar{p}_j(s),\bar{q}_j(s))|^2ds\leq
\frac{M}{N_{\bar{k}}}.
$
Therefore, from above estimates, we get from (\ref{minor ee11}) that, for any $t\in[0,T]$,
\[
\mathbb{E}\sup_{0\leq s\leq T}|\Xi(s)|^2\leq M\mathbb{E}\int_0^t|\Xi(s)|^2+|l^{(N)}(s)-\Phi(s)|^2\Big]ds+\frac{M}{N_{\bar{k}}},
\]
which yields, by using Gronwall's inequality, that
\begin{equation}\label{minor k type sum}
\begin{aligned}
\mathbb{E}\sup_{0\leq s\leq t}|l^{({\bar{k}})}(s)-m_{\bar{k}}(s)|^2
\leq  M\mathbb{E}\int_0^t\Big[|l^{(N)}(s)-\Phi(s)|^2\Big]ds
+\frac{M}{N_{\bar{k}}}.
 \end{aligned}
\end{equation}
Consequently, noticing (\ref{minor k other type sum}) and  (\ref{minor k type sum}), we have for each $1\leq k\leq K$,
\begin{equation}\label{minor each k type sum}
\begin{aligned}
\mathbb{E}\sup_{0\leq s\leq t}|l^{(k)}(s)-m_{k}(s)|^2
\leq  M\mathbb{E}\int_0^t\Big[|l^{(N)}(s)-\Phi(s)|^2\Big]ds
+\frac{M}{N_{k}}.
 \end{aligned}
\end{equation}
Since
\[
\begin{aligned}
l^{(N)}(s)-\Phi(s)=&\sum_{k=1}^{K}\left[\pi_k^{(N)}\frac{1}{N_k}\sum_{i\in\mathcal{I}_k}l_i(s)
-\pi_k\mathbb{E}(\alpha_k(s)|\mathcal{F}_s^{W_0})\right]\\
=&\sum_{k=1}^{K}\left[\pi_k^{(N)}\left(\frac{1}{N_k}\sum_{i\in\mathcal{I}_k}l_i(s)
-\mathbb{E}(\alpha_k(s)|\mathcal{F}_s^{W_0})\right)+
\left(\pi_k^{(N)}-\pi_k\right)\mathbb{E}(\alpha_k(s)|\mathcal{F}_s^{W_0})\right]\\
=&\sum_{k=1}^{K}\pi_k^{(N)}(l^{(k)}(s)-m_{k}(s))+\sum_{k=1}^{K}
\left(\pi_k^{(N)}-\pi_k\right)\mathbb{E}(\alpha_k(s)|\mathcal{F}_s^{W_0}),
\end{aligned}
\]
then by (\ref{bounded}),  (\ref{minor each k type sum})  and $\pi_k^{(N)}=\frac{N_k}{N}\leq 1$, we have for any $t\in[0,T]$,
\[
\begin{aligned}
\mathbb{E}\sup_{0\leq s\leq t}|l^{(N)}(s)-\Phi(s)|^2
\leq & \mathbb{E}\sum_{k=1}^{K}\sup_{0\leq s\leq t}\pi_k^{(N)}|l^{(k)}(s)-m_{k}(s)|^2+M\varepsilon_{N}^2\\
\leq & M\mathbb{E}\int_0^t\Big[|l^{(N)}(s)-\Phi(s)|^2\Big]ds+\frac{M}{N}+M\varepsilon_{N}^2.
\end{aligned}
\]
Finally, by using Gronwall's inequality, we complete the proof.
\end{proof}
\medskip

Now, we introduce the following system of decentralized limiting state  with the perturbation strategy of minor agent $\mathcal{A}_i$: for $1\leq j\leq N$, $j\neq i$,
  \begin{equation}\label{perturbed minor limit}
\left\{
\begin{aligned}
d\bar{l}_{0}&=\left[A_0\bar{l}_{0}\!+\!B_0\varphi_{0}(\bar{p}_0,\bar{q}_0)\!+\!F_0^1\Phi\!+\!b_0\right]dt
        \!+\!\left[C_0\bar{l}_{0}\!+\!D_0\varphi_{0}(\bar{p}_0,\bar{q}_0)
      +F_0^2\Phi\!+\!\sigma_0\right]dW_0(t)\\
d\bar{l}_{i}&=\left[A_{\theta_i}\bar{l}_{i}\!+\!Bu_i\!+\!F_1\Phi
      \!+\!b_0\right]dt\!+\!\left[C\bar{l}_{i}\!+\!D_{\theta_i}u_i
      +F_2\Phi\!+\!H\bar{l}_0\!+\!\sigma_0\right]dW_i(t),\\
d\bar{l}_{j}&=\left[A_{\theta_j}\bar{l}_{j}\!+\!B\varphi_{\theta_j}(\bar{p}_j,\bar{q}_j)
      \!+\!F_1\Phi
      \!+\!b_0\right]dt\!+\!\left[C\bar{l}_{j}\!+\!D_{\theta_j}\varphi_{\theta_j}(\bar{p}_j,\bar{q}_j)
      +F_2\Phi\!+\!H\bar{l}_0\!+\!\sigma_0\right]dW_i(t),\\
      \bar{l}_{0}(0)&=x_0, \quad \bar{l}_{i}(0)=\bar{l}_{j}(0)=x.
    \end{aligned}
    \right.
\end{equation}
\begin{lemma}\label{perturbed minor lemma 3}
Under the assumptions of \emph{\textbf{(A1)-(A4)}}, we have
\begin{align}\label{perturbed minor lemma 3 estimate}
\mathbb{E}\sup_{0\leq t\leq T}\left(\Big|l_{0}(t)-\bar{l}_{0}(t)\Big|^2+\Big|l_{i}(t)-\bar{l}_{i}(t)\Big|^2
\right)=O\Big(\frac{1}{N}+\varepsilon_N^2\Big).
\end{align}
\end{lemma}
\begin{proof}
From both the first equation of (\ref{perturbed minor}) and (\ref{perturbed minor limit}), we obtain
\[
\left\{
\begin{aligned}
d(l_0-\bar{l}_0)&=\left[A_0(l_0-\bar{l}_0)+F_0^1(l^{(N)}-\Phi)\right]dt
+\left[C(l_0-\bar{l}_0)+F_0^2(l^{(N)}-\Phi)\right]dW_0(t),\\
l_0(0)-\bar{l}_0(0)&=0.
\end{aligned}
\right.
\]
With the help of classical estimates of SDE and Lemma \ref{perturbed minor lemma 1}, we have
\begin{align}\label{perturbed minor lemma 3 estimate major}
\mathbb{E}\sup_{0\leq t\leq T}\Big|l_{0}(t)-\bar{l}_{0}(t)\Big|^2=O\Big(\frac{1}{N}+\varepsilon_N^2\Big).
\end{align}
Now, from both the second equation of (\ref{perturbed minor}) and (\ref{perturbed minor limit}), we obtain
\[
\begin{aligned}
d(l_i-\bar{l}_i)&=\left[A_{\theta_i}(l_i-\bar{l}_i)+F_1(l^{(N)}-\Phi)\right]dt\\
&\qquad+\left[C(l_i-\bar{l}_i)+F_2(l^{(N)}-\Phi)+H(l_0-\bar{l}_0)\right]dW_0(t),\qquad
l_i(0)-\bar{l}_i(0)=0.
\end{aligned}
\]
From the classical estimates of SDE, Lemma \ref{perturbed minor lemma 1} and (\ref{perturbed minor lemma 3 estimate major}), it is easy to obtain (\ref{perturbed minor lemma 3 estimate}).
\end{proof}
\medskip

\begin{lemma}\label{perturbed minor lemma 4}
Under the assumptions of \emph{\textbf{(A1)-(A4)}}, for each $1\leq i\leq N$, for the minor agent $\mathcal{A}_i$'s perturbation control $u_{i}$, we have
\[
\Big|\mathcal{J}_i({u}_i, \bar{u}_{-i})-J_0({u}_i)\Big|=O\Big(\frac{1}{\sqrt{N}}+\varepsilon_N\Big).
\]
\end{lemma}
\begin{proof}
Recall (\ref{minor agent cost functional}), (\ref{minor agent cost functional limit}) and (\ref{sum of mk}), we have
\begin{equation}\label{minor agent ee12}
\begin{aligned}
&\mathcal{J}_i({u}_i, \bar{u}_{-i})-J_i({u}_i)\\
=&\frac{1}{2}\mathbb{E}\bigg [ \int_0^T\Big( \Big <Q\big(l_i-\rho l^{(N)}\big)-(1-\rho)l_0,l_i-\rho l^{(N)}-(1-\rho)l_0\Big>\\
&\qquad\qquad-\Big <Q\big(\bar{l}_i
-\rho \Phi-(1-\rho)\bar{l}_0\big),\bar{l}_i-\rho \Phi-(1-\rho)\bar{l}_0\Big>\Big)dt\\
&+\Big<G\big(l_i(T)-\rho l^{(N)}(T)-(1-\rho)l_0(T)\big),l_i(T)-\rho l^{(N)}(T)-(1-\rho)l_0(T)\Big>\\
&\qquad\qquad-\Big<G\big(\bar{l}_i(T)
-\rho \Phi(T)-(1-\rho)\bar{l}_0(T)\big),\bar{l}_i(T)-\rho \Phi(T)-(1-\rho)\bar{l}_0(T)\Big>\bigg].
\end{aligned}
\end{equation}
Similar to the proof of Lemma \ref{first lemma for cost}, by using Lemma \ref{perturbed minor lemma 1}, \ref{perturbed minor lemma 3} and $\mathbb{E}\left(\left|\bar{l}_0(t)\right|^2+\left|\bar{l}_i(t)\right|^2+\left|\Phi(t)\right|^2\right)\leq M$, we have
\[
\begin{aligned}
&\Bigg|\mathbb{E} \int_0^T\Big( \Big <Q\big(l_i-\rho l^{(N)}-(1-\rho)l_0\big),l_i-\rho l^{(N)}-(1-\rho)l_0\Big>\\
&\qquad\qquad-\Big <Q\big(\bar{l}_i
-\rho \Phi\big)-(1-\rho)\bar{l}_0,\bar{l}_i-\rho \Phi-(1-\rho)\bar{l}_0\Big>\Big)dt\Bigg|\\
\leq &M\int_0^T\mathbb{E}\left|l_i-\bar{l}_i\right|^2dt+M\int_0^T\mathbb{E}\left|l_0-\bar{l}_0\right|^2dt
+M\int_0^T\mathbb{E}\left|l^{(N)}-\Phi\right|^2dt\\
&\qquad+M\int_0^T\left(\mathbb{E}\left|l_i-\bar{l}_i\right|^2
+\mathbb{E}\left|l_0-\bar{l}_0\right|^2
+\mathbb{E}\left|l^{(N)}-\Phi\right|^2\right)^{\frac{1}{2}}dt
\!=\!O\Big(\frac{1}{\sqrt{N}}\!+\!\varepsilon_N\Big)
\end{aligned}
\]
and
\[
\begin{aligned}
&\Bigg|\mathbb{E}\bigg[\Big<G\big(l_i(T)-\rho l^{(N)}(T)-(1-\rho)l_0(T)\big),l_i(T)-\rho l^{(N)}(T)-(1-\rho)l_0(T)\Big>\\
&\qquad\qquad-\Big<G\big(\bar{l}_i(T)
-\rho \Phi(T)-(1-\rho)\bar{l}_0(T)\big),\bar{l}_i(T)-\rho \Phi(T)-(1-\rho)\bar{l}_0(T)\Big>\bigg]\Bigg|
\!=\!O\Big(\frac{1}{\sqrt{N}}\!+\!\varepsilon_N\Big).
\end{aligned}
\]
The proof is completed by noticing (\ref{minor agent ee12}).
\end{proof}
\begin{theorem}\label{minor Nash equilibrium theorem}
Under the assumptions of \emph{\textbf{(A1)-(A4)}}, $(\bar{u}_0,\bar{u}_1,\cdots,\bar{u}_N)$ is an $\varepsilon$-Nash equilibrium of \emph{\textbf{Problem (CC)}} for minor agents.
\end{theorem}
\begin{proof}
For each $1\leq i\leq N$, combining Lemma \ref{first lemma for cost} and Lemma \ref{perturbed minor lemma 4}, we have
\[
\begin{aligned}
\mathcal{J}_i(\bar{u}_i, \bar{u}_{-i})
\leq J_i(\bar{u}_i)+O\Big(\frac{1}{\sqrt{N}}\!+\!\varepsilon_N\Big)
\leq J_i({u}_i)+O\Big(\frac{1}{\sqrt{N}}\!+\!\varepsilon_N\Big)
\leq \mathcal{J}_i({u}_i, \bar{u}_{-i})+O\Big(\frac{1}{\sqrt{N}}\!+\!\varepsilon_N\Big),
\end{aligned}
\]
where the second inequality comes from the fact that $J_i(\bar{u}_i)=\inf_{u_i\in\mathcal{U}_{ad}^i}J_i(u_i)$.
Consequently, Theorem \ref{minor Nash equilibrium theorem} holds with $\varepsilon=O\Big(\frac{1}{\sqrt{N}}+\varepsilon_N\Big)$.
\end{proof}
\medskip

By combining Theorems \ref{major Nash equilibrium theorem}, \ref{minor Nash equilibrium theorem}, we obtain the following main result of this paper:
\begin{theorem}\label{Nash equilibrium theorem}
Under the assumptions of \emph{\textbf{(A1)-(A4)}},
$(\bar{u}_0,\bar{u}_1,\cdots,\bar{u}_N)$ is an $\varepsilon$-Nash equilibrium of \emph{\textbf{Problem (CC)}}, where
$\bar{u}_0=\varphi_0(\bar{p}_0,\bar{q}_0)$, $\bar{u}_i=\varphi_{\theta_i}(\bar{p}_i,\bar{q}_i)$, $1\leq i\leq N$ for
\[
\varphi_0(p,q):=
\mathbf{P}_{\Gamma_0}\big[R_0^{-1}\big(B^{\prime}_0p+D^{\prime}_0q\big)\big], \quad
\varphi_{\theta_i}(p,q):=\mathbf{P}_{\Gamma_{\theta_i}}\big[R_{\theta_i}^{-1}\big(B^{\prime}p
\!+\!D_{\theta_i}^{\prime}q\big)\big].
\]
\end{theorem}

\section{Appendix}
We give this appendix to prove Theorem \ref{wellposedness MF-FBSDE PT}. The fully-coupled structure of MF-FBSDE (\ref{General MF-FBSDE}) arises difficulties for establishing its wellposedness. Motivated by Pardoux and Tang \cite{PT-1999} Theorem 3.1, we can  establish the
wellposedness of MF-FBSDE (\ref{General MF-FBSDE}) for arbitrary time duration when it is weakly coupled.

Let us first note that for a given $(Y(\cdot),Z(\cdot))\in L^2_{\mathcal F}(0,T;\mathbb{R}^m)\times L^2_{\mathcal F}(0,T;\mathbb{R}^{m\times (d+1)})$, the forward equation in the MF-FBSDE (\ref{General MF-FBSDE}) has a unique solution $X(\cdot)\in L^2_{\mathcal F}(0,T;\mathbb{R}^n)$, thus we introduce a map $\mathcal{M}_1:L^2_{\mathcal F}(0,T;\mathbb{R}^m)\times L^2_{\mathcal F}(0,T;\mathbb{R}^{m\times (d+1)})\rightarrow L^2_{\mathcal F}(0,T;\mathbb{R}^n)$, through
\begin{equation}\label{Forward General MF-FBSDE}
\begin{array}
[c]{rl}%
X(t)=x+&\displaystyle{\int_0^t} b(s,X(s),\mathbb{E}[X(s)|\mathcal{F}^{W_0}_s],Y(s),Z(s)])ds\medskip\\
+&\displaystyle{\int_0^t}\sigma(s,X(s),\mathbb{E}[X(s)|\mathcal{F}^{W_0}_s],Y(s),Z(s))dW(s).
\end{array}
\end{equation}We mention that the wellposedness of (\ref{Forward General MF-FBSDE}) can be established by using the contraction mapping method under assumption $(H_1),(H_2)$, although it has the term $\mathbb{E}[X_{s}|\mathcal{F}^{W_0}_s]$. We omit the proof. Moreover, with the help of BDG inequality, it follows that $\mathbb{E}\sup_{t\in[0,T]}|X(t)|^2<\infty$.
\begin{lemma} Let $X_i$ be the solution of (\ref{Forward General MF-FBSDE}) corresponding  to $(Y_i(\cdot),Z_i(\cdot))\in  L^2_{\mathcal F}(0,T;\mathbb{R}^m)\times L^2_{\mathcal F}(0,T;\mathbb{R}^{m\times (d+1)})$, $i=1,2$. Then for all $\lambda\in\mathbb{R}$, $K_1$, $K_2>0$, we have
\begin{equation}\label{estimate 1 mapping M1}
\begin{array}
[c]{rl}%
&e^{-\lambda t}\mathbb{E}|X_1(t)-X_2(t)|^2+\overline{\lambda}_1\displaystyle{\int_0^t}e^{-\lambda s}\mathbb{E}|X_1(s)-X_2(s)|^2ds\medskip\\
\leq & (k_2K_1+k_9^2)\displaystyle{\int_0^t}e^{-\lambda s}\mathbb{E}|Y_1(s)-Y_2(s)|^2ds+(k_3K_2+k_{10}^2)\displaystyle{\int_0^t}e^{-\lambda s}\mathbb{E}|Z_1(s)-Z_2(s)|^2ds,
\end{array}
\end{equation}where $\overline{\lambda}_1:=\lambda-2\lambda_1-k_2K_1^{-1}-k_3K_2^{-1}-2k_1-k_7^2-k_8^2$. Moreover,
\begin{equation}\label{estimate 2 mapping M1}
\begin{array}
[c]{ll}%
e^{-\lambda t}\mathbb{E}|X_1(t)-X_2(t)|^2\leq &(k_2K_1+k_9^2)\displaystyle{\int_0^t}e^{-\overline{\lambda}_1(t-s)}e^{-\lambda s}
\mathbb{E}|Y_1(s)-Y_2(s)|^2ds
\medskip\\
&+(k_3K_2+k_{10}^2)\displaystyle{\int_0^t}e^{-\overline{\lambda}_1(t-s)}e^{-\lambda s}\mathbb{E}|Z_1(s)-Z_2(s)|^2ds.
\end{array}
\end{equation}
\end{lemma}
\begin{proof}We denote $\overline{X}:=X_1-X_2$, $\overline{Y}:=Y_1-Y_2$, $\overline{Z}:=Z_1-Z_2$, $\overline{b}:=b(X_1,\mathbb{E}[X_1|\mathcal{F}^{W_0}],Y_1,Z_1)-b(X_2,\mathbb{E}[X_2|\mathcal{F}^{W_0}],Y_2,Z_2)$, $\overline{\sigma}:=\sigma(X_1,\mathbb{E}[X_1|\mathcal{F}^{W_0}],Y_1,Z_1)
-\sigma(X_2,\mathbb{E}[X_2|\mathcal{F}^{W_0}],Y_2,Z_2)$.
Applying It\^{o}'s formula to $e^{-\lambda t}\mathbb{E}|X_1(t)-X_2(t)|^2$ and taking expectation, we obtain
\begin{equation}\label{equation 1 general MF-FBSDE}
e^{-\lambda t}\mathbb{E}|\overline{X}(t)|^2= -\lambda\displaystyle{\int_0^t}e^{-\lambda s}\mathbb{E}|\overline{X}(s)|^2ds+2\mathbb{E}\displaystyle{\int_0^t}e^{-\lambda s}\langle\overline{X}(s),\overline{b}(s)\rangle ds+\mathbb{E}\displaystyle{\int_0^t}e^{-\lambda s}|\overline{\sigma}(s)|^2 ds.
\end{equation}
Noticing that
\[
\begin{array}
[c]{rl}%
&2\langle\overline{X}(s),\overline{b}(s)\rangle\medskip\\
=&2\langle\overline{X}(s),b(s,X_1(s),\mathbb{E}[X_1(s)|\mathcal{F}^{W_0}_s], Y_1(s),Z_1(s))-b(X_2(s),\mathbb{E}[X_1(s)|\mathcal{F}^{W_0}_s],Y_1(s),Z_1(s))\rangle\medskip\\
&+2\langle\overline{X}(s),b(s,X_2(s),\mathbb{E}[X_1(s)|\mathcal{F}^{W_0}_s],Y_1(s),Z_1(s))
-b(X_2(s),\mathbb{E}[X_2(s)|\mathcal{F}^{W_0}_s],Y_2(s),Z_2(s))\rangle\medskip\\
\leq& 2\lambda_1|\overline{X}(s)|^2+2|\overline{X}(s)|(k_1|\mathbb{E}[\overline{X}(s)|\mathcal{F}^{W_0}_s]|+
k_2|\overline{Y}(s)|+k_3|\overline{Z}(s)|)\medskip\\
\leq& (2\lambda_1+k_2K_1^{-1}+k_3K_2^{-1})|\overline{X}(s)|^2+2k_1\mathbb{E}[|\overline{X}(s)|^2|\mathcal{F}^{W_0}_s]|
+k_2K_1|\overline{Y}(s)|^2+k_3K_2|\overline{Z}(s)|^2,
\end{array}
\]
and
\[
\begin{array}
[c]{rl}%
|\overline{\sigma}(s)|^2
\leq& k_7^2|\overline{X}(s)|^2+k_8^2|\mathbb{E}[\overline{X}(s)|\mathcal{F}^{W_0}_s]|^2+
k_9^2|\overline{Y}(s)|^2+k_{10}^2|\overline{Z}(s)|^2\medskip\\
\leq&  k_7^2|\overline{X}(s)|^2+k_8^2\mathbb{E}[|\overline{X}(s)|^2|\mathcal{F}^{W_0}_s]+
k_9^2|\overline{Y}(s)|^2+k_{10}^2|\overline{Z}(s)|^2.
\end{array}
\]
Then, from (\ref{equation 1 general MF-FBSDE}) and $\mathbb{E}\left[\mathbb{E}[|\overline{X}(s)|^2|\mathcal{F}^{W_0}_s]\right]=\mathbb{E}|\overline{X}(s)|^2$, we can obtain (\ref{estimate 1 mapping M1}).

Now, we apply It\^{o}'s formula to $e^{-\overline{\lambda}_1(t-s)}e^{-\lambda s}\mathbb{E}|X_1(s)-X_2(s)|^2$ for $s\in[0,t]$ and taking expectation, it follows that
\begin{equation}\label{equation 2 general MF-FBSDE}
\begin{array}
[c]{rl}%
e^{-\lambda t}\mathbb{E}|\overline{X}(t)|^2= & -(\lambda-\overline{\lambda}_1)\displaystyle{\int_0^t}e^{-\overline{\lambda}_1(t-s)}e^{-\lambda s}\mathbb{E}|\overline{X}(s)|^2ds+2\mathbb{E}\displaystyle{\int_0^t}e^{-\overline{\lambda}_1(t-s)}e^{-\lambda s}\langle\overline{X}(s),\overline{b}(s)\rangle ds\medskip\\
&+\mathbb{E}\displaystyle{\int_0^t}e^{-\overline{\lambda}_1(t-s)}e^{-\lambda s}|\overline{\sigma}(s)|^2 ds.
\end{array}
\end{equation}From above estimates and (\ref{equation 2 general MF-FBSDE}), one can prove (\ref{estimate 2 mapping M1}).
\end{proof}
\begin{remark}By integrating both sides of (\ref{estimate 2 mapping M1}) on $[0,T]$ and using  $\frac{1-e^{-\overline{\lambda}_1(T-s)}}{\overline{\lambda}_1}\leq
\frac{1-e^{-\overline{\lambda}_1T}}{\overline{\lambda}_1}$, $\forall s\in[0,T]$, we have
\begin{equation}\label{estimate 3 mapping M1}
\|X_1-X_2\|_{\lambda}^2\leq \frac{1-e^{-\overline{\lambda}_1T}}{\overline{\lambda}_1}\left[(k_2K_1+k_9^2)\|Y_1-Y_2\|^2_{\lambda}
+(k_3K_2+k_{10}^2)\|Z_1-Z_2\|^2_{\lambda}\right].
\end{equation}
Let $t=T$ in (\ref{estimate 2 mapping M1}) and noticing that $e^{-\overline{\lambda}_1(T-s)}\leq 1\vee e^{-\overline{\lambda}_1T}$, $\forall s\in[0,T]$, thus
\begin{equation}\label{estimate 4 mapping M1}
e^{-\lambda T}\mathbb{E}|X_1(T)-X_2(T)|^2\leq \left[1\vee e^{-\overline{\lambda}_1T}\right]\left[(k_2K_1+k_9^2)\|Y_1-Y_2\|^2_{\lambda}
+(k_3K_2+k_{10}^2)\|Z_1-Z_2\|^2_{\lambda}\right].
\end{equation}In particular, if  $\overline{\lambda}_1>0$, we have
\begin{equation}\label{estimate 5 mapping M1}
e^{-\lambda T}\mathbb{E}|X_1(T)-X_2(T)|^2\leq (k_2K_1+k_9^2)\|Y_1-Y_2\|^2_{\lambda}
+(k_3K_2+k_{10}^2)\|Z_1-Z_2\|^2_{\lambda}.
\end{equation}
\end{remark}
\medskip

Similarly, for a given $X(\cdot)\in L^2_{\mathcal F}(0,T;\mathbb{R}^n)$, the backward equation in the MF-FBSDE (\ref{General MF-FBSDE}) has a unique solution $(Y(\cdot),Z(\cdot))\in L^2_{\mathcal F}(0,T;\mathbb{R}^m)\times L^2_{\mathcal F}(0,T;\mathbb{R}^{m\times (d+1)})$, thus we can introduce another map $\mathcal{M}_2:L^2_{\mathcal F}(0,T;\mathbb{R}^n) \rightarrow L^2_{\mathcal F}(0,T;\mathbb{R}^m)\times L^2_{\mathcal F}(0,T;\mathbb{R}^{m\times (d+1)})$, through
\begin{equation}\label{Backward General MF-FBSDE}
Y(t)=g(X(T),\mathbb{E}[X(T)|\mathcal{F}^{W_0}_T])+\int_t^T f(s,X(s),\mathbb{E}[X(s)|\mathcal{F}^{W_0}_s],Y(s),Z(s)])ds
-\int_t^TZ(s)dW(s).
\end{equation}The wellposedness of (\ref{Backward General MF-FBSDE}) under assumption $(H_1), (H_2)$ is referred to Darling and Pardoux \cite{DP-1997} Theorem 3.4 or Buckdahn and Nie \cite{BN-2016} Lemma 2.2.  Moreover, we have $\mathbb{E}\sup_{t\in[0,T]}|Y(t)|^2<\infty$.
\begin{lemma} Let $(Y_i(\cdot),Z_i(\cdot))$ be the solution of (\ref{Backward General MF-FBSDE}) corresponding  to $X_i\in  L^2_{\mathcal F}(0,T;\mathbb{R}^n)$, $i=1,2.$ Then for all $\lambda\in\mathbb{R}$, $K_3$, $K_4>0$, we have
\begin{equation}\label{estimate 1 mapping M2}
\begin{array}
[c]{rl}%
&e^{-\lambda t}\mathbb{E}|Y_1(t)-Y_2(t)|^2+\overline{\lambda}_2\displaystyle{\int_t^T}e^{-\lambda s}\mathbb{E}|Y_1(s)-Y_2(s)|^2ds\medskip\\
&\qquad\qquad\qquad+(1-k_6K_4)\displaystyle{\int_t^T}e^{-\lambda s}\mathbb{E}|Z_1(s)-Z_2(s)|^2ds\medskip\\
\leq &(k_{11}^2+k_{12}^2)e^{-\lambda T}\mathbb{E}|X_1(T)-X_2(T)|^2+(k_4+k_{5})K_3\displaystyle{\int_t^T}e^{-\lambda s}\mathbb{E}|X_1(s)-X_2(s)|^2ds,
\end{array}
\end{equation}where $\overline{\lambda}_2:=-\lambda-2\lambda_2-(k_4+k_5)K_3^{-1}-k_6K_4^{-1}$. Moreover,
\begin{equation}\label{estimate 2 mapping M2}
\begin{array}
[c]{ll}%
&e^{-\lambda t}\mathbb{E}|Y_1(t)-Y_2(t)|^2+(1-k_6K_4)\displaystyle{\int_t^T}e^{-\overline{\lambda}_2 (s-t)}e^{-\lambda s}\mathbb{E}|Z_1(s)-Z_2(s)|^2ds\medskip\\
\leq &(k_{11}^2+k_{12}^2)e^{-\overline{\lambda}_2 (T-t)}e^{-\lambda T}\mathbb{E}|X_1(T)-X_2(T)|^2\medskip\\
&+(k_4+k_{5})K_3\displaystyle{\int_t^T}e^{-\overline{\lambda}_2 (s-t)}e^{-\lambda s}\mathbb{E}|X_1(s)-X_2(s)|^2ds.
\end{array}
\end{equation}
\end{lemma}
\begin{proof}We denote $\overline{X}:=X_1-X_2$, $\overline{Y}:=Y_1-Y_2$, $\overline{Z}:=Z_1-Z_2$, $\overline{f}:=f(X_1,\mathbb{E}[X_1|\mathcal{F}^{W_0}],Y_1,Z_1)-f(X_2,\mathbb{E}[X_2|\mathcal{F}^{W_0}],Y_2,Z_2)$.
Applying It\^{o}'s formula to $e^{-\lambda t}\mathbb{E}|Y_1(t)-Y_2(t)|^2$ and taking expectation, we obtain
\begin{equation}\label{equation 3 general MF-FBSDE}
\begin{array}
[c]{rl}%
&e^{-\lambda t}\mathbb{E}|\overline{Y}(t)|^2-\lambda\displaystyle{\int_t^T}e^{-\lambda s}\mathbb{E}|\overline{Y}(s)|^2ds+\mathbb{E}\displaystyle{\int_t^T}e^{-\lambda s}|\overline{Z}(s)|^2 ds\medskip\\
=&e^{-\lambda T}\mathbb{E}|\overline{Y}(T)|^2 +2\mathbb{E}\displaystyle{\int_t^T}e^{-\lambda s}\langle\overline{Y}(s),\overline{f}(s)\rangle ds.
\end{array}
\end{equation}Noticing that
\[
\begin{array}
[c]{rl}%
&2\langle\overline{Y}(s),\overline{f}(s)\rangle\medskip\\
=&2\langle\overline{Y}(s),f(s,X_1(s),\mathbb{E}[X_1(s)|\mathcal{F}^{W_0}_s], Y_1(s),Z_1(s))-f(X_1(s),\mathbb{E}[X_1(s)|\mathcal{F}^{W_0}_s],Y_2(s),Z_1(s))\rangle\medskip\\
&+2\langle\overline{Y}(s),f(s,X_1(s),\mathbb{E}[X_1(s)|\mathcal{F}^{W_0}_s],Y_2(s),Z_1(s))
-f(X_2(s),\mathbb{E}[X_2(s)|\mathcal{F}^{W_0}_s],Y_2(s),Z_2(s))\rangle\medskip\\
\leq& 2\lambda_2|\overline{Y}(s)|^2+2|\overline{Y}(s)|(k_4|\overline{X}(s)|
+k_5|\mathbb{E}[\overline{X}(s)|\mathcal{F}^{W_0}_s]|+k_6|\overline{Z}(s)|)\medskip\\
\leq& (2\lambda_2+k_4K_3^{-1}+k_5K_3^{-1}+k_6K_4^{-1})|\overline{Y}(s)|^2
+k_4K_3|\overline{X}(s)|^2+k_5K_3\mathbb{E}[|\overline{X}(s)|^2|\mathcal{F}^{W_0}_s]|+k_6K_4|\overline{Z}(s)|^2,
\end{array}
\]
and
\[
\begin{array}
[c]{rl}%
|\overline{Y}(T)|^2
=& |g(X_1(T),\mathbb{E}[X_1(T)|\mathcal{F}^{W_0}_T])-g(X_2(T),\mathbb{E}[X_2(T)|\mathcal{F}^{W_0}_T])|^2\medskip\\
\leq&  k_{11}^2|\overline{X}(T)|^2+k_{12}^2\mathbb{E}[|\overline{X}(s)|^2|\mathcal{F}^{W_0}_s].
\end{array}
\]
Then, from (\ref{equation 3 general MF-FBSDE}) and $\mathbb{E}\left[\mathbb{E}[|\overline{X}(s)|^2|\mathcal{F}^{W_0}_s]\right]=\mathbb{E}|\overline{X}(s)|^2$, we can obtain (\ref{estimate 1 mapping M2}).

Now, we apply It\^{o}'s formula to $e^{-\overline{\lambda}_2(s-t)}e^{-\lambda s}\mathbb{E}|Y_1(s)-Y_2(s)|^2$ for $s\in[t,T]$ and taking expectation, it follows that
\begin{equation}\label{equation 4 general MF-FBSDE}
\begin{array}
[c]{rl}%
&e^{-\lambda t}\mathbb{E}|\overline{Y}(t)|^2-(\lambda+\overline{\lambda}_2)
\displaystyle{\int_t^T}e^{-\overline{\lambda}_2(s-t)}e^{-\lambda s}\mathbb{E}|\overline{Y}(s)|^2ds
+\mathbb{E}\displaystyle{\int_t^T}e^{-\overline{\lambda}_2(s-t)}e^{-\lambda s}|\overline{Z}(s)|^2 ds\medskip\\
= & e^{-\overline{\lambda}_2(T-t)}e^{-\lambda T}\mathbb{E}|\overline{Y}(s)|^2+2\mathbb{E}\displaystyle{\int_t^T}e^{-\overline{\lambda}_2(s-t)}e^{-\lambda s}\langle\overline{Y}(s),\overline{f}(s)\rangle ds.
\end{array}
\end{equation}From above estimates and (\ref{equation 4 general MF-FBSDE}), one can prove (\ref{estimate 2 mapping M2}).
\end{proof}
\begin{remark}Now we choose $K_4$ satisfying $0<K_4\leq k_6^{-1}$, then
by integrating  both sides of (\ref{estimate 2 mapping M2}) on $[0,T]$ and using  $\frac{1-e^{-\overline{\lambda}_2s}}{\overline{\lambda}_2}\leq
\frac{1-e^{-\overline{\lambda}_2T}}{\overline{\lambda}_2}$, $\forall s\in[0,T]$, we have
\begin{equation}\label{estimate 3 mapping M2}
\|Y_1-Y_2\|_{\lambda}^2\leq \frac{1-e^{-\overline{\lambda}_2T}}{\overline{\lambda}_2}\left[(k_{11}^2+k_{12}^2)e^{-\lambda T}\mathbb{E}|X_1(T)-X_2(T)|^2
+(k_4+k_{5})K_3\|X_1-X_2\|^2_{\lambda}\right].
\end{equation}
Let $t=0$ in (\ref{estimate 2 mapping M2}) and noticing that $1\wedge e^{-\overline{\lambda}_2T}\leq e^{-\overline{\lambda}_2s}\leq 1\vee e^{-\overline{\lambda}_2T}$, $\forall s\in[0,T]$, thus
\begin{equation}\label{estimate 4 mapping M2}
\|Z_1-Z_2\|^2_{\lambda}\leq
\frac{(k_{11}^2\!+\!k_{12}^2)e^{-\overline{\lambda}_2T}e^{-\lambda T}\mathbb{E}|X_1(T)\!-\!X_2(T)|^2
\!+\!(k_4\!+\!k_{5})K_3(1\vee e^{-\overline{\lambda}_2T})\|X_1\!-\!X_2\|^2_{\lambda}}{(1-k_6K_4)(1\wedge e^{-\overline{\lambda}_2T})}.
\end{equation}On the other hand, if  $\overline{\lambda}_2>0$, let $t=0$ in (\ref{estimate 1 mapping M2}), we have
\begin{equation}\label{estimate 5 mapping M2}
\|Z_1-Z_2\|^2_{\lambda}\leq
\frac{(k_{11}^2+k_{12}^2)e^{-\lambda T}\mathbb{E}|X_1(T)-X_2(T)|^2
+(k_4+k_{5})K_3\|X_1-X_2\|^2_{\lambda}}{1-k_6K_4}.
\end{equation}
\end{remark}
\medskip

Now, we present the proof of Theorem \ref{wellposedness MF-FBSDE PT}.
\medskip

\begin{proof}{\textbf{of Theorem \ref{wellposedness MF-FBSDE PT}}} We define $\mathcal{M}:=\mathcal{M}_2\circ \mathcal{M}_1$, recalling that $\mathcal{M}_1$ is defined through (\ref{Forward General MF-FBSDE}) and
 $\mathcal{M}_2$ is defined through (\ref{Backward General MF-FBSDE}). Thus $\mathcal{M}$ maps $L^2_{\mathcal F}(0,T;\mathbb{R}^m)\times L^2_{\mathcal F}(0,T;\mathbb{R}^{m\times (d+1)})$ into itself. To prove the theorem, it is only need to show that $\mathcal{M}$
 is a contraction mapping for some equivalent norm $\|\cdot\|_{\lambda}$. In fact, for  $(Y_i,Z_i)\in L^2_{\mathcal F}(0,T;\mathbb{R}^m)\times L^2_{\mathcal F}(0,T;\mathbb{R}^{m\times (d+1)})$, let $X_i:=\mathcal{M}_1(Y_i,Z_i)$ and $(\overline{Y}_i,\overline{Z}_i):=\mathcal{M}((Y_i,Z_i))$,  from (\ref{estimate 3 mapping M1}), (\ref{estimate 4 mapping M1}), (\ref{estimate 3 mapping M2}) and (\ref{estimate 4 mapping M2}) based on Lemma 7.1, 7.2, we have
\[
\begin{array}
[c]{rl}%
&\|\overline{Y}_1-\overline{Y}_2\|_{\lambda}^2+\|\overline{Z}_1-\overline{Z}_2\|_{\lambda}^2\medskip\\
\leq &\left[\frac{1-e^{-\overline{\lambda}_2T}}{\overline{\lambda}_2}+\frac{
1\vee e^{-\overline{\lambda}_2T}}{(1-k_6K_4)(1\wedge e^{-\overline{\lambda}_2T})}\right]\medskip\\
&\times\left[(k_{11}^2+k_{12}^2)e^{-\lambda T}\mathbb{E}|X_1(T)-X_2(T)|^2
+(k_4+k_{5})K_3\|X_1-X_2\|^2_{\lambda}\right]\medskip\\
\leq &\left[\frac{1-e^{-\overline{\lambda}_2T}}{\overline{\lambda}_2}+\frac{
1\vee e^{-\overline{\lambda}_2T}}{(1-k_6K_4)(1\wedge e^{-\overline{\lambda}_2T})}\right]\medskip\\
&\times\left[(k_{11}^2+k_{12}^2)(1\vee e^{-\overline{\lambda}_1T})
+(k_4+k_{5})K_3\frac{1-e^{-\overline{\lambda}_1T}}{\overline{\lambda}_1}\right]\medskip\\
&\times\left[(k_2K_1+k_9^2)\|Y_1-Y_2\|^2_{\lambda}
+(k_3K_2+k_{10}^2)\|Z_1-Z_2\|^2_{\lambda}\right].
\end{array}
\]
Recalling that $\overline{\lambda}_1:=\lambda-2\lambda_1-k_2K_1^{-1}-k_3K_2^{-1}-2k_1-k_7^2-k_8^2$ and
$\overline{\lambda}_2:=-\lambda-2\lambda_2-(k_4+k_5)K_3^{-1}-k_6K_4^{-1}$. Then
by choosing suitable $\lambda$ (e.g. we can easily choose $\lambda$ big enough such that $\overline{\lambda}_1\ge 1$ and  $\overline{\lambda}_2\leq 0$), the first assertion is immediate.

Now let us prove the second assertion. Since $2(\lambda_1+\lambda_2)<-2k_1-k_6^2-k_7^2-k_8^2$, we can choose $\lambda\in \mathbb{R}$, $0<K_4\leq k_6^{-1}$ and sufficient large $K_1,K_2,K_3$ such that
\[
\overline{\lambda}_1>0, \qquad \overline{\lambda}_2>0, \qquad 1-K_4k_6>0.
\]
Then from (\ref{estimate 3 mapping M1}), (\ref{estimate 5 mapping M1}), (\ref{estimate 3 mapping M2}) and (\ref{estimate 5 mapping M2}),we have
\[
\begin{array}
[c]{rl}%
&\|\overline{Y}_1-\overline{Y}_2\|_{\lambda}^2+\|\overline{Z}_1-\overline{Z}_2\|_{\lambda}^2\medskip\\
\leq &\left[\frac{1}{\overline{\lambda}_2}+\frac{
1}{1-k_6K_4}\right]\times\left[(k_{11}^2+k_{12}^2)e^{-\lambda T}\mathbb{E}|X_1(T)-X_2(T)|^2
+(k_4+k_{5})K_3\|X_1-X_2\|^2_{\lambda}\right]\medskip\\
\leq &\left[\frac{1}{\overline{\lambda}_2}+\frac{1}{1-k_6K_4}\right]\times\left[k_{11}^2+k_{12}^2
+(k_4+k_{5})K_3\frac{1}{\overline{\lambda}_1}\right]\medskip\\
&\times\left[(k_2K_1+k_9^2)\|Y_1-Y_2\|^2_{\lambda}
+(k_3K_2+k_{10}^2)\|Z_1-Z_2\|^2_{\lambda}\right].
\end{array}
\]
This completes the second assertion.
\end{proof}


\begin{thebibliography}{99}
\bibitem{A2015} S. Ahuja. Wellposedness of mean field games with common noise under a weak monotonicity conditions. \textit{SIAM J. Control Optim.}, \textbf{54}, 30-48 (2016).




\bibitem{B2012} M. Bardi. Explicit solutions of some linear-quadratic mean field games. \textit{Netw. Heterog. Media.} \textbf{7}, 243-261 (2012).

\bibitem{BP2014} M. Bardi, F. Priuli. Linear-quadratic N-person and mean-field games with ergodic cost. \textit{SIAM J. Control Optim.} \textbf{52}, 3022-3052 (2014).

\bibitem{BCY2016} A. Bensoussan, M. Chau, S. Yam. Mean Field Games with a Dominating Player. \textit{Appl. Math. Optim.}\textbf{74}, 91-128 (2016).

\bibitem{BFY} A. Bensoussan, J. Frehse, S. Yam. Mean Field Games and Mean Field Type Control Theory. Springer, New York, 2013.

\bibitem{BWYY2014}
A. Bensoussan, K. Wong, and S. Yam, S. Yung. Time-Consistent Portfolio Selection under Short-Selling Prohibition: From
Discrete to Continuous Setting. \textit{SIAM J. Financial Math.} \textbf{5}, 153-190 (2014).

\bibitem{BSY2013} A. Bensoussan, K. Sung, S. Yam. Linear-quadratic time-inconsistentmean field games. \textit{Dyn.
Games Appl.} \textbf{3}, 537-552 (2013).


\bibitem{BSYY}
A. Bensoussan, K. Sung, and S. Yam, S. Yung. Linear quadratic mean field games. \textit{J. Optim. Theory Appl.} \textbf{169}, 496-529 (2016).

\bibitem{B1972}
R. Brammer. Controllability in linear automous systems with positive controllers. \textit{SIAM J. Control Optim.}, \textbf{10}, 339-353 (1972).

\bibitem{Brezis 2010} H. Brezis. Functional Analysis, Sobolev Spaces and Partial Differential Equations. Springer, New York Dordrecht Heidelberg London, 2010.


\bibitem{BLP2014}
R. Buckdhan, J. Li, S. Peng. Nonlinear stochastic differential games invovling a major player and a large number of collectively acting minor agents. \textit{SIAM J. Control Optim.}, \textbf{52}, 451-492 (2014).

\bibitem{BN-2016} R. Buckdahn and T. Nie. Generalized Hamilton-Jacobi-Bellman equations with Dirichlet boundary condition and stochastic exit time optimal control problem, \textit{SIAM. J. Control. Optim.} 54, 602-631 (2016).

\bibitem{CD}
R. Carmona and F. Delarue. Probabilistic analysis of mean-field games.
\textit{SIAM J. Control Optim.}, \textbf{51}, 2705-2734 (2013).



\bibitem{CZ2016}   R. Carmona, X. Zhu. A probabilistic approach to mean field games with major and minor players. \textit{Ann. Appl. Probab.}, \textbf{26(3)}, 1535-1580 (2016).

\bibitem{CZ2004}   X. Chen, X. Zhou. Stochastic linear quadratic control with conic control constraints on an infinite time horizon. \textit{SIAM J. Control Optim.}, \textbf{43}, 1120-1150 (2004).

\bibitem{DP-1997} R. Darling and E. Pardoux. Backward SDE with random terminal time and applications to semilinear elliptic PDE, \textit{Ann. Probab.} 25, 1135-1159 (1997).


\bibitem{ET2015} G. Espinosa, N. Touzi. Optimal investment under relative performance concerns. \textit{Math. Finance}, \textbf{25}, 221-257 (2015).


\bibitem{Graham 1981} A. Graham. Kronecker products and Matrix Calculus: with applications. Ellis Horwood, Chichester, UK, 1981.

\bibitem{G-L-L 2011}
O. Gueant, J. M. Lasry, P. L. Lions.
Mean field games and applications.
\textit{Paris-Princeton Lectures on Mathematical Finance 2010.
Lecture Notes in Mathematics}, \textbf{2003}, 205-266 (2011).


\bibitem{HIM2005}   Y. Hu, P. Imkeller, M. M\"{u}ller. Utility maximization in incomplete markets. \textit{Ann. Appl. Probab.}, \textbf{15}, 1691-1712 (2005).

\bibitem{HHL2016}   Y. Hu, J. Huang, X. Li. Linear Quadratic Mean Field Game with Control Input
Constraint. {\it  ESAIM Control Optim. Calc. Var.} In press.	arXiv:1610.05895 [math.OC] (2016).

\bibitem{HP1995}   Y. Hu, S. Peng. Solution of forward-backward stochastic differential equations. \textit{Probab. Theory Relat. Fields}, \textbf{103}, 273-283 (1995).

\bibitem{HZ2005}   Y. Hu, X. Zhou. Constrained stochastic LQ control with random coefficients and applications to portfolio selection. \textit{SIAM J. Control Optim.}, \textbf{44}, 444-466 (2005).

\bibitem{H 2010}
M. Huang.
Large population LQG games involving a major player: The Nash certainty equivalence principle.
\textit{SIAM J. Control Optim.}, \textbf{48}, 3318-3353 (2010).



\bibitem{hcm06} M. Huang, P. Caines and R. Malham\'{e}. Distributed multi-agent decision-making with partial observations: asymptotic Nash equilibria. \emph{Proceedings of the 17th International Symposium on Mathematical Theory of Networks and Systems}, Kyoto, Japan (2006).

\bibitem{HCM2007} M. Huang, P. Caines, R. Malham\'{e}. An invariance principle in large population stochastic dynamic games. \textit{J. Syst. Sci. Complex}. \textbf{20(2)}, 162¨C172 (2007)

\bibitem{H-C-M 2007}
M. Huang, P. Caines, and R. Malham\'{e}.
Large-population cost-coupled LQG problems with non-uniform agents: Individual-mass behavior and decentralized $\varepsilon$-Nash equilibria.
\textit{IEEE Trans. Automat. Control}, \textbf{52}, 1560-1571 (2007).


\bibitem{hmc06} M. Huang, R. Malham\'{e} and P. Caines. Large population stochastic dynamic games: closed-loop McKean-Vlasov systems and the Nash certainty equivalence principle. \textit{Commun. Inf. Syst.}, \textbf{6}, 221-251 (2006).

\bibitem{KMC2012} A. Kizilkale, S. Mannor, P. Caines. Larger scale real-time bidding in the smart grid: A mean field framework. \textit{Proceeding of the 51st IEEE CDC}, Mani, HI, 3680-3687 (2012).

\bibitem{LL-2006} J. Lasry and P. Lions. Jeux \`{a} champ moyen.I. Le cas stationnaire. \textit{Comptes Rendus Math\'{e}matique}, \textbf{343}, 619-625 (2006).

\bibitem{LL-20062} J. Lasry and P. Lions. Jeux \`{a} champ moyen.II. Horizon fini et contr\^{o}le optimal. \textit{Comptes Rendus Math\'{e}matique}, \textbf{343}, 679-684 (2006).

\bibitem{ll} J. M. Lasry and P. L. Lions. Mean field games. \textit{Japan J. Math.}, \textbf{2}, 229-260 (2007).



\bibitem{L-Z 2008}
T. Li and J. F. Zhang.
Asymptotically optimal decentralized control for large population stochastic multiagent systems.
\textit{IEEE Trans. Automat. Control}, \textbf{53}, 1643-1660 (2008).

\bibitem{LZL2002}
X. Li, X. Zhou, A. Lim. Dynamic mean-variance portfolio selection with no-shorting constraints.
\textit{SIAM J. Control Optim.}
\textbf{40}, 1540-1555 (2002).


\bibitem{my}J. Ma and J. Yong. Forward-backward stochastic differential equations and their applications. Springer-Verlag, Berlin Heidelberg, 1999.

\bibitem{MNZ-2005}
D. Majerek, W. Nowak, W. Zi\c{e}ba.
Conditional strong lae of large number.
\textit{Int. J. Pure Appl. Math.}, \textbf{20}, 143-156 (2005).

\bibitem{NC-2013}
M. Nourian, P. Caines, $\epsilon$-Nash mean field game theory for nonlinear stochastic dynamical systems with major and minor agents.
\textit{SIAM J. Control Optim.}
\textbf{51}, 3302-3331 (2013).

\bibitem{N-H 2012}
S. L. Nguyen, M. Huang.
Linear-quadratic-Gaussian mixed games with continuum-parameterized minor players.
\textit{SIAM J. Control Optim.}
\textbf{50}, 2907-2937 (2012).


\bibitem{PT-1999} E. Pardoux and S. Tang. Forward-backward stochastic differential equations and quasilinear parabolic PDEs, \textit{Probab. Theory Relat. Fields}. 114, 123-150 (1999).

\bibitem{PW1999}
S. Peng, Z. Wu.
Fully coupled forward-backward stochastic differential equations and applications to optimal control.
\textit{SIAM J. Control Optim.}
\textbf{37}, 825-843 (1999).


\bibitem{P2014} F. Priuli. Linear-quadratic N-person and mean-field games: infinite horizon games with discounted
cost and singular limits. \textit{Dyn. Games Appl.} \textbf{5}, 397-419 (2014).

\bibitem{T-Z-B 2014}
H. Tembine, Q. Zhu, T. Basar.
Risk-sensitive mean-field stochastic differential games.
\textit{IEEE Trans. Automat. Control}, \textbf{59}, 835-850 (2014).


\bibitem{W-Z 2012}
B. Wang, J. Zhang.
Mean field games for large-population multiagent systems with markov jump parameters.
\textit{SIAM J. Control Optim.}, \textbf{50}, 2308-2334 (2012).


\bibitem{yz} J. Yong, X. Zhou. Stochastic controls: Hamiltonian systems and HJB equations. Springer-Verlag, New York ,1999.

\bibitem{Y2002}J. Yong. A leader-follower stochastic linear quadratic differential games. \textit{SIAM. J. Control Optim.} \textbf{41}, 1015-1041 (2002).



\end{thebibliography}
\end{document}